\documentclass{amsart}

\usepackage{amsmath,bbold,mathrsfs,amscd,lineno,hyperref,amsthm}
\usepackage{mathtools,afterpage}
\usepackage{tikz}
\usepackage{mathrsfs}
\usepackage{yfonts,amsbsy} %mathabx

\usepackage{amsfonts,subcaption}
\usepackage{amssymb}
\usepackage{enumitem}
\usepackage{tikz-cd}

\usepackage{float}
\usepackage{tabulary}
\usepackage{booktabs}
\usepackage{mathtools} 
\usepackage[normalem]{ulem}

\usepackage{tipa}

\usepackage[utf8]{inputenc}

\usetikzlibrary{matrix, patterns,snakes,calc}
\usetikzlibrary{math,arrows,shapes}

\newcommand{\proj}{{\operatorname{proj}}}

\renewcommand{\sp}{{\ \ }}

\newcommand{\hide}[1]{}

\newcommand{\blambda}{{\boldsymbol\lambda}}

\newcommand{\bkappa}{{\boldsymbol\kappa}}

\everymath{\displaystyle}

\newcommand{\reminder}[1]{\textsl{$\lhd\rhd$#1$\lhd\rhd$}\marginpar{$\lhd\rhd\lhd\rhd\lhd\rhd$}}

\newcommand{\upbullet}[1]{\overset{\bullet}{ #1}{}}

\newcommand{\ext}{{\operatorname{ext}}}
\renewcommand{\div}{{\operatorname{div}}}

\newcommand{\Fam}{{\mathcal F}}
\newcommand{\Width}{{\mathcal W}}

\newcommand{\stab}{{\text{stab}}}
\newcommand{\bistab}{{\text{bistab}}}

\newcommand\LL{{\mathcal L}}

\makeatletter
\newcommand{\setword}[2]{%
  \phantomsection
  #1\def\@currentlabel{\unexpanded{#1}}\label{#2}%
}

\newcommand\TT{\mathcal{T}}

\newcommand\HH{\mathcal{H}}%{\boldsymbol{\Delta}}%

\newcommand{\bbv}{{\mathbf{v}}}

\newcommand{\bA}{{\mathbf{A}}}

\newcommand\RR{\mathcal{R}}

\newcommand\XX{\mathcal{X}}

\newcommand\FF{\mathcal{F}}

\newcommand\intr{\operatorname{int}}

\newcommand{\bUpsilon}{{\mathbf \Upsilon}}

\newcommand{\bM}{{\mathbf M}}

\newcommand{\wI}{{\widehat I}}
\newcommand{\wZ}{{\widehat Z}}

\newcommand{\bL}{{\mathbf L}}

\newcommand{\bT}{{\mathbf T}}

\newcommand{\bK}{{\mathbf K}}

\newcommand{\bbm}{{\mathbf m}}

\newcommand{\bbs}{{\mathbf s}}

\newcommand{\bbw}{{\mathbf w}}

\newcommand{\bbt}{{\mathbf t}}

\newcommand{\Level}{\operatorname{Level}}

\newcommand{\qq}{\mathfrak{q}}
\renewcommand{\tt}{\mathfrak{t}}

\newcommand{\length}{{\mathfrak l}}

\newcommand\inn{{\operatorname{inn}}}

\newcommand\DD{\mathcal{D}}

\newcommand\PP{\mathcal{P}}

\newcommand{\grnd}{{\operatorname{grnd}}}

\newcommand\new{{\operatorname{new}}}
\newcommand\New{{\operatorname{New}}}

\newcommand\ver{{\operatorname{ver}}}
\newcommand\per{{\operatorname{per}}}

\newtheorem{claim}{Claim}

\newcommand{\wC}{\widehat{\mathbb{C}}}
\newcommand{\C}{\mathbb{C}}
\newcommand{\Q}{\mathbb{Q}}
\newcommand{\R}{\mathbb{R}}
\newcommand{\N}{\mathbb{N}}

\DeclareFontFamily{U}{mathb}{\hyphenchar\font45}
\DeclareFontShape{U}{mathb}{m}{n}{ <5> <6> <7> <8> <9> <10> gen * mathb <10.95> mathb10 <12> <14.4> <17.28> <20.74> <24.88> mathb12 }{}
\DeclareSymbolFont{mathb}{U}{mathb}{m}{n}
\DeclareFontSubstitution{U}{mathb}{m}{n}
\DeclareMathSymbol{\selfmap}{3}{mathb}{"FD}

\newcommand{\Z}{\mathbb{Z}}

\newcommand{\Disk}{\mathbb{D}}

\newcommand{\Dbb}{{\mathfrak{D}}}

\newtheorem{thm}{Theorem}[section]

\newtheorem{cor}[thm]{Corollary}
\newtheorem{lem}[thm]{Lemma}

\newtheorem{prop}[thm]{Proposition}
\newtheorem{rem}[thm]{Remark}

\newtheorem{caliblem}[thm]{Calibration Lemma}
\newtheorem{ampthm}[thm]{Amplification Theorem}

\newtheorem{lairlmm}[thm]{Snake-Lair Lemma}

\newtheorem{mainthm:bounds}[thm]{Uniform Bounds Theorem}

\newtheorem{conj}[thm]{Conjecture}

\newtheorem{mainthm:sector_bounds}[thm]{Sector Bounds Theorem}
\newtheorem{mainthm:radial_str}[thm]{Star-Like Theorem}
\newtheorem{mainthm:psi_ql}[thm]{QL Siegel bounds Theorem}

\theoremstyle{remark}

\numberwithin{equation}{section}

\theoremstyle{definition}
\newtheorem{defn}[thm]{Definition}
\font\nt=cmr7

\def\be{\begin{equation}}
\newcommand{\filled}{{\mathcal {K} }}
\newcommand{\bnd}{{\mathrm{bnd}}}

\renewcommand{\sec}{\mathrm{sec}}
\usepackage{color,todonotes,comment}

\newcommand{\dist}{\operatorname{dist}}

\renewcommand{\mod}{\operatorname{mod}}

\newcommand{\orb}{\operatorname{orb}}

\newcommand{\Fjord}{{\mathfrak F}}

\newcommand{\CP}{{\operatorname{CP}}}

\newcommand{\CC}{{\mathcal C}}

\newcommand{\bdelta}{{\boldsymbol{ \delta}}}

\newcommand{\bchi}{{\boldsymbol{ \chi}}}

\newcommand{\RN}[1]{\MakeUppercase{\romannumeral#1}}%

\def\note#1
{\marginpar
{\nt $\leftarrow$
\par
\hfuzz=20pt \hbadness=9000 \hyphenpenalty=-100 \exhyphenpenalty=-100
\pretolerance=-1 \tolerance=9999 \doublehyphendemerits=-100000
\finalhyphendemerits=-100000 \baselineskip=6pt
#1}\hfuzz=1pt}


\DeclareFontFamily{U}{mathb}{\hyphenchar\font45}
\DeclareFontShape{U}{mathb}{m}{n}{ <5> <6> <7> <8> <9> <10> gen * mathb <10.95> mathb10 <12> <14.4> <17.28> <20.74> <24.88> mathb12 }{}
\DeclareSymbolFont{mathb}{U}{mathb}{m}{n}
\DeclareFontSubstitution{U}{mathb}{m}{n}
\DeclareMathSymbol{\righttoleftarrow}{3}{mathb}{"FD}

\begin{document}

\title[Uniform \emph{a priori} bounds for $\psi^\bullet$-ql Siegel maps]{Uniform \emph{a priori} bounds  \\ for neutral renormalization.\\ Variation~\RN{2}: $\psi^\bullet$-ql Siegel maps.}
\author{Dzmitry Dudko}
\author{Yusheng Luo}
\author{Mikhail Lyubich}

\date{\today}
\begin{abstract}
We extend uniform pseudo-Siegel bounds for neutral quadratic polynomials to $\psi^\bullet$-quadratic-like Siegel maps. In this form, the bounds are compatible with the $\psi$-quadratic-like renormalization theory and are easily transferable to various families of rational maps. 
The main theorem states that the degeneration of a Siegel disk is equidistributed among combinatorial intervals. This provides a precise description of how the $\psi^\bullet$-quadratic-like structure degenerates around the Siegel disk on all geometric scales except on the ``transitional scales'' between two specific combinatorial levels.
\end{abstract}
\maketitle

\setcounter{tocdepth}{1}

\tikzset{%
  block/.style    = {draw, thick, rectangle, minimum height = 3em,
    minimum width = 3em},
  sum/.style      = {draw, circle, node distance = 2cm}, % Adder
  input/.style    = {coordinate}, % Input
  output/.style   = {coordinate} % Output
}

\tableofcontents

\section{Introduction}\label{apendixp:C} Let  $f_\theta: z\mapsto e^{2\pi i \theta} z + z^2$ be a quadratic polynomial. We assume that its rotation number 
\[\theta\in \Theta_\bnd=\{\theta=[0;a_1,a_2,\dots] \mid a_i \le M_\theta\}\subsetneq \Theta\coloneqq  (\R\setminus \Q)/\Z\]
 is of bounded-type; i.e., $f_\theta$ has a quasi-conformal closed Siegel disk $\overline Z_\theta$ at the origin with critical point on the boundary of $\overline Z_\theta$. In~\cite{DL22}, it was established uniform pseudo-Siegel bounds controlling oscillations of $\partial Z_\theta$ in all scales. Consequently, pseudo-Siegel bounds endure under taking the closure \[\Theta_\bnd \leadsto \overline \Theta= \{\theta= [0;a_1,a_2,\dots] \mid a_i \in \N_{\ge 1}\cup \{\infty\}\}\supset \Theta\] and provide uniform \emph{a priori} bounds for all neutral quadratic polynomials.

In this paper, we extend pseudo-Siegel bounds to $\psi^\bullet$-quadratic-like maps  (``$\psi^\bullet$-ql maps'').  In such a form, pseudo-Siegel bounds can be easily transferred to various classes of rational maps using $\psi^\bullet$-ql renormalization; see~\cite{DL:HypComp} and Examples in \S\ref{ss:inflat}. 

Pseudo-Siegel bounds for quadratic polynomials are recalled in~\S\ref{ss:quadr_polynom}. We state the pseudo-Siegel bounds for quadratic-like (``ql'') maps in Theorems~\ref{thm:combinatorial main} and~\ref{thm:main:ql maps}, where the latter is a refinement of the former. The case of $\psi^\bullet$-ql maps is similar and is stated as Theorem~\ref{thm:main:psi*-ql maps}.

We remark that Equidistribution Theorem~\ref{thm:main:psi*-ql maps} has a stronger form than parallel results for quadratic-like renormalization~\cite{K,molecules,DL:Feigen}. Consequently, Theorem~\ref{thm:main:psi*-ql maps} provides a satisfactory control on how wide families cross (iterated preimages of) Siegel disks which is important in applications \cite{DL:HypComp};~see also a short discussion in~\S\ref{ss:Application}.

\subsection{Pseudo-Siegel bounds for quadratic polynomials} \label{ss:quadr_polynom} Following~\cite[Definition 1.4 and~\S11.0.2]{DL22},  a quadratic polynomial $f_\theta$ with $\theta\not\in \Q$ has \emph{pseudo-Siegel bounds}, if it has a sequence of nested disks 
\begin{multline}
\label{eq:dfn:wZ^m:intro}\qquad
\wZ^{-1}_{f_\theta}\supseteq \wZ^{0}_{f_\theta}\supseteq \wZ^1_{f_\theta}\supseteq \dots \supseteq H_{f_\theta}\ \ni 0, \\ f_\theta\mid \wZ^m_{f_\theta} \text{ is injective, }\sp\sp \bigcap_{m\ge -1} \wZ^m_{f_\theta} =H_{f_\theta},\qquad
\end{multline}
where $H_{f_\theta}$ is the \emph{Mother Hedgehog} (the filled postcritical set) of ${f_\theta}$  such that
\begin{itemize}
\item $\wZ^m_{f_\theta}$ is almost $f_\theta^{\qq_{m+1}}$-invariant pseudo-Siegel disk (see~\S\ref{sss:iota peripheral}), where $\qq_{m+1}$ is the first closest return~\eqref{eq:p_n q_n};
\item $\partial \wZ^m_{f_\theta}$ has a nest of tiling $\TT \left(\wZ^m_{f_\theta}\right)$ with essentially bounded geometry independent of $m$.
\end{itemize}
If the geometric bounds on  $\TT \left(\wZ^m_{f_\theta}\right)$ are independent of $\theta$, then $f_\theta$ has \emph{uniform pseudo-Siegel bounds}. Such uniform bounds imply that $\wZ^{-1}_{f_\theta}$ are uniformly qc while $\wZ^{m}_{f_\theta}$ become uniformly qc $\wZ^{-1}_{f_{n,\theta}}$ under appropriate sector sector renormalization $\RR^n_\sec\colon f_{\theta}\mapsto f_{n,\theta},$ where $n=n(m,\theta)\le m$, see~\cite{DL:sector bounds}.

A pseudo-Siegel disk $\wZ^m_{f_\theta}$ can be constructed out of $H_{f_\theta}$ by adding parabolic fjords bounded by $H_{f_\theta}$ and hyperbolic geodesics connecting appropriately specified points of $H_{f_\theta}$. In this case, $\wZ^m_{f_\theta}$ is called a \emph{geodesic} pseudo-Siegel disk. Note that the image of a geodesic pseudo-Siegel disk is not geodesic: images of fjords will be bounded by ``quasi-geodesics''. If $\theta\in\Theta_\bnd$, then $H_{f_\theta}=\overline Z_{f_\theta}$.

All neutral quadratic polynomials $f_\theta$ have uniform pseudo-Siegel bounds by the main result in \cite{DL22}. Various consequences, such as sectorial bounds, are stated in~\cite{DL:sector bounds}.

\subsection{Quadratic-like maps} 
\label{ss:into:ql maps}
Let us now discuss the generalization of pseudo-Siegel bounds for quadratic-like maps.
Let $f\colon X\to Y$ be a quadratic-like map, i.e., $X, Y$ are Jordan domains with $\overline{X} \subseteq Y$, and $f\colon X\to Y$ is a degree $2$ branched cover.
Let $\mathcal{K}(f)$ be the filled Julia set of $f$. 
We suppose that there exists a Siegel disk $Z_f \subseteq \mathcal{K}(f)$ at its $\alpha$-fixed point with bounded-type rotation number. 
Denote by \[K_f=\Width(f)\coloneqq \Width\big(Y\setminus \overline Z_f\big)\] the {\em degeneration} of $f$ around its Siegel disk. 
Here $\Width\big(Y\setminus \overline Z_f\big)$ is the extremal width of the vertical family, equivalently, the reciprocal of the modulus, of the annulus $Y\setminus \overline Z_f$.

Note that $f: \partial Z_f \longrightarrow \partial Z_f$ is topologically conjugate to the rigid rotation $R_{\theta}: S^1 \longrightarrow S^1$. Let $\theta = [0; a_1,a_2,...]$ be the continued fraction expansion. 
Note that $\theta$ is of {\em bounded-type} if $a_n$ are bounded; and is {\em eventually-golden-mean} if $a_n = 1$ for all sufficiently large $n$.
Consider the sequence of best approximations of $\theta$
\begin{equation}
\label{eq:p_n q_n}
\frac{\mathfrak{p}_n}{\mathfrak{q}_n} := \begin{cases} 
      [0;a_1,a_2,..., a_n] & \text{if } a_1 > 1 \\
      [0;a_1,a_2,..., a_{n+1}] & \text{if } a_1 = 1
   \end{cases}\; ,
\end{equation}
and set $\mathfrak{q}_0 :=1$. Then $f^{\mathfrak{q}_0}, f^{\mathfrak{q}_1},...$ is the sequence of combinatorially closest returns of $f|{\partial Z_f}$, and we denote
$$
\length_n:= \dist(f^{\mathfrak{q}_n}(x), x), \sp \sp x\in \partial Z_f
$$
where $\dist (x,y) = \dist_{\R/\Z}(h(x), h(y))$ is the combinatorial distance defined by means of the conjugacy $h: \partial Z_f \longrightarrow S^1$ to the rotation.
By convention, we set $\length_{-1} = 1$.

Let $I \subseteq \partial Z_f$ be an interval.
Let $\mathcal{W}_{\lambda}^+(I)$ be the extremal width of the family of curves $\gamma \in Y - Z_f$ connecting $I$ to $\partial Y \cup (\partial Z_f - \lambda I)$, where $\lambda > 1$ and $\lambda I$ is the enlargement of $I$ by attaching two intervals of (combinatorial) length $\frac{\lambda-1}{2} |I|$ on either side of $I$.
The {\em vertical degeneration} $\mathcal{W}^{+, \ver}(I)$ is the extremal width of the family of curves $\gamma \in Y - Z_f$ connecting $I$ to $\partial Y$, while the {\em peripheral $\lambda$-degeneration} $\mathcal{W}_{\lambda}^{+, \per}(I)$ is the extremal width of the family of curves $\gamma \subset Y - Z_f$ connecting $I$ to $\partial Z_f - \lambda I$.
Note that
$$
\mathcal{W}_{\lambda}^+(I) = \mathcal{W}^{+, \ver}(I) + \mathcal{W}_{\lambda}^{+, \per}(I)-O(1).
$$
These definitions also generalize to {\em pseudo-Siegel disks} (see \cite[\S 5]{DL22} and \S\ref{sss:intervalpsuedoSiegel}).
An interval $I \subseteq \partial Z_f$ is called level $m$ {\em combinatorial} if $|I| = |h(I)|_{\R/\Z} = \length_m$. Every such level $m$ combinatorial interval is of the form $I=[x, f^{\qq_{m+1}}(x)]$.

\begin{thm}\label{thm:combinatorial main}
    Let $f: X \longrightarrow Y, \ X\Subset Y$ be a quadratic-like map with a bounded-type Siegel disk $Z_f$ at its $\alpha$-fixed point.
    Let $K_f = W(f)$ be the degeneration of $f$ around its Siegel disk. Then for any level $m$ combinatorial interval $I$, we have
    $$
    \mathcal{W}_3^+(I) = O(\length_mK_f+1).
    $$

    Moreover, there is an absolute constant $\bK\gg 1$ such that if $\length_m K_f \ge \bK$, then 
    $$
    \mathcal{W}_3^+(I) \asymp  \length_m K_f.
    $$
    for all $f$ and $m$ as above.
\end{thm}

In the following, we formulate a refined theorem accounting for all intervals. The statement gives a precise description of how $Y\setminus \overline Z$ degenerates in all geometric scales except for intervals whose length is between $\length_{\bbm_f}$ and $\length_{\bbm_f+1}$, where $\bbm_f$ is the ``special transitional level'' defined as follows.

Let $\bK\gg 1$ be a sufficiently big threshold (satisfying, in particular, Theorem~\ref{thm:combinatorial main}).
We define the {\em special transition level} $\bbm_f$ for $f$ with respect to the threshold $\bK$ as follows. 
\begin{itemize}
    \item If $K_f\le \bK$, we set $\bbm_f\coloneqq -2$;
    \item Otherwise, we set $\bbm_f$ to be the level satisfying 
\[ \frac{1}{\length_{\bbm_f}} < \frac{K_f}{\bK} \le \frac{1}{\length_{\bbm_f\ +1}} \sp\sp\sp  \begin{array}{c}\text{ or, }\\ \text{equivalently,}\end{array}\sp\sp\sp \sp   \begin{array}{c}
\length_{\bbm_f} K_f>\bK, \text{ and}\vspace{0.2cm}\\
\length_{\bbm_f\ +1} K_f\le \bK.
\end{array} \]
\end{itemize}

We refer to $\ell_m K_f$ as the \emph{average vertical degeneration at level $m$.}  Similarly, if $\eta\in (0,1)$ is any number representing the combinatorial length, then the \emph{average weight of intervals $J$ with $|J|=\eta$} is $\eta K_f=|J|K_f$. Such an interval $J$ is \emph{grounded rel $\wZ^m$} if its endpoints are away from ``inner buffers''; see~\S\ref{sss:intervalpsuedoSiegel}. Geometrically, grounded intervals are ``visibly'' from the external boundary $\partial Y$. If $|J|\ge \length_m$, then any such $J$ is ``essentially'' grounded.

\begin{thm}[Equidistribution, see Theorems~\ref{thm:main:psi-ql maps} and~\ref{thm:main:psi-ql maps:extra}]\label{thm:main:ql maps} There exists an absolute constant $\bK\gg 1$ so that the following holds.

Let $f: X \longrightarrow Y,\ X\Subset Y$ be a quadratic-like map with an eventually-golden-mean Siegel disk $Z_f$ at its $\alpha$-fixed point.
Let $K_f = W(f)$ be the degeneration of $f$ around its Siegel disk, and let $\bbm_f$ be the special transition level with respect to the threshold $\bK$.
Then there is a nested sequence of geodesic pseudo-Siegel disks as in~\S\ref{sss:geodes:wZ} \[\wZ^m,\sp  m\ge -1,\hspace{1cm}\text{ where }\sp \wZ^{\bbm_f}= \wZ^{\bbm_f-1}=\dots =\wZ^{-1}\] such that for every grounded interval $J\subset \partial Z$ with $\length_{m+1}< |J|\le \length_{m}$, the following holds for the projection $J^m$ of $J$ to $\wZ^m$:
\begin{enumerate}[ label=(\Alph*)]
    \item \label{thm:apbs:A} if $m> \bbm_f$, then \[\Width_{\wZ^m}^{+,\ver} (J^m)=O(1) \sp\sp\text{ and }\sp\sp \Width_{3,\wZ^m}^{+,\per} (J^m)\asymp 1,\]
    \item\label{thm:apbs:B} if $m= \bbm_f$, then 
    \[\Width_{\wZ^m}^{+,\ver} (J^m)= O(\length_m K_f)  \sp\sp\text{ and }\sp\sp \Width_{3,\wZ^m}^{+,\per} (J^m)=O\big(\sqrt{\length_m K_f }\big),\]
  \item \label{thm:apbs:C} if $m< \bbm_f$, then 
    \[\Width_{\wZ^m}^{+,\ver} (J^m)\asymp |J|K_f  \sp\sp\text{ and }\sp\sp \Width_{3,\wZ^m}^{+,\per} (J^m)=O(1),\]
\end{enumerate}

Moreover, all pseudo-Siegel disks $\wZ^m$ are constructed with explicit combinatorial thresholds~\eqref{eq:dfn:bM}, and $\wZ^{-1}$ is $\mu(K_f)$-qc disk; i.e.~the dilatation of $\wZ^{-1}$ is bounded in terms of $K_f$.
\end{thm}

We remark that in Cases \ref{thm:apbs:C} and \ref{thm:apbs:B}, we have
$ |J|K_f, \ \length_m K_f \  \succeq\  1  $. We also remark that in all three cases, we have: \[\Width^+_3(I) =O(\length_m K_f+1)=O\left(\frac{K_f}{\qq_{m+1}}+1\right), \sp\sp\text{ where }|I|=\length_m.\]

In short, Case \ref{thm:apbs:A} says that on deep scales, the local geometry of $f$ is uniformly bounded, and the estimates (see~\S\ref{subsec:deeptransitional}) are almost equivalent to the case of quadratic polynomials. Case \ref{thm:apbs:C} says that on shallow scales, vertical degeneration prevails over peripheral and is uniformly distributed over all (combinatorial or not) intervals. Consequently, there can not be any regularizations on such scales: $\wZ^{-1}=\wZ^0=\dots =\wZ^{\bbm_f}$. The Transitional Case~\ref{thm:apbs:B} is uncertain: there might be intervals $J$ so that $\Width_{\wZ^m}^{+,\ver} (J^m)$ is big while $\Width_{\wZ^m}^{+,\per} (J^m)$ is bounded, and there might be intervals $J$ with the opposite behavior. The bound on $\Width_{3,\wZ^m}^{+,\per}$ stated in Case~\ref{thm:apbs:B} is a useful ingredient for applications (e.g., for rational maps). Other ingredients should come from ``global external input'', see~\cite{DL:HypComp} and~\S\ref{ss:comblocal}.

The combinatorial threshold~\eqref{eq:dfn:bM} for $\wZ^m$ in Theorem~\ref{thm:main:psi-ql maps} is universal on all levels except transitional. Such an explicit combinatorial threshold implies that $\partial \wZ^m$ possesses a nest of tilings $\TT(\partial \wZ^m)$ obtained by projecting the diffeo-tiling of $\partial Z$ (see \cite[\S2.1.6]{DL22}) onto $\partial \wZ^m$ so that $\TT(\partial \wZ^m)$ has essentially bounded geometry independent of $K_f$ for all $m>\bbm_f$.

\subsection{$\psi^\bullet$-ql maps}\label{ss:into:psi ql maps}
Given a rational map $g\colon \wC\selfmap$ with a periodic Siegel disk $\overline Z=g^p(\overline Z)$, one may want to construct an appropriate renormalization around $\overline Z$ to identify the dynamics of $f\coloneqq g^p$ around $\overline Z$ with the associated quadratic polynomial. There are two major challenges:
\begin{itemize}
  \item The formalism of polynomial-like maps is usually unavailable beyond polynomials or ``Sierpinski-type'' rational maps. For instance, matings of neutral quadratic polynomials do not admit quadratic-like structures. 
    \item Even when there is a quadratic-like restriction $f\colon X\to Y$ around $\overline Z$, there is no natural choice of the domain $Y$ so that $\Width(Y\setminus \overline Z)$ is optimal.
 \end{itemize}
To handle the second challenge, a $\psi$-ql renormalization was introduced in \cite{K} by J. Kahn. By extending or ``inflating'' $f\colon X\to Y$ along all possible paths outside of a forward-invariant set containing the postcritical set of $g$ and the filled Julia set of $f\colon X\to Y$, we obtain a $\psi$-ql map $F=(f,\iota)\colon U\rightrightarrows V$, where
\begin{itemize}
\item $f\colon U\to V$ is a branched covering extending $f\colon X\to Y$; and
\item  $\iota$ is an immersion extending the inclusion $X\subset Y$. 
\end{itemize}

To overcome the first challenge, we consider any proper covering $f\colon A\overset{2:1}{\longrightarrow} B$ around $\overline Z$ and relax the containment requirement ``$A\Subset B$'' into a homotopy equivalence between $\partial A$ and $\partial B$ rel an appropriate enlargement of the postcritical set, see Figure~\ref{Fig:UnwindingSiegelFull} and Definition~\ref{dfn:SiegelPreNorm}. By inflating $f\colon A\overset{2:1}{\longrightarrow} B$, see Proposition~\ref{prop:g:unwind}, we obtain a  $\psi^\bullet$-ql map $F\colon U \rightrightarrows V$ as in Definition~\ref{dfn:psi:Siegel:ql}. In particular, the $\psi^\bullet$-ql formalism is applicable to matings of neutral quadratic polynomials; see~\S\ref{exm:C1}. Then, the Sector Renormalization can be used as a substitute for the Douady-Hubbard straightening map.

\begin{thm}[$\psi^\bullet$-case, see Theorems~\ref{thm:main:psi-ql maps} and~\ref{thm:main:psi-ql maps:extra}]\label{thm:main:psi*-ql maps}
Let $F\colon U\rightrightarrows V$ be a $\psi^\bullet$-ql map with an eventually-golden-mean Siegel disk $Z_f$ at its $\alpha$-fixed point.
Let $K_F = W^\bullet(F)$ be the degeneration of $f$ around its Siegel disk.
Then there is a nested sequence of geodesic pseudo-Siegel disks $\wZ^m,\  m\ge -1$ such that ~\ref{thm:apbs:A} ~\ref{thm:apbs:C} ~\ref{thm:apbs:B} hold.

Moreover, all pseudo-Siegel disks $\wZ^m$ are constructed with explicit combinatorial thresholds~\eqref{eq:dfn:bM}, and $\wZ^{-1}$ is $\mu(K_F)$-qc disk; i.e.~the dilatation of $\wZ^{-1}$ is bounded in terms of $K_F$.
\end{thm}

\subsection{Outline of the paper} Let us here briefly outline the organization of the paper. We remark that all sections have short summaries at the beginning.

In Section~\ref{s:Defn}, we will introduce the notions of $\psi^\bullet$-ql Siegel maps (see Definition~\ref{dfn:psi:Siegel:ql}) and discuss how they naturally appear as inflations of rational maps around Siegel disks (see Figure~\ref{Fig:UnwindingSiegelFull} and Proposition~\ref{prop:g:unwind}).

In Section~\ref{s:Degener_of_psi_bullet}, we prepare degeneration tools to analyze $\psi^\bullet$-maps. The full families $\Fam$, the outer families $\Fam^+$, and various subtypes of the former play a central role in the paper, see \S\ref{subsec:familyofarcs}. Unlike the quadratic case, the first natural splitting of $\Fam^{+}$ is into the vertical $\Fam^{+,\ver}$ and the peripheral $\Fam^{+,\per}$ components, see~\eqref{eq:Fam^+}.

In Section~\ref{sec:parabolicfjords}, we extend the description of parabolic fjords from~\cite[\S4]{DL22} to the $\psi^\bullet$-setting. The analysis is similar to the case of quadratic polynomials and relies on shifts of wide rectangles (see~Items~\ref{shift:to:v} and~\ref{shift:to:w} in~\S\ref{ss:BalSubrect}). Key results are centered around the Log-Rule stated in Theorem~\ref{thm:parabolicfjords}. Pathological behaviors usually lead to Exponential Boosts in degeneration, see
~\S\ref{sss:ExponBoost}.
\cite[Corollary 7.3]{DL22} on the conditional regularization $\wZ^{m+1}\leadsto \wZ^m$ is restated as Theorem~\ref{thm:regul}. The argument is non-dynamical (relies on Log-Rules) and thus is applicable to the $\psi^\bullet$-setting.

Amplification Theorem~\ref{thm:SpreadingAround} extends \cite[Theorem 8.1]{DL22} to the $\psi^\bullet$-case. The statement and proof take into account a possible vertical degeneration; this leads to an additional Alternative~\ref{item:2:thm:SpreadingAround} in Theorem~\ref{thm:SpreadingAround}.

Calibration Lemma~\ref{lmm:CalibrLmm} extends~\cite[Lemma 9.1]{DL22} to the $\psi^\bullet$-setting. Since $\psi^\bullet$-maps can have vertical degeneration, Lemma~\ref{lmm:CalibrLmm} has an additional Alternative~\ref{Cal:Lmm:Concl:a} asserting a semi-equidistribution of such a vertical degeneration. This semi-equidistribution is refined in Lemma~\ref{lem:replication}, ``Equidistribution or Combinatorial Localization'', into equidistribution. The latter is a key preparation tool for Cases~\ref{thm:apbs:B} and~\ref{thm:apbs:C} of main Theorems~\ref{thm:main:psi-ql maps}~\ref{thm:main:psi*-ql maps}.

In Section~\ref{s:Proof}, we prove Theorem~\ref{thm:main:psi*-ql maps} restated as Theorems~\ref{thm:main:psi-ql maps} and~\ref{thm:main:psi-ql maps:extra}. The central induction is in the proof of Theorem~\ref{thm:main:psi-ql maps} claiming the equidistribution property. The induction is from deeper to shallower levels and is quite similar to the quadratic case: compare the diagram in Figure~\ref{fig:main pattern} with that in \cite[Figure 27]{DL22}. The key difference is the ``contradiction box $\bA_m \gg \bA_m$'' describing a possibility that an unexpected big vertical degeneration is developed. Theorem~\ref{thm:main:psi-ql maps:extra} establishes explicit combinatorial thresholds for regularization $\wZ^{m+1}\leadsto \wZ^m$; the proof of Theorem~\ref{thm:main:psi-ql maps:extra} is in the \emph{a posteriori} regime relative to Theorem~\ref{thm:main:psi-ql maps}.

\section{$\psi^\bullet$-ql Siegel maps}
\label{s:Defn}
In \S\ref{ss:psi bullet:defn}, we define the class of $\psi^\bullet$-ql Siegel maps; their iterations are discussed in ~\S\ref{subsec:iter}. Items~\ref{dfn:psib-ql:1} -- \ref{dfn:psib-ql:4} of Definition~\ref{dfn:psi:Siegel:ql} are similar to the case of $\psi^\bullet$-quadratic-like maps with a bush being replaced with a Sigel disk $\overline Z$, see~\cite[\S~3.4.2]{DL:Feigen}. Items~\ref{dfn:psib-ql:5} and \ref{dfn:psib-ql:6} are new requirements; they are naturally satisfied in applications (see Proposition~\ref{prop:g:unwind}) and are used in defining the pullback operation $F^*=\iota_*\circ f^*$ within zones, see the discussion at the beginning of Section~\ref{s:Degener_of_psi_bullet}.

In~\S\ref{ss:inflat} and~\S\ref{ss:exms}, we show how $\psi^\bullet$-Siegel maps naturally arise from rational dynamics by inflating a Siegel prerenormalization, see Definition~\ref{dfn:SiegelPreNorm}, Proposition~\ref{prop:g:unwind}, and Remark~\ref{rem:prop:g:unwind}.

\subsection{Definition of $\psi^\bullet$-ql maps}\label{ss:psi bullet:defn}
A map $\iota\colon A\to B$ between open Riemann surfaces is called an \emph{immersion} if every $x\in X$ has a neighborhood $U_x$ such that $\iota\colon  U_x\to \iota(U_x)$ is a conformal isomorphism, see an example on Figure~\ref{Fig:psibullet maps}. A compact subset $S\Subset B$ is called \emph{$\iota$-proper} if $\iota\colon \iota^{-1}(S)\to S$ is a homeomorphism. In this case, we often \emph{identify} $S\simeq \iota^{-1}(S)$. 

We say that an immersion $\iota\colon A\to B$ is a \emph{covering embedding rel $S\subset B$} if 
\begin{itemize}
  \item $S$ is $\iota$-proper;
  \item $\iota\mid A\setminus \iota^{-1}(S)$ is a covering onto the image; i.e., 
  \[\iota \colon \ A\setminus \iota^{-1}(S)\longrightarrow \iota(A)\setminus S\]
is a covering map. 
\end{itemize}

\begin{defn}
\label{dfn:psi:Siegel:ql}
A {\em  pseudo${}^\bullet$-quadratic-like Siegel map} (``$\psi^\bullet$-ql Siegel map'') 
is a pair of holomorphic  maps
\begin{equation}
\label{eq:dfn:psib-ql}
F=(f,\iota)\colon \sp ( U, \overline Z_U)  \rightrightarrows (V, \overline Z ),\hspace{0.5cm}\text{ so }\sp  \overline Z_U \subseteq f^{-1}(\overline Z) \cap  \iota^{-1}(\overline Z) 
\end{equation} between two conformal  disks $U$, $V$
with the following properties:

\begin{enumerate}[label=\text{(\Roman*)},font=\normalfont,leftmargin=*]
\item\label{dfn:psib-ql:1}  $f\colon U\to V$ is a double branched covering with a unique critical point $c_0\in \partial Z_U$;

\item \label{dfn:psib-ql:2} $\overline Z$ is $\iota$-proper; in particular, $\iota \colon \overline Z_U\ \overset{\simeq}\longrightarrow \ \overline Z;$

\item \label{dfn:psib-ql:5}
 there exist neighborhoods $X_U\supset \overline Z _U$ and $X \supset \overline Z$ with the following property: $\iota: X_U \to X$ is a conformal isomorphism such that
\begin{equation}\label{eq:dfn:ql b domain}
f_X\coloneqq f\circ \big(\iota\mid X_U\big)^{-1} : X \to f(X_U)\eqqcolon Y 
\end{equation} 
 is a \emph{Siegel map}: $\overline Z \Subset X\cap Y$ is a closed qc Siegel disk around the fixed point $\alpha \in Z=\intr \overline Z$ with bounded-type rotation number;

\end{enumerate}
Define inductively $K_0\coloneqq \overline Z$, $K_{1, U}\coloneqq f^{-1}(\overline Z)$, $K_{1}\coloneqq\iota(K_{1, U})$,  and, for $n\ge 1$   
\begin{equation}
\label{eq:C:dfn:K_n}
\begin{matrix} K_{n,U}\coloneqq f^{-1}(K_{n-1})=f^{-1}\circ \big(\iota \circ f^{-1}\big)^{n-1} (\overline Z), \vspace{0.1cm}\\ K_n \coloneqq \iota(K_{n,U})=\big(\iota \circ f^{-1}\big)^n (\overline Z);\hspace{1.2cm} \end{matrix}
\end{equation}
\begin{enumerate}[label=\text{(\Roman*)},font=\normalfont,leftmargin=*,start=4]
\item  \label{dfn:psib-ql:4} for all $n\ge 0$, the restriction
$\iota \colon \ K_{n,U} \overset{\simeq}{\longrightarrow} K_n $ is a homomorphism;

\item \label{dfn:psib-ql:5} $\iota\colon U\to V$ is a covering embedding rel $K_1$;
\end{enumerate}

We remark that $F\colon U\rightrightarrows V$ can be naturally iterated $F^n\colon U^n\rightrightarrows V$, see~\S\ref{subsec:iter}. We also require that
\begin{enumerate}[label=\text{(\Roman*)},font=\normalfont,leftmargin=*,start=6]
\item \label{dfn:psib-ql:6} for all $n\ge 1$, the iterates $\iota^n\colon U^n \to V$ are covering embedding rel $K_n$.
\end{enumerate}
\end{defn}

Since $\iota$ is a conformal isomorphism in a neighborhood of $\overline Z$, we will below identify 
\begin{equation}\label{eq:Z is Z_U}
\overline Z\simeq \overline Z_U\equiv \overline Z_F\sp\sp \text{ and write }\sp F\colon ( U, \overline Z)  \rightrightarrows (V, \overline Z) \sp\sp\text{ or }\sp F\colon U \rightrightarrows V.
\end{equation}
Similarly, we identify $K_{n}\simeq K_{n,U}$.  Item~\ref{dfn:psib-ql:4} is illustrated on Figure~\ref{Fig:Item4}.

\begin{figure}
   \includegraphics[width=13cm]{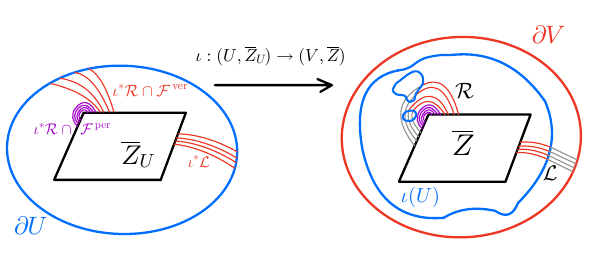}
   \caption{An example of an immersion $\iota$, see also~\S\ref{ss:push forw curves}. A vertical lamination $\LL$ always lifts to a vertical lamination $\iota^*\LL$. The lift of a peripheral lamination $\RR$ can have vertical and peripheral components (\S~\ref{subsec:familyofarcs}).}
   \label{Fig:psibullet maps}
\end{figure}

The \emph{width} of $F$ is 
\[\Width^\bullet(F)\coloneqq \Width(V\setminus \overline Z),\]
where $\Width(V\setminus \overline Z) = \frac{1}{\mod(V\setminus \overline Z)}.$
If $\Width^\bullet(F)\le K$, then $X_U,~X$ in Item~\ref{dfn:psib-ql:5} can be selected so that
\begin{equation}
\label{eq:app:good X}
\mod(X\setminus \overline Z)\ge \varepsilon(K). 
\end{equation}

\subsection{Iteration ${F^k\colon U^k\rightrightarrows V}$}\label{subsec:iter}
We recall from~\cite[\S2.2.2]{K},~\cite[\S 3.4.6]{DL:Feigen}  that $F\colon U\rightrightarrows V$ has the natural \emph{fiber product} (also known as the $\psi$-\emph{restriction}, or \emph{pullback}, or the \emph{graph}) 
denoted by 
\[F=(f,\iota)\colon U^2\rightrightarrows  U^1=U,\sp\sp \text{ where }\sp U^2=\{(x,y)\in U\times U\mid\sp f(y)=\iota(x)\},\]
where $f,\iota\colon U^2\rightrightarrows U$ are component-wise projections:
\[\begin{tikzpicture}
\node (A00) at (0,0){$ (x,y)\in U^2 $}; 
\node (A01) at (0,-1.2){$ y\in U $};
\node (A10) at (7,0){$ x\in U $}; 
 \node (A11) at (7,-1.2){$ f(y)=\iota(x)\in V .$}; 
\draw (A00) edge[->] node[right]{$\iota\sp\sp (\psi-\text{restriction})$} (A01);
\draw (A10) edge[->] node[right]{$\iota$} (A11);
\draw (A00) edge[->] node[above]{$f\sp\sp (\psi-\text{restriction}$)} (A10);
\draw (A01) edge[->] node[below]{$f$} (A11);

\end{tikzpicture}
\]

Repeating the construction, we obtain the sequence
\[F\colon U^k\rightrightarrows U^{k-1},\hspace{1cm} n\ge 1, \sp U^0=V ,\sp U^1=U,\]
together with induced iterations denoted by
\begin{equation}
\label{eq:F^k}
 F^k = \big(f^k, \iota^k\big) \colon U^k \to V.
\end{equation}  
Note that by construction, each $U^k$ is conformally a disk, and $f: U^k \longrightarrow U^{k-1}$ is a degree $2$ branched covering.

Note that $\iota\mid U^{k}$ is the lift of $\iota\mid U^{k-1}$ under the covering $f$:
$$
\begin{tikzcd}
& U^k\setminus  K_{k} \arrow[dl, ->, "f"'] \arrow[dr, ->, "\iota"] & \\
U^{k-1}\setminus  K_{k-1} \arrow[dr, ->, "\iota"'] & & \iota\big( U^{k}\setminus  K_{k}\big) \arrow[dl, ->, "f"] \\
& \iota\big( U^{k-1}\setminus  K_{k-1}\big) &
\end{tikzcd}
$$

Therefore, Items~\ref{dfn:psib-ql:4} and~\ref{dfn:psib-ql:5} for $\iota\mid U$ from~\S\ref{ss:psi bullet:defn} imply the following properties for $\iota\mid U^k$:
\begin{enumerate}[label=\text{(\Roman*${}^k$)},font=\normalfont,leftmargin=*,start =4]

\item  \label{dfn:psib-ql:4:4} for all $n\ge 0$ and all $k>0$, the immersion $\iota\colon U^k\to U^{k-1}$ restricts to a homemorphism from $K_{n, U^k} \simeq K_n$ onto $K_{n, U^{k-1}}\simeq K_n$;
\end{enumerate}
\begin{enumerate}[label=\text{(\Roman*${}^k$)},font=\normalfont,leftmargin=*,start =5]
\item \label{dfn:psib-ql:2:+} $\iota\colon U^{k+1}\to U^{k}$ is a covering embedding rel $K_{k+1}$.
\end{enumerate}
Similarly, Item~\ref{dfn:psib-ql:6} implies a more general property:
\begin{enumerate}[label=\text{(\Roman*${}^k$)},font=\normalfont,leftmargin=*,start =6]
\item \label{dfn:psib-ql:6:+} for all $n\ge 1$ and $k\ge 0$, the iterats $\iota^n\colon U^{n+k}\to U^{k}$ are covering embedding rel $K_{n+k}$.
\end{enumerate}
In other words,  $F\colon U^k\rightrightarrows  U^{k-1}$ is almost a $\psi^\bullet$-ql map where Item~\ref{dfn:psib-ql:5} is replaced by Item~\ref{dfn:psib-ql:2:+}.

\begin{figure}
   \includegraphics[width=13cm]{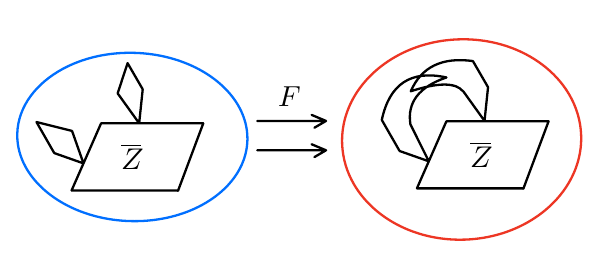}
   \caption{Item \ref{dfn:psib-ql:4} in~\S\ref{ss:psi bullet:defn} prevents overlapping of immersed limbs as illustrated on this figure.}
   \label{Fig:Item4}
\end{figure}

\subsection{Inflation of rational maps} \label{ss:inflat} In this subsection, we will consider a more general setting of polynomial-like maps (``pl maps'') and demonstrate that $\psi^\bullet$-pl Siegel maps naturally emerge by inflation of Siegel prerenormalizations of rational maps; see Proposition~\ref{prop:g:unwind} below. A key assumption about the Siegel prerenormalization (Definition~\ref{dfn:SiegelPreNorm}) is the existence of a partial self-covering restriction~\eqref{eq:part:self-cover}.

In degree $d=2$, Proposition~\ref{prop:g:unwind} constructs a $\psi^\bullet$-ql Siegel map satisfying Definition~\ref{dfn:psi:Siegel:ql}. We will not provide an axiomatization of $\psi^\bullet$-pl Siegel maps in degree $d>2$; Proposition~\ref{prop:g:unwind} should be taken as a guidance, see Remark~\ref{rem:prop:g:unwind}.

Assume $g\colon \wC\selfmap$ is a rational map with a periodic bounded-type Siegel disk $Z$. By~\cite{Zha11},  $\overline Z$ is a closed qc disk containing at least one critical point on the boundary. Let us fix an iterate 
\begin{equation}\label{eq:dfn:f:exm}
 f\coloneqq g^p,\hspace{1cm} f(\overline Z)=\overline Z,
\end{equation} 
preserving the Siegel disk; we do not require $p$ to be the minimal period.

Let us next assume that there is a forward invariant set $\bUpsilon=f(\bUpsilon)$ containing the postcritical set of $f$ (which is equal to the postcritical set of $g$) such that $\overline Z$ is a connected component of $\bUpsilon$. Then 
\begin{equation}
\label{eq:part:self-cover}
(f,\ \hookrightarrow)\ \colon\ \wC\setminus f^{-1}(\bUpsilon) \ \rightrightarrows\ \wC\setminus \bUpsilon
\end{equation}
is a \emph{partial self-covering}: $f$ represents a covering from a smaller set to a bigger set.

Let $B$ be any open Jordan neighborhood of $\overline{Z}$ such that $\bUpsilon \cap B=\overline Z$. Let $A$ be a unique component of $f^{-1}(B)$ that contains $\overline Z$. 
If 
\begin{itemize}
    \item $\bUpsilon \cap A= \overline Z,$ and
    \item $\partial A$ and $\partial B$ are homotopic rel $\bUpsilon$,
\end{itemize}
then we say that $\bUpsilon$ is \emph{locally saturated at $\overline Z$;} this notion is independent of the choice of $B$. This implies that $A$ is also a Jordan neighborhood of $\overline Z$ and we have a branched covering map
\begin{equation}
\label{eq:f:BtoA:pre} f=g^p\ \colon A\overset{d:1}\longrightarrow B
\end{equation} restricting to a covering $f\colon A\setminus K_1\to B\setminus \overline Z$, where $K_1\coloneqq [f\mid A]^{-1}(\overline Z)$.

\begin{defn}[Siegel prerenormalization]\label{dfn:SiegelPreNorm} Assume a rational map $g\colon \wC\selfmap$ is endowed with  the following data:
\begin{itemize}
\item an iterate~\eqref{eq:dfn:f:exm} around a Siegel disk $\overline Z$,
\item a partial self-covering restriction~\eqref{eq:part:self-cover} so that $\bUpsilon$ is locally saturated at $\overline Z$.
\end{itemize} 
Then the map
\begin{equation}
\label{eq:f:BtoA} f=g^p\ \colon A\overset{d:1}\longrightarrow B
\end{equation} a \emph{Siegel prerenormalization} of $f$ at $\overline Z$.
\end{defn}

The \emph{inflation of~\eqref{eq:f:BtoA} by means of~\eqref{eq:part:self-cover}} is the extension of $f\colon A\to B$ along all paths in $\wC\setminus \bUpsilon$; see~\S\ref{sss:dfn:unwinding} for details.  We have:

\begin{prop}[Inflation of~\eqref{eq:f:BtoA}, see Figure~\ref{Fig:UnwindingSiegelFull}]\label{prop:g:unwind} Consider a rational map $g\colon\wC\selfmap$ with a Siegel prerenormalization as in Definition~\ref{dfn:SiegelPreNorm}. Then the inflation  of~\eqref{eq:f:BtoA} by means of~\eqref{eq:part:self-cover} is a correspondence   
\begin{equation}
\label{eq:prop:g:unwind}
F=(f,\iota)\colon U\rightrightarrows V
\end{equation}
where
\begin{enumerate}[label=\text{(\Alph*)},font=\normalfont,leftmargin=*]
\item\label{item:A:prop:g:unwind} $f\colon U\overset{d:1}\longrightarrow V$ is a proper branched covering extending $f\colon A\overset{d:1}\longrightarrow  B$; and 
\item\label{item:B:prop:g:unwind} $\iota\colon U\longrightarrow V$ is a covering embedding rel $K_1$ extending $A\cap B \ \overset{\simeq}{\longrightarrow}\  A\cap B$. \\ (In particular, $K_1$ is $\iota$-proper.)
\end{enumerate}

Moreover, the iterate \[F^k=(f^k,\iota^k)\colon U^{n+k}\rightrightarrows U^n,\hspace{1cm} U^0=V\] is the inflation of \[f^k\ \colon A^{n+k}\ \overset{d^k:1}{\longrightarrow}\ A^n,\hspace{1cm} A^0=B\]
by means of 
\[(f^k,\ \hookrightarrow)\ \colon\ \wC\setminus f^{-n-k}(\bUpsilon) \ \rightrightarrows\ \wC\setminus  f^{-n}(\bUpsilon).\]
Here $A^n$ is the unique component of $f^{-n}(B)$ containing $\overline Z$.

If $d=2$, i.e., $K_1$ contains a unique simple critical point, then~\eqref{eq:prop:g:unwind} is a $\psi^\bullet$-ql Siegel map satisfying Definition~\ref{dfn:psi:Siegel:ql}.
\end{prop}

The proof of Proposition~\ref{prop:g:unwind} is given in~\S\ref{sss:dfn:unwinding}.

\begin{rem}[Variations of Proposition~\ref{prop:g:unwind}]\label{rem:prop:g:unwind} As we have already mentioned, Proposition~\ref{prop:g:unwind} should be viewed as a guidance in degree $d>2$. Let us mention two possible modifications of the setting.

First, the local saturation condition at $\overline Z$ can be relaxed into a local saturation condition at $K_m$ for some $m\ge 1$ (i.e., $\bUpsilon$ may intersect $K_m$ but not $K_{m+1}\setminus K_m$). 

Second, instead of dealing with the first return map, one can develop the concept of ``$\psi^\bullet$-correspondences with multiple connected components''
\begin{equation}
    \label{eq:rem:prop:g:unwind}
F=(f,\iota)\colon \ \Big(\bigcup_{i} U_i\Big)\ \rightrightarrows \ \Big(\bigcup_{i} V_i\Big)\ \hspace{1cm} f\colon U_i\to V_{f(i)}, \sp \iota\colon U_i\to V_{i},
\end{equation}
where the $U_i$ are chosen to contain the original critical points of $g$. (The first return map can significantly increase the number of critical points.) Framework~\eqref{eq:rem:prop:g:unwind} is convenient for developing the associated renormalization theory.
\end{rem}

\begin{figure}
   \includegraphics[width=13.5cm]{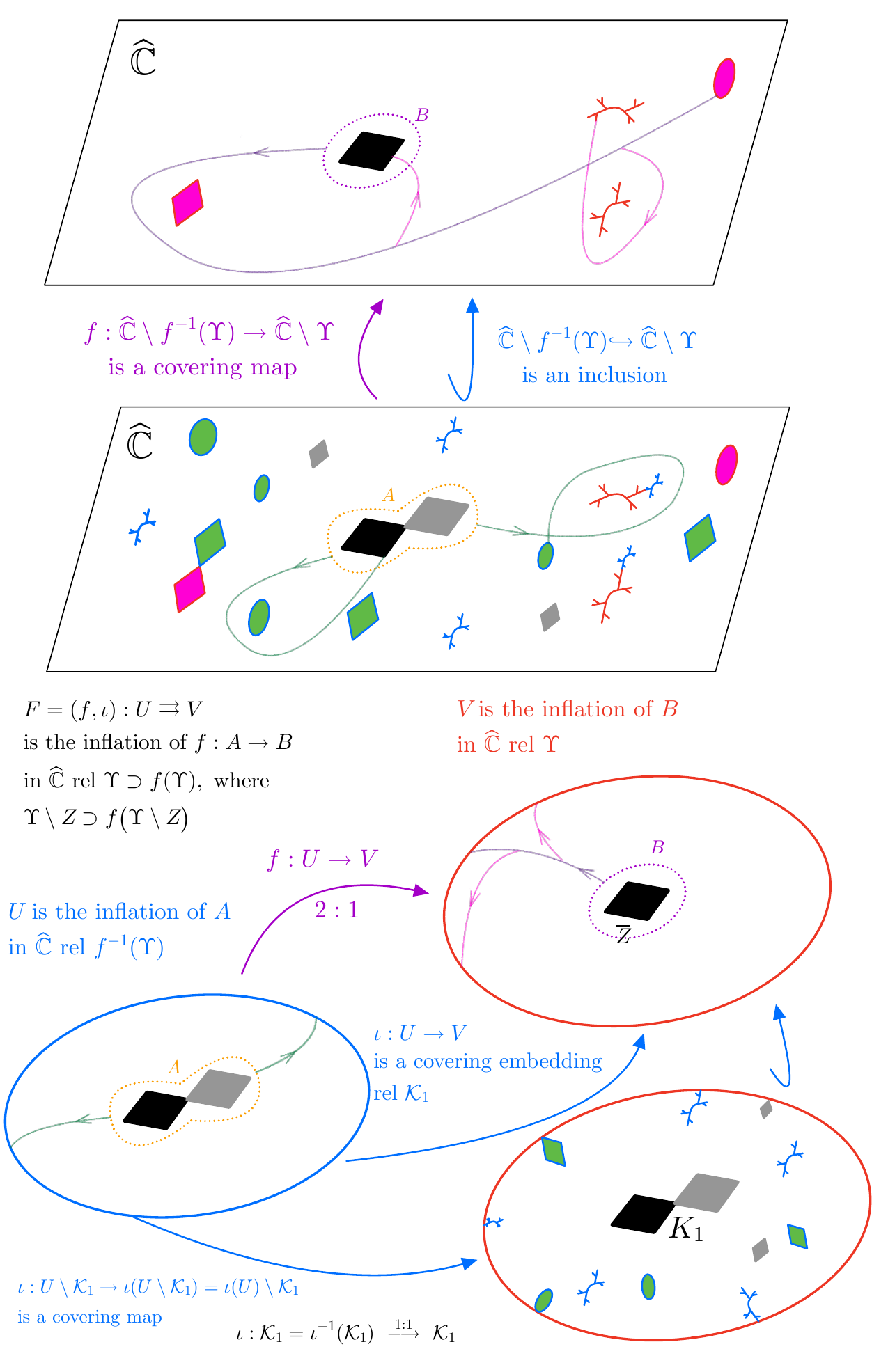}
   \caption{Illustration to Proposition~\ref{prop:g:unwind}: the inflation of $f\colon A\to B$ in $\wC$ around a Siegel disk $\overline Z$
    assuming that $\bUpsilon$ controls the postcritical set; see~\S\ref{sss:Fig:UnwindingSiegelFull} for details regarding the figure.}
   \label{Fig:UnwindingSiegelFull}
\end{figure}

\subsubsection{Inflation: detailed description} \label{sss:dfn:unwinding} Let $V$ be the inflation of $B$ in $\wC$ rel $\bUpsilon$; the inflation is obtained by extending $B$ along all paths in $\wC\setminus \bUpsilon$ defined as follows. Let $\mathbb{A}^0$ be the universal covering of $\wC \setminus \bUpsilon$ rel $B'\coloneqq B\setminus \partial Z$. This means that $\mathbb{A}^0\to\wC \setminus \bUpsilon$ is a covering that opens up all loops outside of $B'$. In particular,  $\mathbb{A}^0$ is annulus with $\pi_1(\mathbb{A}^0) \simeq \pi_1(B')$ via $\mathbb{A}^0\to\wC \setminus \bUpsilon$. Note that $\mathbb{A}^0$ has a neighborhood boundary component that is conformally identified with $B'$. The inflation of $B$ in $\wC$ rel $\bUpsilon$ is obtained by gluing $\mathbb{A}^0$ and $B$ along $B'$
$$
V = \mathbb{A}^0 \cup B/\text{\Small $B'$} \simeq \mathbb{A}^0 \sqcup \overline Z.
$$

Similarly, let $U$ be the inflation of $A$ in $\wC$ rel $f^{-1}(\bUpsilon)$:
$$
V = \mathbb{A}^1 \cup A/\text{\Small $A'$} \simeq \mathbb{A}^1 \sqcup K_1,
$$
where $A'\coloneqq A\setminus K_1$ and $\mathbb{A}^1$ is the universal covering of $\wC\setminus f^{-1}(\bUpsilon)$ rel $A'$. Since $A\setminus K_1$ is an annulus, $U\setminus K_1$ is an annulus.

The covering map $f\colon A \overset{d:1}{\longrightarrow} B$ lifts to a degree $d$ branched covering map $f: U \longrightarrow V$, and the inclusion map $\iota:\wC \setminus f^{-1}(\bUpsilon) \xhookrightarrow{} \wC \setminus \bUpsilon$ lifts to a covering embedding $\iota: U \longrightarrow V$ rel $K_1$.

Since $\bUpsilon$ is locally saturated at $\overline Z$, the set $K_1$ is $\iota$-proper.

\begin{proof}[Proof of Proposition~\ref{prop:g:unwind}] 
Item~\ref{item:A:prop:g:unwind} has been justified above. 

The local saturation Condition~\ref{eq:part:self-cover} implies that, up to homotopy $B$ rel $\bUpsilon$, we may assume that $K_1\subset A\cap B$. This implies Item~\ref{item:B:prop:g:unwind} by a routine argument. 

The claim concerning the iterate is standard; it follows from the naturality of the definition of the inflation.
 
Assume that $d=2$. Items~\ref{dfn:psib-ql:1}--\ref{dfn:psib-ql:5} of Definition~\ref{dfn:psi:Siegel:ql} are immediate. Item~\ref{dfn:psib-ql:6} of Definition~\ref{dfn:psi:Siegel:ql}  follows from Item~\ref{item:B:prop:g:unwind} of Proposition~\ref{prop:g:unwind} applied to an iterate of $f$.
\end{proof}

\subsubsection{Comments to Figure~\ref{Fig:UnwindingSiegelFull}}\label{sss:Fig:UnwindingSiegelFull} 

The upper half of the figure illustrates a partial self-covering \eqref{eq:part:self-cover}. The Siegel disk $\overline Z$ from~\eqref{eq:dfn:f:exm} is colored in black; its preimage $K_1$ under $f\colon A\to B$ is colored in black and gray. Components of $\bUpsilon\setminus \overline Z$ are colored in red, while components of $f^{-1}(\bUpsilon)\setminus (\bUpsilon\cup K_1)$ are colored in green and blue.

The inflation of $f\colon A\to B$ by means of~\eqref{eq:part:self-cover} is illustrated in the lower half of the figure. The set $V$ is obtained by extending $B$ along all paths in $\wC\setminus \bUpsilon$; example paths are colored in purple and red. The set $U$ is obtained from by extending $A$ along all paths in $\wC\setminus f^{-1}(\bUpsilon)$; example paths are colored in green.

\subsection{Examples}\label{ss:exms} Let us discuss how $\psi^\bullet$-maps emerge from the rational dynamics. In~\S\ref{exm:C1}, we consider the hyperbolic component $ \mathcal H^\text{rat}_{z^2}$ of $z\mapsto z^2$. The most complicated part of its boundary $\partial H^\text{rat}_{z^2}$ is where both attracting fixed points become neutral. We give an explicit extraction of $\psi^\bullet$-ql maps; the $\iota$-map is, in fact, an inclusion in this case. General hyperbolic components are discussed in~\S\ref{ss:exm:general} and in~\S\ref{sss:disj:type}.

\subsubsection{Matings of neutral quadratic polynomials}\label{exm:C1} Let us illustrate the setting of~\S\ref{ss:inflat} for matings of quadratic polynomials with neutral $\alpha$-fixed points. Such matings arise on the boundary of the hyperbolic component $\mathcal H^\text{rat}_{z^2}$ of $z\mapsto z^2$ in the space of quadratic rational maps.

Consider $g\in \partial \mathcal H^\text{rat}_{z^2}$ with two neutral fixed points at $0$ and $\infty$ that have rotation numbers $\theta_0,\theta_\infty\in \R\setminus \Q$ with $\theta_0\not= -\theta_\infty$. Assume first that  $\theta_0,\theta_\infty\in \Theta_\bnd$ are of bounded type. Then $g$ has closed Siegel qc-disks $\overline Z_0, \overline Z_\infty$ at $0$ and $\infty$. We set:
\begin{itemize}
\item $f_0=f_\infty\coloneqq g$,
\item $\bUpsilon_{f_0}=\bUpsilon_{f_\infty}=\bUpsilon\coloneqq \overline Z_0 \cup \overline Z_\infty$; then the partial self-covering map~\eqref{eq:part:self-cover} has the form
\begin{equation}
\label{eq:matings:Upsilon} (f_0=f_\infty,=g, \ \ \hookrightarrow)\ \colon\ \wC\setminus g^{-1}(\bUpsilon) \ \rightrightarrows\ \wC\setminus \bUpsilon,
\end{equation}
where $g^{-1}(\bUpsilon) = g^{-1}(\overline Z_0) \cup g^{-1}(\overline Z_\infty)$ consists of $\bUpsilon$ and two prefixed (immediate preimages) Siegel disks;
\item $B_0$ and $B_\infty$ are open Jordan neighborhoods of $\overline Z_0$ and $\overline Z_\infty$ such that
\[B_0\cap  \bUpsilon =\overline Z_0 \hspace{1cm}\text{ and } \hspace{1cm} B_\infty\cap  \bUpsilon =\overline Z_\infty,\]
\item $A_0\coloneqq g^{-1}(B_0)$ and $A_\infty\coloneqq g^{-1}(B_\infty)$ so that 
\begin{equation}
\label{eq:matings:quadr} f_0\equiv g \colon A_0\ \overset{2:1}{\longrightarrow} \ B_0 \hspace{1cm}\text{ and } \hspace{1cm}    f_\infty\equiv g \colon A_\infty\ \overset{2:1}{\longrightarrow} \ B_\infty,
\end{equation}
\end{itemize}

Then the inflations of maps in~\eqref{eq:matings:quadr} by means of~\eqref{eq:matings:Upsilon} have the following forms:
\[F_0= (f_0=g,\ \hookrightarrow )\  \colon\  (U_0, \overline Z_0)\rightrightarrows (V_0,  \overline Z_0), \sp\sp V_0=\wC \setminus \overline Z_\infty, \ U_0=g^{-1}(V_0),\]
\[F_\infty= (f_\infty=g, \ \hookrightarrow )\  \colon\  (U_\infty, \overline Z_\infty)\rightrightarrows (V_\infty,  \overline Z_\infty), \sp\sp V_\infty=\wC \setminus \overline Z_0, \ U_\infty=g^{-1}(V_\infty)\]
where the immersion $\iota=$``$\hookrightarrow$'' is an embedding. Both $F_0, F_\infty$ are $\psi^\bullet$-ql maps.

 We have:
\begin{equation}
\label{eq:Width:F_0}
\Width^\bullet(F_0) =\Width^\bullet(F_\infty)=\Width\big(\wC\setminus [\overline Z_0\cup \overline Z_\infty]\big).
\end{equation}
By~\cite{DL:HypComp},  $\Width^\bullet(G_0) $ is bounded in terms of the (combinatorial) distance between $\theta_0$ and $-\theta_\infty$ and, consequently, such bound for $\Width^\bullet(F_0)=\Width^\bullet(F_\infty)$ persists for all  $\theta_0,\theta_\infty \in \R\setminus \Q$. Consequently, correspondences $F_0$ and $F_\infty$ can be defined for all such irrational $\theta_0\not= -\theta_\infty$. For the original $g$, the $0$ and $\infty$ have disjoint (possibly Cremer) Mother Hedgehogs canonically identified with the Mother Hedgehogs of the associated quadratic polynomials.

\subsubsection{Neutral Dynamics at boundaries of hyperbolic components}\label{ss:exm:general} We expect that the setting of~\S\ref{exm:C1} can be extended, with appropriate refinements and modifications, to all hyperbolic component $\HH$ in the space of rational maps with connected Julia sets. Namely, every such $\HH$ has a unique postcritically finite (PCF) center $f_{pcf}\in \HH$ encoding the combinatorics of $\HH$. For $f\in \HH$, let us denote by $\mathcal C_f$ the centers of Fatou components containing either a critical or a postcritical point. If $f=f_{pcf}$, then $\mathcal C_{f_{pcf}}$ it is its critical and postcritical set. For $\tau\in \mathcal C_f$, let $F_\tau$ be the Fatou component centered at $\tau$. We naturally have maps
\begin{equation}
\label{eq:BlCl:f_pcf}
\big[f_{pcf}\colon F_\tau\ \overset{d_\tau:1}{\longrightarrow}\ F_{f(\tau)}\big] \ \ \overset{\text{RM}}{\simeq}\  \ \big[ z\mapsto z^{d_\tau}\ \colon \ \Disk \overset{d_\tau:1}{\longrightarrow}\ \Disk\big],
\end{equation}
where $f(\tau)=f_{pcf}(\tau)$ is the image of $\tau$, the identification ``$\ \overset{\text{RM}}{\simeq}\ $'' represents a marked Riemann map between $(F_\tau, \tau)$ and the unit disk $(\Disk, 0)$, and the ``$z\mapsto z^{d_\tau}$'' represents the Blaschke class of $f_{pcf}\mid F_\tau$.

After introducing a finite set of markings for $\HH$ to resolve its orbifold points, the set $\mathcal C_f$ moves holomorphically with $f\in \HH$; i.e., $\mathcal C_f$ is canonically identified between maps in $\mathcal H$. Equation~\eqref{eq:BlCl:f_pcf} then takes form
\begin{equation}
\label{eq:BlCl:f}
\big[f\colon F_{\tau,f}\ \overset{d_\tau:1}{\longrightarrow}\ F_{f(\tau),f}\big] \ \ \overset{\text{RM}}{\simeq}\  \ \big[ B_{\tau,f}\ \colon \ \Disk \overset{d_\tau:1}{\longrightarrow}\ \Disk\big],
\end{equation}
where $B_{\tau,f}$ is the Blaschke class of $f$ at $\tau$. The class $B_{\tau,f}$ is trivial if $d_\tau=1$. The set of all Blaschke classes $(B_{\tau,f})$ provide a natural parametrization of $\HH$. At the boundary $\partial \HH$, some of the points in $\mathcal C_f$ may become neutral. 

\begin{conj}[Mother Hedgehog is well defined] \label{conj:MH}If a periodic $\tau\in \mathcal C_f$ becomes neutral non-parabolic in the dynamical plane of $f\in \partial \HH$, then there is a well-defined periodic Mother Hedgehog $H_{\tau,f}$ centered at $\tau\equiv \tau(f)$. The periodic cycle of $H_{\tau,f}$ contains at least one critical point of $f$. Sector Renormalization provides an $f$-equivariant qc identification of $H_f$ with a Mother Hedgehog of some polynomial.
\end{conj}

We believe that the theory of $\psi^\bullet$-ql maps can be developed to handle Conjecture~\ref{conj:MH}.

\subsubsection{Disjoint type hyperbolic components} \label{sss:disj:type}If $\mathcal{H}$ is a hyperbolic component of disjoint type, i.e., any map $f \in \mathcal{H}$ has exactly $2d-2$ attracting cycles, then the multipliers of non-repelling cycles uniquely parametrize $\mathcal H$:
\[ \mathcal H \ \simeq \ \Disk^{2d-2};\]
in particular, Blaschke classes in~\eqref{eq:BlCl:f} are uniquely determined by the multipliers.

Let $g \in \partial \mathcal{H}$ be a rational map with non-repelling periodic cycles $\mathcal{C}_j$, $j = 1,..., 2d-2$ so that the multiplier $\rho_j$ is either $|\rho_j| < 1$ or $\rho_j = e^{2\pi i \theta_j}$ with $\theta_j \in \R/\Q$ of bounded type.
Let $\Upsilon$ be the closure of the union of periodic Siegel disks and valuable parts of the periodic attracting components. (Roughly, the latter are bounded by equipotentials though the critical value.) Let $\overline Z\subset \bUpsilon$ be a Siegel disk with minimal period $p$. Then Proposition~\ref{prop:g:unwind} constructs a $\psi^\bullet$-ql Siegel map $F=(f, \iota)\colon U \rightrightarrows V$ by inflation $f=g^p$ restricted to a neighborhood of $\overline Z$ rel $\bUpsilon$ (c.f. \cite[\S3]{DL:HypComp}).
\begin{comment}
Denote the corresponding Siegel disks and attracting Fatou components by $\Omega_{j,k}$, where $j = 1,..., 2d-2$, $k = 1,..., p_j$ and $p_j$ is the period of the cycle $\mathcal{C}_i$.
Suppose that $Z = \Omega_{1,1}$ is a Siegel disk.
Then we can construct a $\psi^\bullet$-ql Siegel map as follows.

Let $X_0 = \widehat\C \setminus \bigcup_{j,k} \Omega_{j,k}$, and $X_1 = g^{-p_1}(X_0) \subseteq \widehat\C$.
Let $\mathcal{A}^0$ be the covering of $X_0$ associated to $\partial Z$, and let $\mathcal{A}^1$ be the covering of $X_1$ associated to the boundary of the component $K_1$ of $g^{-p_1}(\overline{Z})$ that contains $\overline{Z}$.
This means that $\mathcal{A}^0$ (or $\mathcal{A}^1$) is an annulus, and the generator of $\pi_1(\mathcal{A}^0)$ or $\mathcal{A}^1$) is mapped to a closed curve homotopic to $\partial Z$ (or $\partial K_1$) respectively.

Note that $\mathcal{A}^0$ has a neighborhood boundary component that is conformally mapped a one-sided neighborhood of $\partial Z$, similarly for $\mathcal{A}^1$.
We define 
$$
V = \mathcal{A}^0 \cup Z \text{ and }U = \mathcal{A}^1 \cup K_1.
$$
The covering map $g^{p_1}: X_1 \longrightarrow X_0$ lifts to a degree $2$ branched covering map $f: U \longrightarrow V$.
The inclusion map $i: X_1 \xhookrightarrow{} X_0$ lifts to an immersion $\iota: U \longrightarrow V$.
This construction is one of the motivations to study these $\psi^\bullet$-maps (see~\cite{DL:HypComp}). 
\end{comment}

\section{Degenerations of $\psi^\bullet$-ql maps}\label{s:Degener_of_psi_bullet}

Recall that a $\psi^\bullet$ Siegel map $F=(f,\iota)\colon U\rightrightarrows V$ consists of a branched covering $f$ and a covering embedding $\iota\colon U\to V$ rel $K_1$. In~\S\ref{subsec:coastal}, we deduce from Item~\ref{dfn:psib-ql:5} of Definition~\ref{dfn:psi:Siegel:ql} that $\iota$ is, in fact, an embedding on the {\em coastal zone} of $K_1$ -- an open disk bounded by the hyperbolic geodesic. Similarly, using Item~\ref{dfn:psib-ql:6:+}, an iterate $\iota^m\colon U^m\to V$ is an embedding in the {\em coastal zone} of $K_m$. 

In~\S\ref{subsec:familyofarcs}, we introduce full families $\Fam$, outer families $\Fam^+$, and various subtypes of the latter, see~\eqref{eq:Fam^+}. These families play a central role in our analysis of the degeneration of Siegel disks; see ~\S\ref{sss:occ:Fam} and~\S\ref{sss:occ:Fam^+} for a brief overview.

We then define the pull-back and push-forward operations on families of arcs and study their properties in \S~\ref{subsec:pullback} and \S~\ref{ss:push forw curves}, respectively. Note that in $F=f\circ \iota^{-1}$, the $\iota$ is more delicate. In particular, the pushforward $F_*=f_*\circ \iota^*$ is well-defined because $f$ is a branched covering; see Figure~\ref{Fig:psi:PushForward}. On the other hand, within coastal zones, the pullback $F^*=\iota_* \circ f^*$ is essentially just $f^*$, which works similarly as for quadratic polynomials. 

Finally, we introduce and discuss some modifications required of pseudo-Siegel disks for $\psi^\bullet$ maps in \S\ref{sss:iota peripheral}, \S\ref{ss: propertyimmersion}, and \S\ref{sss:SpreadAround:rel:wZ}.

\hide{\begin{itemize}
\item Recall that a $\psi^\bullet$ Siegel map $F=(f,\iota)\colon U\rightrightarrows V$ consists of a branched covering $f$ and a covering embedding $\iota\colon U\to V$ rel $K_1$. 

    \item In~\S\ref{subsec:pullback}, we show that $\iota$ is in fact an embedding on the coastal zone of $K_1$ -- an open disk bounded by the hyperbolic geodesic. Similarly, an iterate $\iota^m\colon U^m\to V$ is an embdedding in the coastal zone of $K_m$.
    \item In \S~\ref{sss:iota peripheral}, we review the notion of a pseudo-Siegel disk. It is obtained by filling in fjords... All the fillings are naturally done in coastal zones.
\end{itemize}}

\subsection{Coastal zones}\label{subsec:coastal}
Let $s \geq k$. Consider $K_{s, U^k} \subseteq U^k$.
We define 
$$
\DD^k_s = \iota^{-k}(V \setminus K_{s, V}) \subseteq U^k \setminus K_{s, U^k} \subseteq U^k \setminus K_{k, U^k}
$$
Since $\iota: U^k \setminus K_{k, U^k} \longrightarrow \iota\big( U^{k}\setminus  K_{k, U^k}\big)$ is a covering map by Item~\ref{dfn:psib-ql:2:+} in \S~\ref{subsec:iter}, and $\iota^{-1}(\iota(\DD^k_s)) = \DD^k_s$ we conclude that
$$
\iota: \DD^k_s \longrightarrow \iota(\DD^k_s) \subseteq \DD^{k-1}_s
$$
is a covering map.
We note that $\DD^0_s = V \setminus K_{s, V}$. 
By Item \ref{dfn:psib-ql:6} in Definition~\ref{dfn:psi:Siegel:ql}, we have
\begin{equation}\label{eq:D^k_k}
    \DD^k_k = U^k \setminus K_{k, U^k}.
\end{equation}

Let $\gamma^k_s$ be the hyperbolic closed geodesic of $\DD^k_s$ representing the simple closed curve around boundary of $K_{s, U^k}$.
We call the disk in $\CC^k_s \subseteq U^k$ bounded by $\gamma^k_s$ the {\em coastal zone}. We remark that $\CC^0_k \subset \CC^0_m$ if $k \leq m$.

\begin{lem}[$\iota\colon \CC^k_s \hookrightarrow \CC^{k-1}_s$] \label{lem:iota|CC_k}
For $s \geq k$, the restriction $\iota: \CC^k_s \longrightarrow \CC^{k-1}_s$ is injective.
Consequently,  $\iota^k \colon \CC^k_s \hookrightarrow \CC^0_s\subset V$ is injective.
\end{lem}
\begin{proof} 
Let $\widetilde\gamma^{k-1}_s$ be the hyperbolic closed geodesic in $\iota(\DD^k_s)$ representing the curve surrounding $K_{k, U^{k-1}} = \iota(K_k) \subseteq U^{k-1}$.
Let $\widetilde\CC^{k-1}_s$ be the disk bounded by $\widetilde\gamma^{k-1}_s$.

Since $\iota: \DD^k_s \longrightarrow \iota(\DD^k_s) = \iota(U^k) \setminus K_{s, U^k}$ is a covering map, $\iota: \gamma^k_s \longrightarrow \widetilde\gamma^{k-1}_s$ is a homeomorphism, and $\iota: \CC^{k}_s \longrightarrow \widetilde\CC^{k-1}_s$ is an conformal isomorphism.

Since $\iota(\DD^k_s) = \subseteq \DD^{k-1}_s$, and $\widetilde\CC^{k-1}_s \cap \iota(\DD^k_s) = \widetilde\CC^{k-1}_s \cap \DD^{k-1}_s$, we conclude that $\widetilde\CC^{k-1}_s \subseteq \CC^{k-1}_s$.
Therefore, $\iota: \CC^k_s \longrightarrow \CC^{k-1}_s$ is injective.
\end{proof}

\subsection{Families of curves}\label{subsec:familyofarcs} Recall that degeneration of Riemann surfaces is encoded and analyzed by the extremal width of various families. Following~\cite{DL22}, we denote by $\Fam$ and $\Fam^+$ the full (i.e., in $V$) and outer (i.e., in $V\setminus \overline Z$ or in $V\setminus \widehat Z^m$) families of curves emerging from certain intervals, see~\S\ref{sss:full_outer} for details. See~\S\ref{sss:occ:Fam} and~\S\ref{sss:occ:Fam^+} for a short discussion on how these families are used.

Unlike the case of quadratic polynomials, the outer family $\Fam^+$ is further decomposed into the vertical $\Fam^{+,\ver}$ and peripheral $\Fam^{+,\per}$ parts, see~\S\ref{sss:Fam:VertPer}. Finally, in \S\ref{subsubsec:externaldiving}, we decompose the peripheral part  $\Fam^{+,\per}$ rel $\filled_{m}$ into the external and diving components. The last decomposition is similar to \cite[\S~2.3.2]{DL22} with the understanding that curves in $\Fam^{+,\per}$ that lift to vertical under $\iota^{\qq_{m+1}}$ are in $\Fam^{+,\per}_{\div,m}$.

All together, we have
\begin{multline}
\label{eq:Fam^+}
    \qquad\qquad\Fam^+ = \Fam^{+,\ver}+\Fam^{+, \per} - O(1)\\ = \Fam^{+,\ver}+\Fam^{+,\per}_{\ext, m}+\Fam^{+,\per}_{\div, m} - O(1).\qquad
\end{multline}
Decomposition~\eqref{eq:Fam^+} is only relevant for parabolic fjords; see~\S\ref{sss:occ:Fam^+}.

\begin{figure}
   \includegraphics[width=10cm]{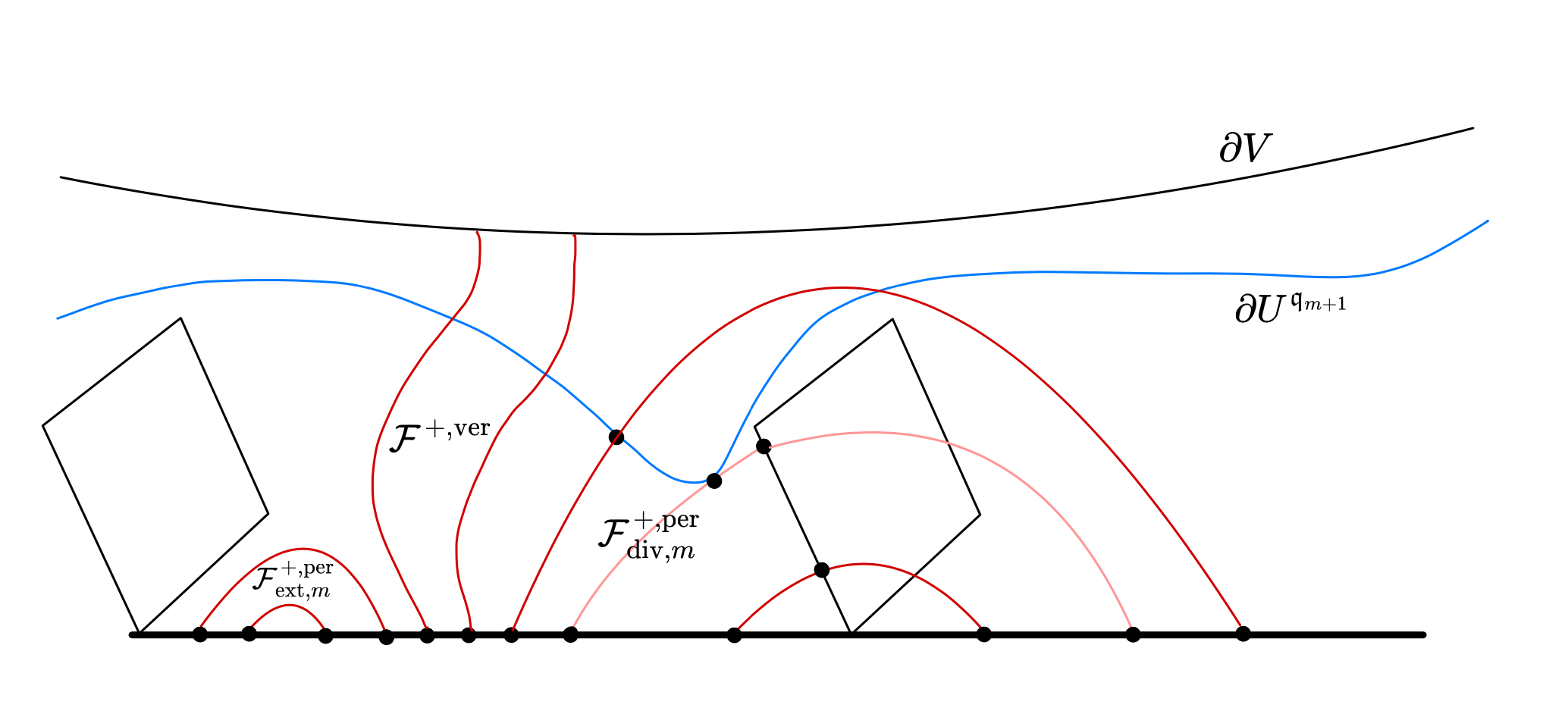}
   \caption{An illustration of the vertical and peripheral families of curves. The peripheral families are further decompose into external and diving components. The diving family may intersect either $\partial U^{\qq_{m+1}}$ or $\mathcal{K}_m$.}
   \label{Fig:SplitRec}
\end{figure}

\subsubsection{Full and outer families} \label{sss:full_outer}
Following the notations in \cite[\S~2.3]{DL22}, we consider the following families.
Let $I$ be an interval in $\partial Z$. Let $\lambda > 1$, and $\lambda I$ be the enlargement of $I$ by attaching two intervals of length $\lambda |I|$ on either side of $I$. 
We write
$$
\Fam^{+}_\lambda(I) = \Fam^{+}_{\lambda, Z}(I) := \{\gamma \subset V \setminus \overline{Z}: \gamma \text{ connects $I$ and $\partial V \cup (\partial Z \setminus \lambda I)$}\}.
$$
Similarly, we write 
$$
\Fam_\lambda(I) = \Fam_{\lambda, Z}(I)
:= \{\gamma \subset V \setminus (I \cup (\partial Z \setminus \lambda I)): \gamma \text{ connects $I$ and $\partial V \cup (\partial Z \setminus  \lambda I)$}\}.
$$
Their widths are denoted by $\Width^+_{\lambda}(I)$ and $\Width_{\lambda}(I)$ respectively.

Let $I$ be an interval in $\partial Z$ and $J$ be an interval in $\partial Z \cup \partial V$. We also write
$$
\Fam^{+}(I, J):= \{\gamma \subset V \setminus \overline{Z}: \gamma \text{ connects $I$ and $J$}\},
$$
and its width by $\Width^{+}(I, J)$.

\subsubsection{Appearance of full families $\Fam$}\label{sss:occ:Fam} The full families $\Fam$ are the main objects in the proof of Amplification Theorem~\ref{thm:SpreadingAround}. There, a wide outer family $\Fam^+$ is spread around in Lemma~\ref{lem:SpredAroundWidth} (an application of the Covering Lemma); the result is a set of wide full families based on the intervals $I^m_s$; see Items~\ref{item:1:lem:SpredAroundWidth} and~\ref{item:2:lem:SpredAroundWidth} of Lemma~\ref{lem:SpredAroundWidth}. In Snake Lair Lemma~\ref{lem:Hive Lemma}, this set of wide families is traded back into an outer wide family with amplification. 

It is important that the argument of Theorem~\ref{thm:SpreadingAround} is applied to pseudo-Siegel disks $\wZ^m$ as the boundaries of non-regularized Siegel disks $\overline Z$ can have uncontrollable oscillations (Cremer phenomenon in the limit). In general, $\wZ^m$ are not $\iota$-proper under $\iota^k\colon U^k\to V$; this subtlety is resolved in Lemma~\ref{lem:liftingviaiota} (see also discussion before the lemma).

\subsubsection{Vertical and peripheral families}\label{sss:Fam:VertPer}
Let $\gamma$ be an proper arc $ V - \overline{Z}$.
We say it is {\em vertical} if $\gamma$ connects a point in $\partial Z$ and a point in $\partial V$. It is {\em peripheral} if $\gamma$ connects two points in $\partial Z$.

The family $\Fam^+_\lambda(I)$ is naturally decomposed into vertical and peripheral subfamilies, and we denote them by
\begin{align*}
\Fam^{+, \ver}(I) = \Fam^{+, \ver}_{Z}(I) &:= \{\gamma \in \Fam^{+}_\lambda(I): \gamma \text{ connects $I$ and $\partial V$}\}, \text{ and }\\
\Fam^{+, \per}_{\lambda}(I) = \Fam^{+, \per}_{\lambda, Z}(I)  &:= \{\gamma \in \Fam^{+}_\lambda(I): \gamma \text{ connects $I$ and $\partial Z \setminus \lambda I$}\}.
\end{align*}
Their extremal widths are denoted by
$$
\Width^{+, \ver}(I)  \text{ and } \Width^{+, \per}_{\lambda}(I^m)  \text{ respectively}.
$$
Note that we have
$$
\Fam^{+, \ver}(I) = \Fam^{+}(I, \partial V) \text{ and } \Fam^{+, \per}_{\lambda}(I) = \Fam^+(I, \partial Z\setminus \lambda I).
$$

The subfamilies $\Fam^{\ver}(I), \Fam^{\per}_{\lambda}(I)$ of $\Fam_\lambda(I)$ and their widths $\Width^{\ver}(I), \Width^{\per}_{\lambda}(I)$ are defined similarly.

\subsubsection{External and Diving families}\label{subsubsec:externaldiving}
Let $I \subset \partial Z$ and $J \subset \partial V \cup \partial Z$ be two intervals.
Fix a level $m$, the family $\Fam^+(I,J)$ can be naturally decomposed into two subfamilies:
\begin{itemize}
    \item the {\em external family} relative to level $m$, denoted by $\Fam^+_{\ext, m}(I,J)$: the set of arcs in $\Fam^+(I,J)$ that can be lifted under $\iota^{\qq_{m+1}}$ to arcs in $U^{\qq_{m+1}}\setminus \filled_m$ still connecting $I$ and $J$ (see \S~\ref{subsec:pullback}); 
    \item the diving family $\Fam^+_{\div, m}(I,J) = \Fam^+(I,J)\setminus \Fam^+_{\ext, m}(I,J)$ otherwise.
\end{itemize}
Their widths are denoted by $\Width^+_{\ext,m}(I,J), \Width^+_{\div, m}$ respectively.
We first remark that 
\[ \Width^+(I,J)=\Width^+_{\ext,m}(I,J) + \Width^+_{\div, m}(I,J)+O(1).\]
We also use the notation $\Fam^{+, \per}_{\lambda, \ext, m}(I), \Fam^{+, \per}_{\lambda, \div, m}(I)$ for $\Fam^+_{\ext, m}(I,J), \Fam^+_{\div, m}(I,J)$ where $J = \partial Z \setminus \lambda I$.

Note that we have
\begin{equation}\label{eqn:decomposition}
    \Width^+_\lambda(I)=\Width^{+,\per}_{\lambda, \ext,m}(I) + \Width^{+,\per}_{\lambda, \div,m}(I)+\Width^{+,\ver}(I) + O(1)
\end{equation}

\subsubsection{Appearance of outer families $\Fam^+$}\label{sss:occ:Fam^+} Outer families $\Fam^+$ are the main objects in the analysis of geometries in fjords in Section~\ref{sec:parabolicfjords} and~\ref{sec:calibration}. Roughly, the external component $\Fam^{+,\per}_{\ext, m}$ is handled by Theorem~\ref{thm:parabolicfjords} (Log-Rule) while the diving component $\Fam^{+,\per}_{\div, m}$ is handled by Calibration Lemma~\ref{lmm:CalibrLmm}. A ``global external input'' is required to genuinely handle the vertical component $\Fam^{+,\ver}$; see~\S\ref{ss:comblocal} for details.

The theory of snakes~\cite[\S6]{DL22} allows us to convert a full into an outer family. Such a convergence should be executed rel $\wZ^m$ as the boundaries of non-regularized Siegel disks $\overline Z$ can have uncontrollable oscillations (Cremer phenomenon in the limit). It is important that outer families based on grounded intervals can be easily converted between $\overline Z$ and $\wZ^m$, see~\S\ref{sss:Outer:wZ}.

\subsubsection{Family and width in $U^k$} \label{sss:famU^k} 
Let $k \leq 10\qq_{m+1}$. We denote $Z_{U^k} \subset U^k$ be the Siegel disk in $U^k$, i.e., the subset so that $\iota^k \mid Z_{U^k}: Z_{U^k} \longrightarrow Z$ is a homeomorphism.
Let $I \subset Z_{U^k}$. The definition of various families of arcs and their width naturally extends in this case. We add $U^k$ in the subscript to denote the objects in $U^k$.

For example, we use the notation $\Fam^{+}_{\lambda, U^k}(I) = \Fam^{+}_{\lambda, Z_{U^k}, U^k}(I)$ to denote the family of arcs $\gamma \subset U^k \setminus Z_{U^k}$ connecting $I$ and $\partial U^k \cup (\partial Z_{U^k} \setminus \lambda I)$.

\subsubsection{External family in $V$ vs $U^{\qq_{m+1}}$} \label{subsubsec:efVvsU}
By definition, $\iota^{\qq_{m+1}}$ lifts univalently the family $\Fam^+_{\ext, m}(I,J)$ into $\Fam^+_{U^{\qq_{m+1}}\setminus \filled_m}(I,J)$ -- the family of curves in $U^{\qq_{m+1}}\setminus \filled_m$ connecting $I$ and $J$. 

Since any limb of $\filled_m\setminus \overline Z$ contains a fundamental interval of the abelian covering \[f^{\qq_{m+1}}\colon U^{\qq_{m+1}}\setminus \filled_m \to V\setminus \overline Z,\]
wide families can not ``bypass'' any limb. 
Therefore, the following families in $ U^{\qq_{m+1}}\setminus \filled_m $ have width $O(1)$:
\begin{itemize}
    \item connecting two different interval $T_1,T_2\in \Dbb_m$;
    \item any winding family from $T\in \Dbb_m$ to itself.
\end{itemize}

Therefore, the relevant case is when $I, J \subseteq T = [v,w]\in \Dbb_m$. By replacing the family by vertical foliations of geodesic rectangles $\RR_m(I,J)\subset U^{\qq_{m+1}}\setminus \filled_m$ connecting $I, J$, we have:
\[\Fam^+_{U^{\qq_{m+1}}\setminus \filled_m}(I,J)- O(1) \subseteq \CC^{\qq_{m+1}}_{\qq_{m+1}},\]
where $\CC^{\qq_{m+1}}_{\qq_{m+1}}$ is the coastal zone (see \S~\ref{subsec:pullback}).
By Lemma \ref{lem:iota|CC_k}, $\iota^{\qq_{m+1}}$ is injective on $\CC^{\qq_{m+1}}_{\qq_{m+1}}$. Therefore, $\iota^{\qq_{m+1}}$ maps conformally the geodesic rectangle $\RR_m(I,J)$ to its image, and we have
$$
\Width^+_{U^{\qq_{m+1}}\setminus \filled_m}(I,J) = \Width^+_{\ext, m}(I,J) + O(1).
$$

\subsection{Bistable regions $\boldsymbol {\Xi^k_m\supset \overline Z}$}\label{subsec:pullback} In this subsection, we will employ properties of coastal zones to justify univalent iterates near $\overline Z$. 

In~\S\ref{sss:F^*}, we will adapt the discussion of fjords from~\cite[\S2.1.7]{DL22} to the $\psi^\bullet$-setting.  Lemma~\ref{cor:pullback} describes pullbacks within fjords. In~\ref{sss:pullbackrect}, we discuss pullbacks of rectangles under $f^k: U^k\setminus K_{k, U^k} \longrightarrow V\setminus \overline{Z}$ assuming that rectangles are within coastal zones.

In~\S\ref{sss:F:stab}, we define the $k$th stable part of a fjord under pullbacks. In~\S\ref{sss:F:bistab}, we define the $k$th bistable part of a fjord as the $k$th pullback of the $2k$ stable part under the first return map. Bistable regions $\Xi_m^k \supset \overline Z$ are defined in \S\ref{sss:biXi}.

\subsubsection{Fjord for $\psi^\bullet$-maps and pullbacks $F^*=\iota_*\ \circ f^*$}\label{sss:F^*} Below, we adopt the discussion of~\cite[\S2.1.7]{DL22}. As in~\cite{DL22}, we write 
\[\filled_m\equiv K_{\qq_{m+1}}=f^{-\qq_{m+1}}(\overline Z).\]

Consider an interval $T=[v,w]\in \Dbb_m, \sp m\ge -1$ in the diffeo-tiling \cite[\S2.1.6]{DL22}. We also consider intervals $[v',w']=f^{\qq_{m+1}}(T)$ and $T'=[v',w]$ as in~\cite[\S2.1.7]{DL22}.
We use the same standing assumption that $f^{\qq_{m+1}}|_T$ moves points clockwise towards $w$.

For an interval $J\subset T$, the level $m$-fjord $\Fjord_m(J) = \Fjord_{m, V}(J) \subset V$ on $J$ is the closed disk bounded by $J$ and the by the hyperbolic geodesic $\gamma^m_J$ of $V\setminus\filled_m$ with end points $\partial J$ such that $\gamma^m_J$ is homotopic to $J$ in $V\setminus\filled_m$.

Similarly, for an interval $S\subset [v',w']$, the fjords $\Fjord(S) = \Fjord_V(S)\subset V$ are defined as the closed disk bounded by $S$ and the by the hyperbolic geodesic $\gamma_S$ of $V\setminus\overline{Z}$ with end points $\partial S$ such that $\gamma_S$ is homotopic to $S$ in $V\setminus\overline{Z}$.

We note that $\Fjord_m(J)$ are contained in the coastal zone $\CC^0_m$. Thus, the following lemma follows immediately from Lemma \ref{lem:iota|CC_k}.

\begin{lem}[{\cite[Lemma~2.4]{DL22} for $\psi^\bullet$-ql maps}]\label{cor:pullback}
Using the above notations,\\ $F^{-\qq_{m+1}}_T=\iota^{\qq_{m+1}}\circ [f\mid T]^{-\qq_{m+1}}\colon [v',w']\to T$ extends to an injective branch
\begin{equation}
\label{eq:0:cor:pullbacks for psi bullet}
F^{-\qq_{m+1}}_T\colon\ \Fjord([v',w'])\ \hookrightarrow \ \Fjord_m(T). 
\end{equation}

If $J\subset T'$ is a subinterval such that \[F_T^{-t\qq_{m+1}}(J)\subset T'\sp\text{ for all }\ t\in \{ 0,1,\dots, s-1\},\sp \sp \text{ then: }\] 
\begin{equation}
\label{eq:cor:pullbacks for psi bullet}
  F_T^{- s\qq_{m+1}} \big(\Fjord(J)\big)\subset\Fjord_m(T);\end{equation}
i.e., $\Fjord(J)$ can be lifted under $F^{ s\qq_{m+1}}$.
\end{lem}

\subsubsection{Pullbacks of rectangles under $f^k: U^k\setminus K_{k, U^k} \longrightarrow V\setminus \overline{Z}$}\label{sss:pullbackrect}
Let $T = [v, w]\in \Dbb_m$ be an interval in the diffeo-tiling, and $T' = f^{\qq_{m+1}}(T)$.
Let $I \subseteq T'$ and $J \subseteq \partial Z$ be an interval disjoint from $I$.
Let $\RR \subseteq V\setminus Z$ be a rectangle with horizontal sides $I, J$.
Let $\RR_{\text{geod}}(I,J)$ be the geodesic rectangle between $I$ and $J$, i.e., whose vertical sides are geodesics in $V \setminus \overline{Z}$ connecting the corresponding endpoints $\partial I, \partial J$.
Up to $O(1)$ of width, we may assume that $\RR$ is contained in $\RR_{\text{geod}}(I,J)$.

Fix an iterate $f^k$ with $k \leq \qq_{m+1}$.
Let $I_{-k} \subseteq \overline{Z}$ so that $f^k(I_{-k}) = I$.
Since $f^k: U^k\setminus K_{k, U^k} \longrightarrow V\setminus \overline{Z}$ is a covering, let $\RR_{-k}$ be the lift of the rectangle $\RR$ in $U^k\setminus K_{k, U^k}$ that emerges from $I_{-k}$.
Note $\RR_{-k}$ is contained in the geodesic rectangle of $\DD^k_k = U^k\setminus K_{k, U^k}$ between $\partial^{h,0}\RR_{-k}, \partial^{h,1}\RR_{-k}$.
Thus, $\RR_{-k}$ is contained in the coastal zone $\CC^k_k$.
By Lemma~\ref{lem:iota|CC_k}, $\iota^k:\CC^k_k \longrightarrow \CC^0_k$ is conformal.
Let $(F^k)^*(\RR) :=\iota^k(\RR_{-k}) \subset V\setminus K_{k,V}$, and we call $(F^k)^*(\RR)$ the pullback of $\RR$ under $F^k$.

\subsubsection{Stable parts $\mathfrak F^k_\stab(T)$ of fjords $\mathfrak F(T)$} \label{sss:F:stab} As before, we assume that $f^{\qq_{m+1}}|_T$ moves points clockwise towards $w$, where $T=[v,w]\in \Dbb_m, \sp m\ge -1$.

Assume that $\length_{m}\ge (k+2)\length_{m+1}$. Set $v_k=v\boxplus k\length_{m+1}=f^{k\qq_{m+1}}(v)$ and observe that $[v_k,w]\subset T$. We set 
\[ \mathfrak F_{\stab}^k(T)\coloneqq  \mathfrak F([v_k,w]) \ \subset \ \mathfrak F(T)\]
and call it \emph{the $k$-stable part} of the fjord $\mathfrak F(T)$.

If $\length_{m}< (k+2)\length_{m+1}$, then we set $\mathfrak F_{\stab}^k(T) \coloneqq \emptyset$. By construction,
\begin{multline*}\qquad \mathfrak F(T) \equiv\mathfrak F_{\stab}^0(T)\supset \mathfrak F_{\stab}^1(T)\supset \mathfrak F_{\stab}^2(T)\supset \dots \supset \mathfrak F_{\stab}^t(T)=\emptyset \\ \text{if }t> \frac{\length_m}{\length_{m+1}}-2. \quad
\end{multline*}

Let us recall that for every $i\le \qq_{m+1}$ there is an interval $T_{-i}\in \Dbb_m$ such that $[f\mid \partial Z]^{-i}$ maps $T$ almost into $T_{-i}$. More precisely, write $T_{-i}=[a,b]$, where $a$ is on the left of $b$ so that $f^{\qq_{m+1}}\mid _{T_{-i}}$ moves points clockwise towards $b$. Write also $[\tilde v,\tilde w]=[f\mid \partial Z]^{-i}(T)$. Then 
\begin{itemize}
    \item either $\tilde v=a$ or $\tilde v = a\boxminus \length_{m+1}$;
    \item either $\tilde w=b$ or $\tilde w = b\boxminus \length_{m+1}$.
\end{itemize}
Similarly, writing $\tilde v_k=[f\mid \partial Z]^{-i}(v_k)$, we have
\begin{itemize}
  \item either $\tilde v_k=a_k$ or $\tilde v_k = a_k\boxminus \length_{m+1}=a_{k-1}$.
\end{itemize}

Set 
\[\widetilde {\mathfrak F}^k_\stab (T)\coloneqq F_T^{-i}\Big(\mathfrak F^k_\stab (T)\Big),\qquad\qquad F_T^{-i}\equiv \iota^{i}\circ f_T^{-i} ,\]
where the branch of $F_T^{-1}$ is taken along $f^{-i}( T)\approx T_{-i}$. It follows that $\widetilde {\mathfrak F}^k_\stab (T)$ is bounded by the hyperbolic geodesic of $\iota^i\big(U^i\setminus K_i\big)$ and $[\tilde v_k, \tilde w]$. Since such a hyperbolic geodesic of $\iota^i\big(U^i\setminus K_i\big)$ is below the assoicated hyperbolic geodesic of $V\setminus \overline Z$, we obtain that

\begin{equation}\label{eq:1:sss:F:stab}
\widetilde {\mathfrak F}^k_\stab (T) \subset \mathfrak F ([\tilde v_k,\tilde w_k])\subset  {\mathfrak F}^{k-1}_\stab (T_{-i}).    
\end{equation}

Inclusion~\eqref{eq:1:sss:F:stab} can be generalized as follows. Consider $i$
 with \[(j-1)\qq_{m+1}< i\le j \qq_{m+1},\qquad \text{ where }\quad 0< j\le k.\]
Then 
\begin{equation}\label{eq:2:sss:F:stab}
 F^{-i}_T \Big(\mathfrak F^k_\stab(T)\Big)\ \subset \  \mathfrak F^{k-j}_\stab(T_{-i}),\qquad T_{-i} \approx (f\mid_{\partial Z}) ^{-i}(T),
\end{equation}
where $ F^{-i}_T$ is a branch extending $T_{-i} \approx (f\mid_{\partial Z}) ^{-i}(T)$.

\subsubsection{Bistable parts $\Xi_T^k$ of fjords $\mathfrak F(T)$} \label{sss:F:bistab}Using~\eqref{eq:2:sss:F:stab}, we define the $k$th bistable part of a fjord $ \mathfrak (T)$ as
\[ \Xi^k_T\coloneqq F^{-k\qq_{m+1}}_T \Big( \mathfrak F_{\stab}^{2k}(T)\Big).\]
Clearly, 
\begin{equation}
    \label{eq:0:sss:F:bistab} \Xi^k_T\subset \mathfrak F^{k}_\stab (T)\subset \mathfrak F(T).
\end{equation}

Note that $ \Xi^k_T$ is empty if and only if $\length_{m}<(2k+2)\length_{m+1}$. It follows from~\eqref{eq:2:sss:F:stab} that for every $i\in \Z$ with 
\[(j-1)\qq_{m+1}< |i|\le j \qq_{m+1},\qquad \text{ where }\quad 0< j\le k,\]
we have 
\begin{equation}\label{eq:1:sss:F:bistab}
 F^{i}_T \Big(\Xi^k_T\Big)\ \subset \  \Xi^{k-j}_{T_i},\qquad T_{i} \approx (f\mid_{\partial Z}) ^{i}(T),
\end{equation}
where $ F^{i}_T$ is a branch extending $T_{i} \approx (f\mid_{\partial Z}) ^{i}(T)$.

The stable part  $\Xi_T^k$ can also be explicitly defined as follows. In addition to $v_k$ specified in \S\ref{sss:F:stab}, set $w_k \coloneqq w\boxminus k \length_{m+1}$ so that $f^{\qq_{m+1}k}(w_k)=w$. Observe that $v< v_k<w_k<w$. Let $\gamma$ be the hyperbolic geodesic of $\iota^{k\qq_{m+1}}( U^{k\qq_{m+1}}\setminus K_{k\qq_{m+1}})$ connecting $v_k$ and $w_k$ that is homotopic to $[v_k,w_k]$ in that space. Then $\Xi^k_T$ is the component bounded by $\gamma \cup [v_k,w_k]$.

\subsubsection{Bistable regions $\Xi_m^k \supset \overline Z$}\label{sss:biXi} Given $m\ge -1$, we denote the {\em $k$-th bistable region} by
$$
\Xi^k_m:= \overline Z \cup \bigcup_{n \geq m}\bigcup_{I \in \Dbb_n} \Xi^k_I,\qquad\qquad \Xi^1_m\equiv \Xi_m.
$$

Equation~\eqref{eq:1:sss:F:bistab} implies the following. For all $i$ with \[(j-1)\qq_{m+1}< |i|\le j \qq_{m+1},\qquad \text{ where }\quad 0< j\le k,\] we have the following induced map along $F^{i}\colon \overline Z \ \overset {1:1}{\longrightarrow}\   \overline Z$:
\begin{equation}
    \label{eq:F_biXi} F^{i}_{\Xi_m^k}\equiv \big(f \circ \iota^{-1}\big)^{i}\ \colon \Xi^k_{m} \ \overset {1:1}{\longrightarrow}\  F^{i}_{\Xi_m^k}(\Xi_m^k) \ \subset \ \Xi_m^{k-j}. 
\end{equation}

\begin{comment}

\subsubsection{TO be integrated? or DELETE}

The following lemma shows if we have some outer protection, then the geodesic auxiliary objects in $V\setminus \overline{Z}$ will be contained in $\Xi_m$.
\begin{lem}\label{lem:geodesicpseudoSiegel}
    Let $T = [v, w]\in \Dbb_m$, and let $I, J \subset T$ disjoint intervals.
    Let $L$ be the component of $\partial Z\setminus(I\cup J)$ that is contained in $T$.
    Suppose there exists a rectangle $\RR\subset \Xi_m$ connecting $I, J$ with $\Width(\RR) \geq 10$.
    Then $\Fjord(L) \subset \Xi_m$.
\end{lem}
\begin{proof}
    Divide the rectangle $\RR$ into two rectangles $\RR^i, i = 1, 2$, each with width $\geq 5$, and horizontal boundary $I^i, J^i$. Suppose $\RR^2$ is nested inside of $\RR^1$. Let $\RR(I^i, J^i)$ be the geodesic rectangle in $V\setminus\overline{Z}$ connecting $I^i$ and $J^i$. Then $\Width(\RR^i) \geq \Width(\RR^i) - 1 \geq 4$ (see \cite[Lemma A.5]{DL22}). Suppose $\Fjord(L) \not\subset \Xi_m$. Then $\RR^1$ cuts through $\RR(I^2, J^2)$. This is not possible as wide rectangles cannot cross.
\end{proof}
\end{comment}

\subsection{Transformation rules and univalent push-forward}\label{ss:push forw curves} In this subsection, we state several rules for pushing-forward degeneration along 
\begin{equation}
\label{eq:transformation_pattern}
     (F^k)_*=(f^k)_*\circ (\iota^k)^*\ \colon \ (V,\overline Z)\ \leadsto\ (U^k,\overline Z)\ \leadsto \ (V,\overline Z).
\end{equation}

Univalent pushforward \S\ref{subsubsec:univalentpush} deals with degenerations outside of the Siegel disk; such a pushforward is applied to laminations in (or emerging from) parabolic fjords.

A simple transformation rule \S\ref{subsubsec:simplepush} is obtained by selecting a cutting arc $\gamma$ that connects $\partial Z$ and $\partial V$ and pulling back the associated open disk along $f^k$ to obtain a finite-degree restriction for~\eqref{eq:transformation_pattern}, see Figure~\ref{Fig:psi:PushForward}. This is also a recipe for a near-degenerate transformation rule relying on the Covering Lemma, see~\S\ref{subsubsec:near-degeneratepush}. (In the case of quadratic polynomials, $\gamma$ is selected to connect $\wZ^m$ and $\infty$, see \cite[\S 8.1.2]{DL22}.)

The simple and near-degenerate transformation rules deal with full degenerations and should be applied for $(V, \wZ^m)$ instead of $(V, \overline Z)$. An appropriate modification is presented in~\S\ref{sss:CL:wZ^m}; see~\eqref{eq:transformation_pattern:wZ} and compare it with~\eqref{eq:transformation_pattern}

\subsubsection{Univalent push-forward}\label{subsubsec:univalentpush}
We now explain the construction of univalent push-forward of a rectangle. We refer the readers to \cite[\S 2.3.3]{DL22} for the quadratic polynomial setting. 
Let $T=[v,w]\in \Dbb_m, \sp m\ge -1$ in the diffeo-tiling \cite[\S2.1.6]{DL22}.
Consider a rectangle $\RR \subseteq V \setminus Z$ with $\partial^{h,0} \RR \subset T$ and $\partial^{h,1} \RR \subseteq \partial Z \cup \partial V$.
Let us fix an iterate $f^i$ with $i \leq \qq_{m+1}$.
We will now describe the univalent push-forward under $f^i$.

Let $\FF(\RR)$ be the lamination of the vertical arcs in $\RR$. Let $\FF_{U^i}$ be the lift of $\FF(\RR)$ under $\iota^i$ that emerges from $\overline{Z}$
Let $\FF_{U^i}'$ be the restriction of $\FF_{U^i}$ in $U^{i} - K_{i, U^i}$: the family $\FF_{U^i}'$ consists of subarcs $\gamma'$ of $\gamma\in \FF_{U^i}$ such that every $\gamma'$ is the shortest subarc of $\gamma$ connecting $\partial^{h, 0}\RR$ to $\partial K_{i, U^i} \cup \partial U^i$. Note that arcs in $\RR_{U^i}'$ form three different types: left, right or vertical. Thus, we obtain at most three rectangles by considering the region bounded by the leftmost and rightmost curves (see \cite[A.1.8]{DL22}). Denote them by $\RR_{l, U^i}, \RR_{r, U^i}, \RR_{\ver, U^i}$.
Let $\widetilde\RR_{l, U^i}, \widetilde\RR_{r, U^i}, \widetilde\RR_{\ver, U^i}$ be the sub-rectangles obtained from $\RR_{l, U^i}, \RR_{r, U^i}, \RR_{\ver, U^i}$ by removing the $1$-buffers on every side.
Let $\widetilde\RR_{U^i} = \widetilde\RR_{l, U^i} \cup \widetilde\RR_{r, U^i}\cup\widetilde\RR_{\ver, U^i}$.
Note that $\Width(\widetilde\RR_{U^i}) \geq \Width(\RR) - 12$.
Then by \cite[Lemma~A.8]{DL22}, we have that the map
\begin{equation}\label{eq:univalentpushforward}
    f^i: \widetilde\RR_{U^i} \longrightarrow f^i(\widetilde\RR_{U^i}) \subset V \setminus Z 
\end{equation}
is injective. 

We refer to \eqref{eq:univalentpushforward} the univalent push-forward of $\RR$. We also denote the image by $F^i_*(\RR) = f^i_*\circ(\iota^k)^*(\RR) = f^i(\widetilde\RR_{U^i})$.

\subsubsection{Simple transformation rule of $\Width(\LL)$}\label{subsubsec:simplepush}
We now discuss a different construction suitable for the application of the covering lemma. 

Let $I \subset \partial Z$, and  $\tau > 1$.
Let $\LL$ be a lamination consisting of arcs in $V \setminus (I \cup (\partial Z \setminus \tau I))$ connecting $I$ to $\partial V \cup (\partial Z \setminus \tau I)$. 
We will now describe the simple transfer of the width family $\LL$ under $f^k$.

Let $I_k = (F|_{\partial Z})^k(I)$. 
Let $\gamma_k \subseteq V$ be a curve connecting $\partial Z \setminus \tau I_k$ to $\partial V$ (see Figure~\ref{Fig:psi:PushForward}).
Denote $Y := V \setminus (\gamma_k \cup (\partial Z \setminus \tau I_k))$. Note that $Y$ is a disk.
Let $X$ be the component of $f^{-k}(Y) \subseteq U^k$ that contains $I$, and let $\iota^k(X)$. Let $\LL'$ be the restriction in $\iota^k(X)$ of the family $\LL$ emerging from $I$. Then we can lift the lamination $\LL'$ via $\iota^k$ and obtain a lamination $\widetilde{\LL}$ in $X$ connecting $I$ to $\partial X$.
Note that $\Width(\widetilde{\LL}) \geq \Width(\LL)$. Thus, this gives a lower bound for the modulus of the annulus $X\setminus I \subseteq U^k$
\begin{equation}\label{eqn:annulusbound}
    \Width(X \setminus I) \geq \Width(\LL).
\end{equation}

Now suppose that 
\begin{enumerate}
    \item\label{cond:gamma_k} the width of the family of arcs connecting $I_k$ to $\gamma_k$ in $V\setminus(I_k \cup \gamma_k \cup (\partial Z \setminus \tau I_k))$ is $O(1)$, which can be achieved by choosing $\gamma_k$ carefully under mild modifications, see \cite[\S 8.1.5]{DL22}; in particular, such a $\gamma_k$ can be selected if $f(\tau I_k)$ is disjoint from $\tau I_k$;
    \item \label{cond:degreebound} the degree of the branched covering $f^k: X \longrightarrow Y$ is $D$, which can be bounded in terms of $\tau$ (see \cite[\S 8.1.2]{DL22})
\end{enumerate}

\begin{rem}\label{rem:gamma_k} The disjointness assumption between $f(\tau I_k)$ and $\tau I_k$ in~\eqref{cond:gamma_k} leads to the combinatorial assumption ``$|I|\le |\theta_0|/(2\blambda_\bbt)$'' in Amplification Theorem~\ref{thm:SpreadingAround}, where $\tau =\blambda_k$ is big.

In the proof of Lemma~\ref{lem:replication}, we will apply the simple transformation rule with $\tau=3$ (after projecting the degeneration to $\wZ^m$). There, the disjointness assumption can be relaxed into the disjointness between $f^i(\tau I_k)$ and $\tau I_k$ for $i\le 7$. If there are less than $7$ combinatorial intervals, then the selection of $\gamma_k$ is not required.
\end{rem}

Let $F_\tau(I_k)$ be the family of arcs in $V \setminus (I_k \cup (\partial Z \setminus \tau I_k))$ connecting $I_k$ to $\partial V \cup (\partial Z \setminus \tau I_k)$.
Then we get
$$
\Width_\tau(I_k):= \Width(F_\tau(I_k)) \geq \Width(Y\setminus I_k) - O(1) \geq \frac{\Width(X \setminus I)}{D} - O(1) \geq \frac{\Width(\LL)}{D} - O(1).
$$
In particular, if $\Width(\LL)$ is big, then we also obtain a large lamination $\LL_k$ of $F_\tau(I_k)$ whose width is at least by $\frac{\Width(\LL)}{D} - O(1)$.
We call such a construction the {\em simple transformation rule} of $\Width(\LL)$.

\subsubsection{Near-degenerate transformation rule of $\Width(\LL)$}\label{subsubsec:near-degeneratepush}
Let us assume the setup of \S~\ref{subsubsec:simplepush}.
Suppose $\tau \gg 1$, Then the degree $D$ in Condition~\ref{cond:degreebound} is also large. In this setting, we use the following modification.  We define $B:= V \setminus (\gamma_k \cup (\partial Z \setminus 10 I_k))$, and let $A$ be the component of $f^{-k}(B)$ that contains $I$.
Then $f^k:(X, A, I) \longrightarrow (Y, B, I_k)$ is a branched covering between nested Jordan disks.
This sets up the application of the covering lemma (see \cite[\S 8.1, Lemma 8.5]{DL22} and \cite{KL} for more details). 
We call such a construction that involves the covering lemma the {\em near-degenerate transformation rule} of $\Width(\LL)$.

\begin{figure}
   \includegraphics[width=13cm]{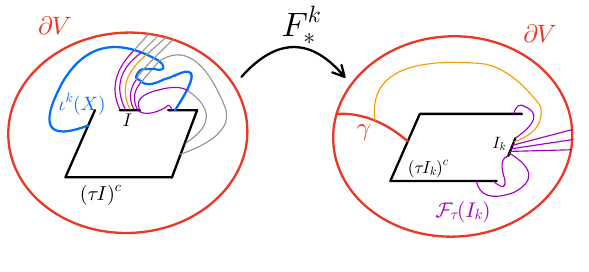}
   \caption{Pushforward of $\Fam_\tau(I)$ under the branched covering. Observe that $Y\coloneqq V\setminus \big[\gamma \cup (\tau I_k)^c \big]$ is a disk.}
   \label{Fig:psi:PushForward}
\end{figure}

\subsection{Pseudo-Siegel disks}\label{sss:iota peripheral}
The definition of level $m$ pseudo-Sigel disk $\wZ^m$ for a $\psi^\bullet$-ql Siegel map is the same as for quadratic polynomials (see~\cite[Definition 5.1]{DL22}) with the understanding that all auxiliary objects at level $m$ (such as {\em channels, dams, collars} and {\em protections}) defined in~\cite[\S~5.1]{DL22} (see Figure \ref{Fig:Welding}, or \cite[Figure~16]{DL22}) are contained in the bistable region $\Xi_m\equiv \Xi^1_m$ sepcified in~\S\ref{sss:biXi}.

Given a peripheral curve $\beta \subset V\setminus Z$ with endpoints in $I$, where $I\in \mathcal D_m$, set $\widehat{\mathfrak F}_\beta$ to be the closure of the connected component of $V\setminus (Z\cup \beta)$ enclosed by $\beta \cup I$. If $\widehat{\mathfrak F}_\beta$ is peripheral rel $\overline Z$, then we call $\widehat{\mathfrak F}_\beta$ the \emph{parabolic fjord} bounded by $\beta$; see Figure \ref{Fig:Welding}. We will refer to $\beta$ as the \emph{dam} of $\widehat{\mathfrak F}_\beta$. If $\beta$ is a hyperbolic geodesic of $V\setminus \overline Z$, then we say that the fjord $\widehat{\mathfrak F}_\beta$ is \emph{geodesic}.

Let us state a definition of pseudo-Siegel disks with slightly simplified notations.  In particular, outer protections $\XX_I$ are considered below rel $\overline Z$ instead of $\widehat Z^{m+1}$ as it was in \cite[Assumption 6]{DL22}. (Protecting rectangle rel $\overline Z$ can be easily projected into a protecting rectangle rel $\wZ^{m+1}$ and vise-versa, see \cite[ Lemma 5.9]{DL22}.) 
\begin{defn}
\label{defn:pseudo-Siegel disks} A \emph{pseudo-Siegel disk} 
$\widehat Z^m \subseteq \Xi_m \subset V$ of $m\ge -1$ and its \emph{territory} $\wZ^m \subset \XX (\wZ^m) \subset \Xi_m$ are disks inductively constructed as follows (from bigger $m$ to smaller ones):
\begin{enumerate}
\item $\widehat Z^m = \overline{Z}$ and $\XX(\widehat Z^m) = \overline{Z}$ for all sufficiently large $m \gg 0$,
\item either \[\widehat Z^m := \widehat Z^{m+1}\text{ and } \XX(\widehat Z^m) := \XX(\widehat Z^{m+1}),\] 
or for every interval $I\in \mathfrak D_m$ there is 
\begin{itemize}
    \item a parabolic fjord $\widehat{\mathfrak F}_I\equiv \widehat{\mathfrak F}_{\beta_I}$ bounded by its dam $\beta_I$ with endpoints in $I$; and
    \item a rectangle $\XX_I$ protecting $\widehat{\mathfrak F}_I$, i.e.,  $\widehat{\mathfrak F}_I\subset \upbullet\XX_I$,
\end{itemize}
 such that    
   \begin{equation}\label{eq:defn:pseudo-Siegel disks}
      \begin{aligned}
     \widehat Z^m &:= \widehat Z^{m+1} \cup \bigcup_{I\in \mathfrak{D}_m} \widehat{\mathfrak F}_{I}\\ 
         \XX \big(\widehat Z^m\big) &:= \XX(\wZ^{m+1})\cup \bigcup_{I\in \mathfrak{D}_m} \upbullet \XX_I 
    \end{aligned}\end{equation}
and such that $\wZ^m$ and $\XX \big(\widehat Z^m\big)$ satisfy $7$ compatibility condition stated in \cite[\S~5.1]{DL22}. 
\end{enumerate}
\end{defn}

Here $\upbullet \XX$ the union of $\XX$ and the closure of the connected component of $V\setminus (\XX\cup \overline Z)$ enclosed by $\partial^v \XX \cup \overset{\circ} {I}$, where $\overset{\circ} {I}$ is $I$ without its endpoints. (Removing the endpoints is relevant if $m=-1$.) 

It follows that $\widehat{\mathfrak F}_I\subset \Xi_m$ that $\widehat{\mathfrak F}_I\subset \Xi_I\subset \mathfrak F(I)$, see~\eqref{eq:0:sss:F:bistab}.

\subsubsection{Bistability of $\wZ^m$}\label{sss:dfn:bistable} Given a pseudo-Siegel disk $\wZ^m$, its bistability $\bistab(\wZ^m)$ is the smallest number $k$ so that $\XX(\wZ^m)\subset \Xi_m^{k+1}$. It follows that for every $i$ with $|i|\le \qq_{m+1}k$, the induced image 
\begin{equation}
\label{eq:dfn:wZ_*i}
\wZ^m_{*i} \coloneqq F^i_{\Xi_m^k}(\wZ^m),\qquad\qquad   F^i_{\Xi_m^k}\mid_{\Xi_m^k}\ \ \text{is in ~\eqref{eq:F_biXi}}    
\end{equation}
 is also a pseudo-Siegel disk satisfying Definition~\ref{defn:pseudo-Siegel disks} with territory \[\XX(\wZ^m_{*i})= F^i_{\Xi_m^k}\big(\XX(\wZ^m)\big)\subset \Xi_m.\]

\subsubsection{Geodesic pseudo-Siegel disks}\label{sss:geodes:wZ} We say that a pseudo-Siegel disk $\wZ^m$ is \emph{geodesic} if
\begin{itemize}
    \item  all the auxiliary objects are hyperbolic geodesics in either $V \setminus \overline{Z}$ or $\wZ^m$; we refer the readers to \cite[\S~5.1.9]{DL22} for more details; in particular, $\wZ^m$ is obtained by adding geodesic fjords;
    \item $\bistab(\wZ^m)\ge 10$.
\end{itemize}

\subsubsection{Remarks about stability and bistability}\label{sss:rem:bistab} In~\cite[\S~5.1.8]{DL22}, the notion of stability $\stab(\wZ^m)$ was introduced for pseudo-Siegel disks $\wZ^m$ in terms of combinatorial distances between the endpoints of $I$ and $\XX_I\cap I$. \cite[Lemma 5.4]{DL22} implies the stability of $\wZ^m$ under pulling-back.

In the current paper, we find it more natural to consider bistability instead of stability. Consequently, the notion of geodesic pseudo-Siegel disks in \S\ref{sss:geodes:wZ} is slightly stronger than the associated notion in \cite[\S~5.1.9]{DL22}.

Exponentially big bistability occurs inside a deep part of the parabolic fjords $\Fjord([\tilde x, \tilde y])$, see Lemma~\ref{lem:geodesicpseudoSiegel}. Consequently, we will show in \S\ref{sss:big:bistability} that pseudo-Siegel disks $\wZ^m$ constructed in Theorem~\ref{thm:main:psi*-ql maps} have bistability $\ge e^{\sqrt \bK}$, where $\bK\gg 1$ is a big constant representing $O(1)$ in Item~\ref{thm:apbs:A} of Theorem~\ref{thm:main:ql maps}. In other words, pseudo-Siegel disks naturally come with arbitrary large (by increasing $\bK$) bistability.

\subsubsection{Types of intervals on $\partial\wZ^m$}\label{sss:intervalpsuedoSiegel}
An interval $I^m \subseteq \partial \wZ^m$ is {\em regular} if its endpoints are contained in $\partial \wZ^m \cap \partial Z$, and is {\em grounded} if it is regular and its endpoints are away from the inner-buffers (see \cite[\S~5.2.3]{DL22}). 
Let $I\subset \partial Z$ be an interval. It is  {\em regular rel $\wZ^m$} if its endpoints are contained in $\partial \wZ^m \cap \partial Z$.
For a regular interval $I$, its {\em projection} $I^m \subseteq \partial \wZ^m$ is the interval on $\partial \wZ^m$ with the same end point and the same orientation as $I$. For a general interval $I \subset\partial Z$, its projection is the projection of the smallest regular interval that contains $I$.
An interval $I$ is {\em grounded rel $\wZ^m$} if its projection $I^m \subseteq \partial \wZ^m$ is grounded. The {\em well-grounded} intervals can be defined in the exact same way as \cite[\S~5.2.2]{DL22}.

We note that grounded intervals are used in the formalization of the Log-Rules for $\intr(\wZ^m)$~\cite[\S5.3, \S5.4]{DL22} and of the snake theory~\cite[\S6]{DL22}.

\subsubsection{Families of curves for pseudo-Siegel disks}
We remark that the definition in \S~\ref{subsec:familyofarcs} naturally extends to the pseudo-Siegel disks, by taking projections.
For example, let $I$ be a regular interval in $\partial Z$, with projection $I^m \subset \wZ^m$.
We denote by $\Fam^{+}_\lambda(I^m) = \Fam^{+}_{\lambda, \wZ^m}(I^m)$ the family of arcs in $V \setminus \wZ^m$ that connects $I^m$ and $\partial V \cup (\partial \wZ^m \setminus (\lambda I)^m)$.
The vertical, peripheral, external and diving subfamilies are defined accordingly.

\subsubsection{Outer geometry of $\wZ$}\label{sss:Outer:wZ}
Under certain minor conditions, outer families rel $\partial Z$ and $\partial \wZ^m$ can be easily converted into each other, see~\eqref{eq:conver:I_J},~\eqref{eq:conver:I_V}, and~\eqref{eq:conver:I_lambda} below. The conversion is based on~\cite[\S5.2]{DL22}; here we summarize details. Below, we will compare the geometries of $V\setminus \overline Z$ and $V\setminus \wZ^m$. A somewhat similar comparison between $U^k\setminus \wZ^m$ and $\widehat U^k\setminus \wZ^m$ is discussed in Lemma~\ref{lem:liftingviaiota}.

Given an interval $I^m\subset \partial \wZ^m$, let $I^{m,\grnd}\subset I^m$ be the biggest grounded subinterval within $I^m$. Then $I^m\setminus I^{m,\grnd}$ has at most two components; each of these components is protected by an annulus with a definite modulus. It easily follows (see~\cite[Lemma 5.18]{DL22}) that if $\dist(I^m,J^m)\ge \length_m$, then
\[\Width(I^m,J^m)=\Width(I^{m,\grnd},J^{m,\grnd})+O(1).\]

Given $I\subset \partial Z$, let $I^\grnd\subset I$ be the biggest regular subinterval of $I$ such that the projection $I^{m,\grnd}$ of $I^\grnd$ onto $\partial \wZ^m$ is grounded. Using~\cite[Lemma 5.10]{DL22}, we obtain:
\begin{multline}
\label{eq:conver:I_J}
  \qquad  \Width^+_{Z}(I,J) =\Width^+_{Z}(I^{\grnd},J^\grnd)+O(1)=\\ \Width^+_{\wZ^m}(I^{m,\grnd},J^{m,\grnd})+O(1), \qquad
\end{multline}
and 
\begin{multline}\label{eq:conver:I_V}
  \qquad  \Width^+_{Z}(I,\partial V) =\Width^+_{Z}(I^{\grnd},\partial V)+O(1)=\\ \Width^+_{\wZ^m}(I^{m,\grnd},\partial V)+O(1). \qquad
\end{multline}

Consequently, 
\begin{multline}\label{eq:conver:I_lambda}
  \qquad  \Width^+_{\lambda, Z}(I) =\Width^+_{Z}\big(I, (\lambda I)^c\big)+\Width^+_{Z}(I,\partial V)-O(1)= \\ \Width^+_{\wZ^m}\left(I^{m,\grnd},\ \Big(\lambda I)^c\Big)^{m,\grnd}\right)+\Width^+_{\wZ^m}(I^{m,\grnd},\partial V)+O(1). \qquad
\end{multline}

To simplify notations, let us denote the second line in~\eqref{eq:conver:I_lambda} as
\begin{multline} \label{eq:dfn:proj:Fam}   \qquad \Width\big(\proj^{+,\grnd}_{\lambda,V, \wZ^m}(I) \big) \equiv  \Width\Big(\proj^\grnd_{V,\wZ^m} \big(\Fam^+_{\lambda, Z}(I)\big) \Big) \coloneqq \\\Width^+_{\wZ^m}\left(I^{m,\grnd},\ \Big((\lambda I)^c\Big)^{m,\grnd}\right)+\Width^+_{\wZ^m}(I^{m,\grnd},\partial V),\qquad
\end{multline}
where $\proj_{V,\wZ^m} \big(\Fam^+_{\lambda, Z}(I)\big)$ is the \emph{projection} of the family $\Fam^+_{\lambda, Z}(I)$ into $V\setminus \overline Z$. Then~\eqref{eq:conver:I_lambda} takes form
\begin{equation}
    \label{eq:conver:I_lambda:2}\Width^+_{\lambda, Z}(I) \ =\ \Width\big(\proj^{+,\grnd}_{\lambda, V, \wZ^m}(I) \big) +O(1)
\end{equation}

We remark that the class of grounded intervals in~\eqref{eq:conver:I_J},~\eqref{eq:conver:I_V},~\eqref{eq:conver:I_lambda} is selected for convenience. For instance, instead of $I^{m,\grnd}\subset I^m$, one can specify a combinatorially smaller (more ``optimal'') subinterval $I'^m\subset I^{m,\grnd}$ by considering the biggest subinterval $I'^m\subset I^m$ whose endpoints are outside of (the boundary of) any fjord $\Fjord^m_I$ appearing in~\eqref{eq:defn:pseudo-Siegel disks}.

\subsection{Lifting from $(V,\widehat Z^m)$ to $(U^k,\widehat Z^m_{U^k})$ under $\iota^k$}\label{ss: propertyimmersion}
As we have already mentioned, in the near-degenerate regime, $\psi^\bullet$-ql maps are quite similar to ql maps.  Figure~\ref{Fig:psibullet maps} illustrates an immersion $\iota\colon U\to V$ defined around a $\iota$-proper disk  $\overline Z$. The $\iota$-proper condition implies that laminations $\LL$ emerging from $\partial Z$ can be ``injectively'' lifted  into 
\[\iota^*\LL\sp\sp \text{ with }\sp\sp \Width(\iota^*\LL)\ge \Width(\LL);\]
this is quite similar to the case when $\iota$ is an inclusion $U\hookrightarrow V$. See \S~\ref{subsubsec:univalentpush} below for more details.

Note that $\wZ^m$ is not $\iota$-proper under $\iota^k\colon U^k\to V$. The following lemma resolves this subtlety by comparing   $\iota^k\colon U^k\to V$ with $\iota^k\colon \widehat U^k\to V$, see \eqref{eq:dfn:wU^k}, \eqref{eq:wU^k:proper}, and by arguing that $(\widehat U^k, \wZ^m)$ and $(U^k, \wZ^m)$ are geometrically close (protections of $\wZ^m\setminus \overline Z$ lift to protections of $\widehat U^k\setminus U^k$). Consequently, pushforwards $F^k_*=f^k_*\circ (\iota^k)^*$ are defined naturally for laminations emerging from $\partial Z$ or from $\partial \wZ^m$ for $k\le \qq_{m+1}$ and have similar properties as in the case of ql-maps; see \cite{K,KL2,molecules,DL:Feigen}.

\begin{lem}\label{lem:liftingviaiota}
Let $I^m \subset \wZ^m$ be a grounded interval, and $\tau > 1$. Let $k \leq 10\qq_{m+1}$.
Then there exists a (small) $\varepsilon \in (0,0.1)$
$$
\Width_{\tau, U^k}(I^m) \ge (1-\varepsilon) \Width_{\tau}(I^m) - O(1),
$$
where $\Width_{\tau}(I^m) = \Width_{\tau, V}(I^m)$.
\end{lem}
\begin{proof}
   Set 
\begin{equation}
  \label{eq:dfn:wU^k}
   \widehat U^k =  \iota^{-k}(V \setminus \wZ^m) \cup \wZ^m_{U^k}\subset U^k
\end{equation}   
      to be $U^k$ minus $\iota^{-k}(\wZ^m)\setminus \wZ^m_{U^k}$; i.e., we remove from $U^k$ components of $\iota^{-k}(\wZ^m)$ attached to its (external) boundary. Consequently, $\wZ^m$ is $\iota$ proper under $\iota^k\colon \widehat U^k\to V$:
      \begin{equation}
          \label{eq:wU^k:proper} \big[\iota^k\mid \widehat U^k\big]^{-1} (\wZ^m_V)=\wZ^m_{U^k}
      \end{equation}

  Recall that the pullback of a family $\Fam$ under an immersion $\iota^k$ consists of two steps: restriction followed by a lift.  Both operation can only increase the width. Thus,
   $$
   \Width_{\tau, \widehat U^k}(I^m) \geq \Width_\tau(I^m).
   $$
Let $\LL$ be a union of finitely many rectangles representing $\Fam_{\tau, \widehat U^k}(I^m)$ up to $O(1)$.

Let us now discuss the difference between $U^k$ and $\widehat U^k$. By construction,  if $\widehat{\mathfrak F}^k_i$ is a connected component of $U^k\setminus \widehat U^k$, then $\widehat{\mathfrak F}_j\coloneqq \iota^k(\widehat{\mathfrak F}^k_i)$ is a component of $\wZ^m\setminus \overline Z$. We recall that $\widehat{\mathfrak F}_j$ is bounded by a dam $\beta_j$ with an extra protection $\XX_j$ separating $\widehat{\mathfrak F}_j$ from $\partial V$. Since $\XX_j$ is in the coastal zone, it has a lift $\XX_i^k$ under $\iota^{k}$ such that $\XX_i^k$ separated $\widehat{\mathfrak F}^k_i$ from the pseudo-Siegel disk $\wZ^m$. Since $\XX_i^k$ are sufficiently wide, \cite[Lemma 5.6]{DL22} implies that components $\widehat{\mathfrak F}^k_i$ affect the width of $\LL$ by a multiplicative factor $(1-\varepsilon)$, i.e., $\Width_{\tau, U^k}(I^m) \ge (1-\varepsilon) \Width(\LL)$.
This implies the main estimate.  
\end{proof}

\subsubsection{Modification of transformation rules for $(V, \wZ^m_{-k*})\leadsto (V, \wZ^m)$}
\label{sss:CL:wZ^m}
Consider now a pseudo-Siegel disk $\wZ^m$ and recall that $\wZ^m_{-k*}=F^{-k}_{\Xi_m^k}(\wZ^m)$ is the induced pullback, see~\eqref{eq:dfn:wZ_*i}. Then simple and near-degenerate transformation rules for $\wZ^m$ follow the following modification~\eqref{eq:transformation_pattern}: 
\begin{equation}
\label{eq:transformation_pattern:wZ}
     (F^k)_*=(f^k)_*\circ (\hookrightarrow)_*\circ (\iota^k)^*\colon (V,\wZ^m_{-k*})\leadsto(\widehat U^k,\wZ^m)\leadsto(U^k,\wZ^m)\leadsto (V,\wZ^m),
\end{equation}
where the composition
\[(\hookrightarrow)_*\circ (\iota^k)^*\colon \ (\widehat U^k,\wZ^m)\leadsto(U^k,\wZ^m)\leadsto (V,\wZ^m)\]
is controlled by Lemma~\ref{lem:liftingviaiota}.
Below, we provide more details with notations.

Let $I^m\subset \partial \wZ^m$ and $\tau > 1$. Let $\LL$ be a lamination consisting of arcs in $V \setminus (I^m \cup (\partial \wZ^m \setminus \tau I^m))$ connecting $I^m$ to $\partial V \cup (\partial \wZ^m \setminus \tau I^m)$.
Let $X, Y$ as constructed in \ref{subsubsec:simplepush} relative to $\wZ^m$. Let $k \leq 10\qq_{m+1}$.

By Lemma~\ref{lem:liftingviaiota}, the simple transformation rule and near-degenerate transformation rule in \S~\ref{subsubsec:simplepush} and \S~\ref{subsubsec:near-degeneratepush} are defined similarly in the psuedo-Siegel disk setting by replacing \eqref{eqn:annulusbound} at the expense of $(1-\varepsilon)$:
\begin{equation}\label{eqn:psuedoannulusbound}
\Width(X\setminus I^m) \ge \Width_{\tau, U^k}(I^m) \ge (1-\varepsilon) \Width_{\tau}(I^m) - O(1) \ge (1-\varepsilon) \Width(\LL) - O(1).
\end{equation}

\subsection{Spreading around a pseudo-Siegel disk}\label{sss:SpreadAround:rel:wZ} 
Consider a combinatorial interval $I\subset \partial Z$ witnessing big outer degeneration 
\[ K\coloneqq \Width^+_{\lambda, \overline Z}(I)\gg 1.\]
Assume $\wZ^m$ is a pseudo-Siegel disk with bistablislity at least up to the first return, \S\ref{sss:dfn:bistable}. To spread around the degeneration $\Fam^+_{\lambda, \overline Z}(I)$ into all $I_k=f^k(I), k <l\qq_{m+1}$ for some $l \leq \bistab(\wZ^m)$,  
\begin{itemize}
    \item we first project $\Fam^+_{\lambda, \overline Z}(I)$ into $\wZ^m_{-k*}$, see~\eqref{eq:dfn:wZ_*i}; this step has a small affect on $K$, see~\eqref{eq:dfn:proj:Fam} and~\eqref{eq:conver:I_lambda:2};
    \item then we apply either a simple or near-degenerate transformation rule along $(V, \wZ^m_{-k*})\leadsto (V, \wZ^m)$ as in~\S\ref{sss:CL:wZ^m}, see also~\eqref{eq:transformation_pattern:wZ}.
\end{itemize}
The associated family is transformed as follows:
\begin{equation}
    \label{eq:SpreadAround:wZm}  \Fam^+_{\lambda, \overline Z}(I)\ \leadsto \ \proj^\grnd_{V,\wZ^m_{-k*}} \big(\Fam^+_{\lambda, Z}(I)\big) \ \overset{(F^k)_*}{\leadsto} \ \proj^\grnd_{V,\wZ^m} \big(\Fam_{\lambda, Z}(I_k)\big).
\end{equation}
where $\proj^\grnd_{V,\wZ^m} \big(\Fam_{\lambda, Z}(I_k)\big)$ represents the full family in $(V,\wZ^m)$ from $I_k^{m,\grnd}$ to $\big((\lambda I_k)^c\big)^{m,\grnd}\cup \partial V.$

In Section~\ref{sss:AppC:proof:Sect8} (just as in~\cite[\S8.2]{DL22}), we will apply the spreading around under the assumption that one of the endpoints of $I$ is in $\CP_m$. With this assumption, the intervals $I_k$ are grounded rel $\wZ^m$ and $\wZ^m_{-k*}$ and the notations involving ``$\proj$'' are not necessary.

On the other hand, in the proof of Lemma~\ref{lem:replication}, we will apply spreading around using notation from~\eqref{eq:SpreadAround:wZm} to all $I_k=f^k(I), k <2\qq_{m+1}$.

\section{Parabolic fjords and their welding with $\wZ^m$} \label{sec:parabolicfjords}

In this section, we discuss the geometry of parabolic fjords. Let us fix an interval $T=[v,w]\in \Dbb_m$ in the level $m$ diffeo-tiling. We assume that $T$ consists of many level $m+1$ combinatorial intervals. Figure~\ref {Fig:ParabolicF} illustrates the geometry of wide families in the associated parabolic fjord. Roughly, we first select the outermost parabolic external rectangle of fixed but big width; we assume that such a rectangle exists and connects $[\tilde x, x]$ and $[y,\tilde y]$. Moreover, we can assume that this rectangle is balanced, see \S \ref{ss:parabolicRectangle} and \S \ref{ss:BalSubrect}. Below the rectangle, the Log-Rule (see Theorem~\ref{thm:parabolicfjords}) is applicable. Outside of the rectangle, wide families are either diving or vertical.

We remark that the interval $[\tilde x,\tilde y]$ is \emph{a priori} specified geometrically. Therefore, the Log-Rule Theorem depends on the following two parameters: \[M=\frac{\dist(\tilde x, x)}{\length_{m+1}}=\frac{\dist(\tilde y, y)}{\length_{m+1}} \qquad \text{ and }\qquad N=\frac{\dist(\tilde y, w)}{\length_{m+1}}, \] see~\eqref{eq:dfn:M} and~\eqref{eq:dfn:N}. If $N\gg_{\lambda}\ M$, then there is a non-external family emerging from $E\subset [\tilde y, w]$, see Case~\ref{item:3:thm:parabolicfjords} in Theorem~\ref{thm:parabolicfjords}. This leads to an \emph{exponential boost} (an application of the Log-Rule, see~\S\ref{sss:ExponBoost}) stated in Lemma~\ref{thm:center:rect}, Case~\ref{item:3:ParabolicF}. In the a posteriori setting, $[\tilde x, \tilde y]$, $N$, and $M$ can be assumed specified combinatorially, see Theorem~\ref{thm:main:psi-ql maps:extra} and the discussion before.

 In~\S\ref{ss:WeildingLmm}, using the Log-Rule, we show that either there is a regularization of the level $m+1$ pseudo-Siegel disk $\wZ^{m+1}\ \leadsto \ \wZ^{m}$, or there is an {\em exponential boost} in the degeneration (see Theorem~\ref{thm:regul}).

 As it has been mentioned in~\S\ref{sss:occ:Fam^+}, this section mostly focuses on the external component $\Fam^{+,\per}_{m,\ext}$ of outer families $\Fam^+$. In~\S\ref{ss:parabolicRectangle} we recall the notion of parabolic rectangles from~\cite{DL22}.  Lemma~\ref{lem:Balanced subrectangle} says that any wide external parabolic rectangle is balanced. After Lemma~\ref{lem:Balanced subrectangle}, the arguments proceed as in~\cite{DL22}. A new addition is Lemma~\ref{lem:geodesicpseudoSiegel} stating that exponentially big bistability occurs inside a deep part of the parabolic fjords $\Fjord([\tilde x, \tilde y])$.

\subsection{Parabolic Rectangles}\label{ss:parabolicRectangle}

Parabolic rectangles encode wide families over intervals with substantial length, i.e.,~\eqref{eq:parab rect} holds; see a rectangle connecting $A$ and $B$ on Figure~\ref{Fig:ParabolicF}. Parabolic rectangles occur withhin parabolic fjords; their width is described by Log-Rule in Theorem~\ref{thm:parabolicfjords}.

Let $T\in \Dbb_m$ be an interval in the diffeo-tiling. Recall Fjords $\Fjord$ and $\Fjord_m$ are defined in \S~\ref{sss:F^*}. A rectangle $\RR \subseteq V \setminus Z$ based on an interval $S\subset T$ is called \emph{parabolic based on $S$} (of level $m$) if
\begin{equation}
\label{eq:parab rect}
\dist_{T}(\partial^{h,0}\RR,  \partial^{h,1}\RR ) \ge 6\min \{ | \partial^{h,0}\RR|,\sp | \partial^{h,1}\RR| \} +3\length_{m+1}
\end{equation}
i.e.~the gap between $ \partial^{h,0}\RR$ and $\partial^{h,1}\RR$ is bigger than the minimal horizontal side of $\RR.$ We say that a parabolic rectangle $\RR$ is \emph{balanced} if $|\partial^{h,0} \RR|=|\partial^{h,1} \RR|$. 

For our application, we state following corollary, which is a reformulation of Lemma~\ref{cor:pullback} with simplified notations. 
\begin{cor}[Pullbacks of external parabolic rectangles]\label{cor:pullbacks for psi bullet} If $\RR\subset V\setminus Z$ is a wide external parabolic rectangle based on $T'$, then, after removing the outermost $1$-buffer from $\RR$, the new rectangle $\RR^\new$ is in $\Fjord[v',w']$ and can be iteratively pullbacked under~\eqref{eq:cor:pullbacks for psi bullet} towards $v$.
\end{cor}
\begin{comment}
\begin{proof} 
\reminder{we do not need any proof and any non-winding notion}
Just as in the quadratic case, parabolic external rectangles on $T'$ are, up to removing $O(1)$-buffers, non-winding (see~\cite[\S4.1]{DL22}): their vertical leaves are homotopic to $T'$ in $V\setminus Z$. Since the non-winding property persists under pullback, external rectangles can be shifted towards $v$ via $F^{-\qq_{m+1}}_T$.
\end{proof}

\reminder{Put here Corollary~\ref{cor:pullbacks for psi bullet}?}
\end{comment}

\begin{figure}
    \centering
    \includegraphics[width=0.9\linewidth]{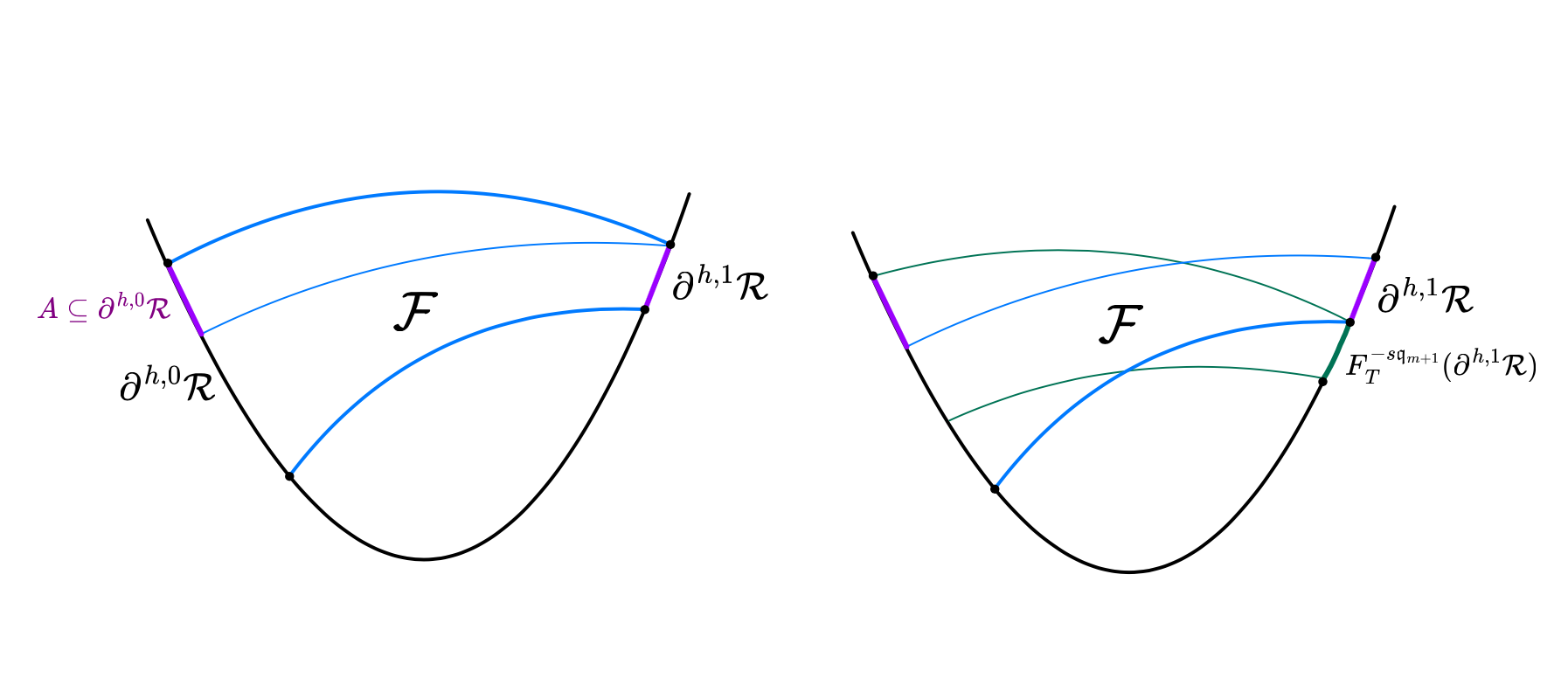}
    \caption{Illustration of the Shift Argument for Property~\ref{cl:2:bpar_rect}. Suppose $|\partial^{h,0} \RR|> |\partial^{h,1} \RR|$. Let $A \subseteq \partial^{h, 0} \RR$ be the interval indicated in the left figure with $|A| = |\partial^{h,1} \RR|$. Consider any rectangle $\RR'$ in $V \setminus Z$ with vertical boundaries $\partial^{h,0} \RR \setminus A$ and $\partial^{h, 0} \RR$. Then we can consecutively pull back by $F^{-\qq_{m+1}}_T$, and obtain a configuration on the right figure. This implies that $\RR'$ has width at most $2$ by \cite[\S A.3, Lemma A.11]{DL22}. Thus, by removing $O(1)$-buffers, we may assume Property~\ref{cl:2:bpar_rect} holds.}
    \label{Fig:ShiftArgument}
\end{figure}

\subsection{Balanced subrectangle}\label{ss:BalSubrect}

It follows from the existence of the shifts (i.e., from Items~\ref{shift:to:v} and~\ref{shift:to:w} in the proof of Lemma~\ref{lem:Balanced subrectangle}) that every wide parabolic rectangle $\RR$ is balanced: $|\partial^{h,0} \RR^\new| = |\partial^{h,1} \RR^\new|$ after removing $O(1)$ buffers; see Items~\ref{cl:2:bpar_rect} and~\ref{cl:4:bpar_rect} in the lemma below.

\begin{lem}\label{lem:Balanced subrectangle}
    Let $\RR$ be a wide external parabolic rectangle. Then $\RR$ contains a geodesic balanced subrectangle $\RR^\new$ with 
    $$
    \Width(\RR^\new) = \Width(\RR) - O(1)
    $$ 
    satisfying the upper Log-Rule. More precisely, we have
    \begin{enumerate}[label=(\arabic*)]
    \item\label{cl:1:bpar_rect} $\RR^\new$ is geodesic in $V - \overline{Z}$;
 \end{enumerate} 
 \begin{enumerate}[label=(\arabic*a),start=2]
    \item\label{cl:2:bpar_rect} $|\partial^{h,0} \RR^\new|\le |\partial^{h,1} \RR^\new|$;
  \end{enumerate} 
 \begin{enumerate}[label=(\arabic*b),start=2]   
    \item\label{cl:4:bpar_rect} $|\partial^{h,0} \RR^\new|\ge |\partial^{h,1} \RR^\new|$;
\end{enumerate} 
 \begin{enumerate}[label=(\arabic*),start=3]
    \item\label{cl:3:bpar_rect} $\RR$ obeys upper Log-Rule: $\Width(\RR^\new) \preceq \log \frac{|\partial^{h,0} \RR^\new|}{\dist(v, \partial^{h,0} \RR^\new)} + 1$.
\end{enumerate}
\end{lem}
\begin{proof}
    The proof is the same as a series of lemmas (more precisely, \cite[Lemma~4.4, Lemma~4.5, Lemma~4.10]{DL22}) in \cite[\S 4]{DL22} and relies on shifts
\begin{itemize}
   \item[\setword{(v)}{shift:to:v}] towards $v$ via the pullback $F^{-\qq_{m+1}}_T$, see~\eqref{eq:0:cor:pullbacks for psi bullet} in  Corollary~\ref{cor:pullbacks for psi bullet};
  \item[\setword{(w)}{shift:to:w}] towards $w$ using univalent push-forward $F^{\qq_{m+1}}$, see \S~\ref{subsubsec:univalentpush}.
\end{itemize}
We remark that shifts towards $v$ are relatively easy, while shifts towards $w$ are more delicate. In general, the shift operation towards $w$ produces three types of curves illustrated \cite[Figure~12]{DL22}. However, when an outer protection by an external family is constructed, wide families can be shifted below the protection towards $w$ with uniformly bounded loss of width, see \cite[Lemma 4.9]{DL22}.

    Let us sketch the main steps of the proof.
    First, we note that by removing $O(1)$-buffers from $\RR$, we may assume that $\RR$ satisfies Property~\ref{cl:1:bpar_rect}, \ref{cl:2:bpar_rect}, \ref{cl:3:bpar_rect}. 
    Indeed, Property~\ref{cl:1:bpar_rect} is general; see~\cite[Lemma A.5]{DL22}. Property~\ref{cl:2:bpar_rect} relies on the Shift Argument towards $v$ illustrated on Figure~\ref{Fig:ShiftArgument}; compare with~\cite[Lemmas 4.4]{DL22}. Property~\ref{cl:3:bpar_rect} follows from iterating~\ref{cl:2:bpar_rect}: we can split $\RR^\new$ into at most $\frac{|\partial^{h,0} \RR^\new|}{\dist(v, \partial^{h,0} \RR^\new)} + 1$ small rectangles and use~\ref{cl:2:bpar_rect} to claim that each small rectangle has width $\preceq 1$, see details~\cite[Lemmas 4.5]{DL22}. 

    Using shifts towards $w$ (push-forward, Item~\ref{shift:to:w}), we can additionally impose Property~\ref{cl:4:bpar_rect} (see~\cite[Lemmas 4.10]{DL22}).
    The argument relies on extracting first an outer buffer $\RR_+$ with width $\asymp 1$. Most of width of $\RR$ is in $\RR^\new\coloneqq \RR\setminus \RR_+$. Since curves in $\RR^\new$ are protected by $\RR_+$, the Shift Argument towards $w$ is applicable to and Properties~\ref{cl:1:bpar_rect}, \ref{cl:2:bpar_rect}, \ref{cl:3:bpar_rect} can be claimed (together with~\ref{cl:4:bpar_rect}) for $\RR^\new$.
\end{proof}

\begin{rem}
It follows from Lemma~\ref{lem:Balanced subrectangle} that a rectangle $\RR$, up to $O(1)$, can be replaced by a family $\Fam^{+}_{\ext,m}(I, J )$ connecting two intervals $I,J$ with $|I|=|J|$:
\[\Width(\RR) = \ \Width^{+}_{\ext,m}(I, J )\qquad \text{where } I = \partial^{h,0}\RR^\new, J=\partial ^{h,1} \RR^\new.\]
Note that $\partial^{h,0}\RR, \partial^{h,0}\RR$ can be much bigger than $\partial^{h,0}\RR^\new, \partial^{h,0}\RR^\new$.
We will not distinguish much between wide parabolic rectangles $\RR$ and families $\Fam^{+}_{\ext,m}(I, J )$. In particular, we formulate Theorem~\ref{thm:parabolicfjords} using pairs of intervals. 
\end{rem}

\subsection{Log-Rule} \label{ss:Log-Rule}
By the discussion in \S~\ref{subsubsec:efVvsU}, we consider a pair of intervals $I, J \subset T = [v,w] \in \Dbb_m$.
Let us also assume that
\begin{align}
\label{eqn:regassumption}
    \text{There is an interval $I \subseteq T$ with } \Width^{+}_{5, \ext, m}(I) \geq C\gg 1.
\end{align}
We remark that if the assumption~\eqref{eqn:regassumption} does not hold, then we do not do any regularization of the Siegel disk on this level.

\begin{figure}
   \includegraphics[width=10cm]{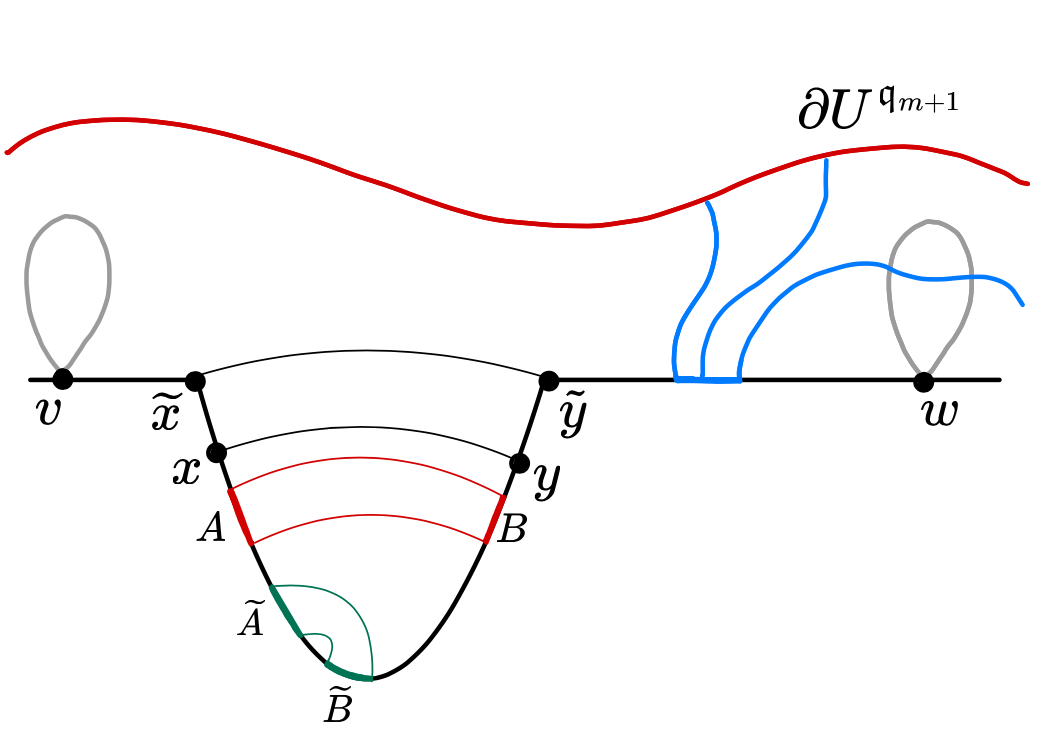}
   \caption{Illustration to the geometry of a parabolic fjord based on $T=[v,w]$. Wide families based on $[x,y]$ are external: the red rectangle between $A_1, B_1$ represents Case \ref{item:2:1:thm:parabolicfjords} in Theorem \ref{thm:parabolicfjords}, while the green rectangle between $A_2, B_2$ represents Case \ref{item:2:2:thm:parabolicfjords}. Wide families outside of $[\tilde x,\tilde y]$ (in blue) are either diving or vertical.}
   \label{Fig:ParabolicF}
\end{figure}

\begin{thm}
\label{thm:parabolicfjords}
Under the assumption~\eqref{eqn:regassumption}, we have the following structure in the parabolic fjords that encodes wide external families.
\begin{enumerate}[label=(\arabic*)]
    \item \label{item:1:thm:parabolicfjords} There is an outer protection: there are $[\widetilde{x}, x], [y, \widetilde{y}]$ with \begin{equation}\label{eq:dfn:M} M\length_{m+1} \coloneqq \dist(\widetilde{x}, x) = \dist(\widetilde{y}, y) \leq \dist(x, y)/100\end{equation} and $\dist(v, \widetilde{x}) < M \length_{m+1}/100$
such that    
    $$
     \Width^{+}_{\ext, m}([\widetilde{x},x], [y, \widetilde{y}]) = C \gg 1.
    $$
    \item \label{item:2:thm:parabolicfjords}  The external families in $[x, y]$ satisfies the Log-Rule. More precisely,
    \begin{enumerate}[label=(\roman*)]
        \item\label{item:2:1:thm:parabolicfjords} If $x < I < J < y$ are two intervals with 
        $$
        \dist(x, I) \asymp \dist(J, y) \preceq |I| \asymp |J| \preceq \dist(I, J)
        $$
        and $|I|, |J|, \dist(I, J) \geq \length_{m+1}$, then 
         \begin{equation}
           \label{eq:1:thm:parabolicfjords}
        \Width^+(I,J) - O(1) = \Width^+_{\ext, m}(I, J) \asymp \log^+ \frac{\min\{|I|, |J|\}}{\dist(v, I)} + 1.
        \end{equation}
        \item \label{item:2:2:thm:parabolicfjords} If $x < I < J < y$ are two intervals with 
        $$
        |\lfloor I, J \rfloor| \leq \frac{1}{2}|[x,y]| \text{ and } \min\{|I|, |J|\} \succeq \dist(I, J) \geq 3 \length_{m+1}
        $$
        and $|I|, |J| \geq \length_{m+1}$, then 
        \begin{equation}
           \label{eq:2:thm:parabolicfjords}         
        \Width^+(I,J) - O(1) = \Width^+_{\ext, m}(I, J) \asymp \log^+ \frac{\min\{|I|, |J|\}}{\dist(I, J)} + 1.
        \end{equation}
    \end{enumerate}
    \item \label{item:3:thm:parabolicfjords} For any interval $I \subseteq T\setminus [x,y]$, we have
    $$
    \Width^{+}_{5, \ext, m}(I) = O(C)\equiv O(1).
    $$
    \item\label{item:4:thm:parabolicfjords} Set 
    \begin{equation}\label{eq:dfn:N}
    N \coloneqq \frac{\dist(\widetilde{y}, w)}{\length_{m+1}}.
\end{equation}
If $\frac{N}{\lambda M} \gg 1$, then there exists an interval $E \subseteq [\widetilde{y}, w]$ so that
\[\Width^{+}_\lambda(E) = \Width^{+, \ver}(E)+\Width^{+, \per}_{\lambda,\div,m}(E) + O(1) \succeq \frac{N}{\lambda M}.\]
\end{enumerate}
\end{thm}
\begin{proof}
Justification of the theorem is the same as for~\cite[Theorem~4.1]{DL22} and \cite[Corollary~4.11]{DL22}.

Let $\RR$ be the outermost external parabolic rectangle; see the choice of $S$ in~\cite[(4.15)]{DL22}. By Lemma~\ref{lem:Balanced subrectangle}, there exists a subrectangle $\RR^\new$ in $\RR$ satisfying Properties~\ref{cl:1:bpar_rect}--\ref{cl:3:bpar_rect}.
We set $[\tilde x, x]$ and $[y,\tilde y]$ to be horizontal sides of $\RR^\new$, see Figure~\ref{Fig:ParabolicF}. Below $\RR^\new$, both left and right shifts are justified and they imply Items~\ref{item:1:thm:parabolicfjords} -- \ref{item:3:thm:parabolicfjords} in the same way as the proof of~\cite[Theorem~4.1]{DL22}.

Item~\ref{item:4:thm:parabolicfjords} is relevant only if $N$ is much bigger than $M$, see~\eqref{eq:dfn:N} and~\eqref{eq:dfn:M}. If this is the case, then move rectangle $\RR^\new$ towards $w$ as shown on \cite[Figure~14]{DL22}; the result is a numerous (i.e., $\asymp  N/M$) collection of families emerging from $[\tilde y, w]$. Only $O(1)$ curves in this collection are external because of the outermost choice of $\RR$. By selecting $E$ well inside (in terms of $\lambda$) $[\tilde y, w]$, we justify Item~\ref{item:4:thm:parabolicfjords}.
\end{proof}

\subsection{Central rectangles} As in~\cite[\S4.5]{DL22}, we say that a parabolic rectangle $\RR$ based on $T'$ is \emph{central} if 
\begin{equation}
    \label{eq:dfn:CentrRect}
0.9< \frac{\dist_T(v, \lfloor \partial^h \RR \rfloor)}{\dist_T( \lfloor \partial^h \RR \rfloor,w)}<1.1;
\end{equation}
i.e.~if the distances from $\partial^h \RR $ to $v$ and $w$ are essentially the same.

Consider a sufficiently big $K\gg 1$ that dominates $O(1)$ in~\eqref{eq:1:thm:parabolicfjords} but with $\log \frac{\dist(x,y)}{M\length_{m+1}}\gg K$. Then we can split the interval $[x,y]$ as a concatenation 
\begin{equation}
    \label{eq:depth:K}
[x,y]=A\cup Q \cup B,\qquad \text{with $ |A|=|B|$ and} \ \Width^+(A,B)=K.
\end{equation}
We say that the interval $Q$ and the associated fjord $\Fjord(Q)$ (see~\S\ref{sss:F^*}) are obtained by submerging into depth $K$.

\begin{lem}[Essentially {\cite[Lemma 4.12]{DL22}}]
\label{thm:center:rect}
Let us fix any $K\gg L \gg 1$. We claim that there are three possibilities, 
\begin{enumerate}[label=(\alph*)]
    \item\label{item:1:ParabolicF} the interval $[x,y]$ does not support external families of width $\geq K$ with combinatorial separation: there are no $I, J \subset [x,y]$ such that
    \[\Width^+(I,J)\ge K, \hspace{1cm} \text{ and }\ |I|=|J|\le \frac 12\dist(I,J).\]
    If this possibility occurs, then we have $\log \frac{\dist(x,y)}{M\length_{m+1}}\preceq K$.
\end{enumerate}
Conversely, if $\log \frac{\dist(x,y)}{M\length_{m+1}}\gg K$, then we have another two possibilities: 
\begin{enumerate}[label=(\alph*),start=2]
    \item\label{item:2:ParabolicF} 
      by submerging into depth $K$ as in~\eqref{eq:depth:K}, we obtain a balanced central geodesic rectangle with width $L$; namely: there is a decomposition $[x,y]=A\cup I \cup P \cup J \cup B$ into concatenation of intervals such that 
      \[\Width^+(A,B)=K, \qquad |A|=|B|,\]
      \[\Width^+(I,J)=L, \qquad |I|=|J|,\]
     \[|P|\ge 6(|A|+|I|+|J|+|B|)=12(|A|+|I|),\]
     and geodesic rectangle $\RR=\RR(I,J)$ connecting $I$ and $J$ is central~\eqref{eq:dfn:CentrRect}; 
    \item\label{item:3:ParabolicF} there exists an interval $E$ with the exponential boost (see~\S\ref{sss:ExponBoost}): \[\log \left(\Width^{+,\ver}(E)+\Width^{+,\per}_{\lambda,\div, m}(E)\right)\succeq K - \log(\lambda).\]
\end{enumerate}
\end{lem}
\begin{proof}
Suppose case~\ref{item:1:ParabolicF} does not occur.
If $\log \frac{N}{M} \le K$, then the interval \[I = [x+e^K M\length_{m+1}, x+e^{K+L}M\length_{m+1}] \sp \text{ and } \sp J= [y-e^{K+L}M\length_{m+1}, y-e^KM\length_{m+1}]\] are balanced central intervals, and the corresponding geodesic rectangle has width $ L+O(1)$ by Theorem~\ref{thm:parabolicfjords} (Log-Rule), Claim  \ref{item:2:thm:parabolicfjords}.
Therefore, we are in case~\ref{item:2:ParabolicF}.

On the other hand, if $\log \frac{N}{M} \ge K$, then applying the $\log$ to the main estimate in Claim \ref{item:4:thm:parabolicfjords} of Theorem~\ref{thm:parabolicfjords}, we obtain 
\begin{multline*}    
\qquad \log \left(\Width^{+,\ver}(E)+\Width^{+,\per}_{\lambda,\div, m}(E)\right)\ge  \log \left(\frac{N}{\lambda M}\right)-O(1)\ \succeq \\ \log\left(\frac{N}{ M}\right)-O(\log \lambda)\ge  K -O( \log(\lambda)),\qquad\end{multline*}    
 and we are in case~\ref{item:3:ParabolicF}.
\end{proof}

\subsection{Bistable part $\Xi^k_T$ under protection}\label{ss:bistablepartunderprotection}
It follows immediately from the Log-Rule that the Fjord obtained by submerging into depth $K$ has large stability.
\begin{lem}[{$\mathfrak F(Q)) \subset \Xi^{k}_T$}]\label{lem:geodesicpseudoSiegel} There is an exponent $a>1$ with the following property. If $K\gg 1$ is sufficiently big, then the Fjord $\mathfrak F(Q))$ obtained by submerging into depth $K$ as in~\eqref{eq:depth:K} is within bistable part $\Xi^{k}_T$ of the fjord $\mathfrak F(T)$ (see~\S\ref{sss:F:bistab}), where $k=\lfloor a^K\rfloor$.
\end{lem}
\begin{proof}
We select $a>1$ to be sufficiently small depending on the comparison ``$\asymp$'' in the Log-Rule of Theorem~\ref{thm:parabolicfjords} and assume that $K\gg 1$ is sufficiently big to dominate $O(1)$ in~\eqref{eq:1:thm:parabolicfjords}. Then the family $\Fam^+(I,J)$ representing $\Width^+(I,J)$ in~\eqref{eq:depth:K} contains a geodesic rectangle $\RR$ with width $K-O(1)$.

    Divide the rectangle $\RR$ into four rectangles $\RR^i, i = 1, 2, 3, 4$, with \[\Width(\RR^1)=K-O(1)\quad \text{ and }\Width(\RR^2), \Width(\RR^3), \Width(\RR^4)\ge 5\] so that $\RR^{i+1}$ is nested inside of $\RR^i$, $i = 1, 2, 3$. 
    We assume that $\RR^i$ has horizontal boundary $I^i, J^i$. By the Log-Rule and potentially making $a$ smaller, we have that $|I^2|,|J^2|\ge (a+\varepsilon)^{K-O(1)}\ge 2k$. 

    Let $\RR(I^i, J^i), i =3, 4$ be the geodesic rectangle in $V\setminus\overline{Z}$ connecting $I^i$ and $J^i$. Then $\Width(\RR(I^i, J^i)) \geq \Width(\RR^i) - 1 \geq 4$ (see \cite[Lemma A.5]{DL22}).
    Since $|I^2|,|J^2|\ge 2k$ $\RR(I^i, J^i) \subseteq \mathfrak F^{2k}_\stab(T)$ for $i=3, 4$.
    
    Suppose for contradiction that $\Fjord(Q) \not\subset \Xi^{k}_T = F^{-k\qq_{m+1}}_T(\mathfrak F^{2k}_\stab(T))$.
    Then the pullback $F^{-k\qq_{m+1}}_T(\RR(I^3, J^3))$ must cross $\RR(I^4, J^4)$. This is not possible as wide rectangles cannot cross.
\end{proof}

\subsection{Welding of $\wZ^{m+1}$ and parabolic fjords into $\wZ^m$}
\label{ss:WeildingLmm} 

As in~\cite[\S~5.1.9]{DL22}, we say that a regularization $\wZ^{m}=\wZ^{m+1} \cup \bigcup_{I} \widehat{\mathfrak F}_I$ is within $\orb_{-\qq_{m+1}+1} \RR$ if all relevant objects are within the backward orbit of a rectangle $\RR$. In particular, we require that $\beta_I$ together with protection $A_I\setminus \wZ^{m+1}$ and extra protection $\XX_I$ are in $\RR$. Then $\wZ^m$ is obtained from $\wZ^{m+1}$ by spreading around the parabolic fjord $\widehat{\mathfrak F}_I$ below $\beta_I$.

The following regularization is a restatement of \cite[Corollary~7.3]{DL22} with simplified notations. The proof is the same: it is non-dynamical and relies only on the \cite[Welding Lemma~7.1]{DL22}, which itself relies only on
the Log-Rule in $\wZ^{m+1}$ and in parabolic fjords.  
For convenience, we comment on how the Welding Lemma implies the regularization $\wZ^{m+1}\leadsto \wZ^m$ in~\S\ref{sss:comments:thm:regul}. 

\begin{figure}
   \includegraphics[width=10cm]{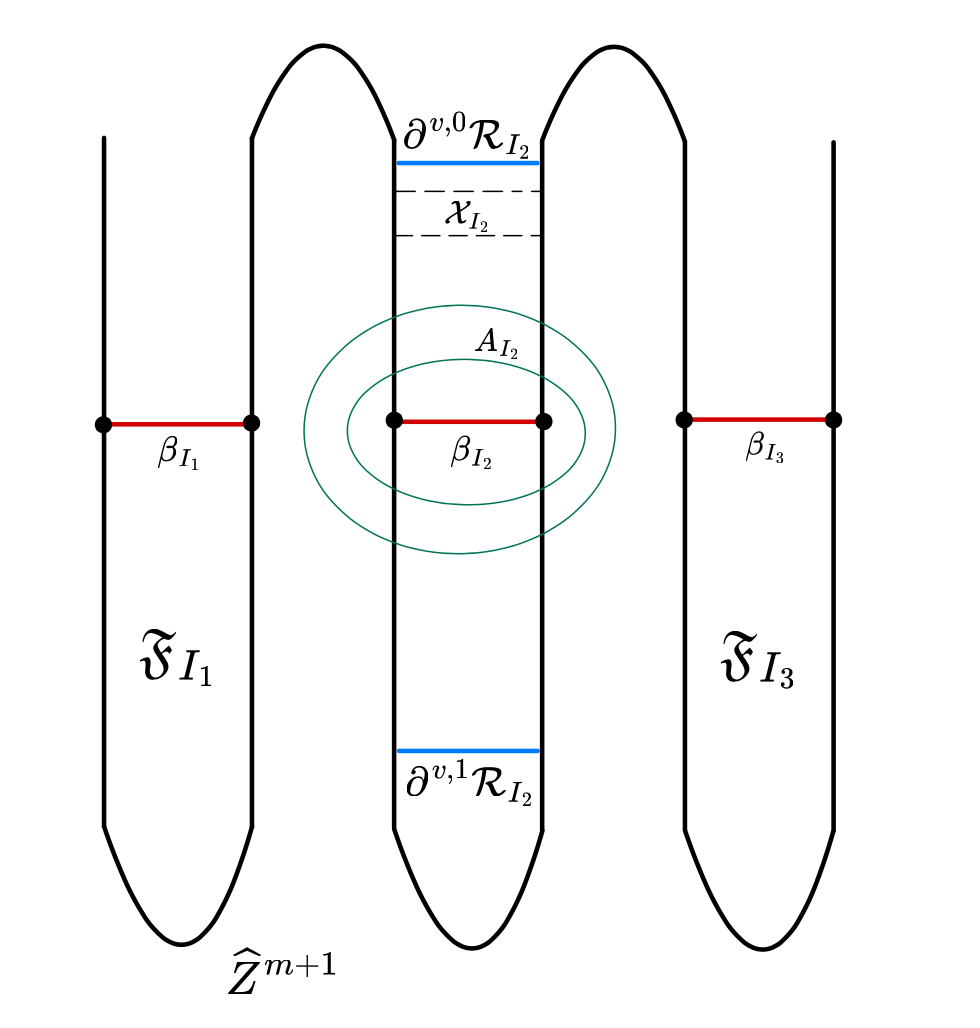}
   \caption{An illustration of the regularization $\wZ^{m+1}\leadsto \wZ^m$ in the rectangle $\RR$.}
   \label{Fig:Welding}
\end{figure}

\begin{thm}[Regularization $\wZ^{m+1}\leadsto \wZ^{m}$]
\label{thm:regul} 
Let $\RR$ be an external parabolic rectangle based on $T'$ with $\Width(\RR)\gg 1$. Let $\wZ^{m+1}$ be a geodesic pseudo-Siegel disk. 
Then:
\begin{enumerate}
\item either there is a geodesic pseudo-Siegel disk $\wZ^{m}=\wZ^{m+1} \cup \bigcup_{I} \widehat{\mathfrak F}_I$ with its level-$m$ regularization within $\orb_{-\qq_{m+1}+1} \RR$;
\label{case:1:thm:regul}
\item\label{case:2:thm:regul} or there is an interval \begin{multline*}
\qquad\quad    I\subset T',\sp\sp   |I|> \length_{m+1} \sp\sp \text{such that}\sp\sp \\
\log\Big(\Width^{+,\ver}(I)+ \Width^{+,\per}_{\lambda,\div, m}(I)\Big)\succeq \Width(\RR);\qquad\qquad
\end{multline*}

\item or there is a grounded rel $\wZ^{m+1}$ interval \[I\subset \partial Z\sp\sp \text{ with }\sp |I|\le \length_{m+1} \sp \text{ such that }\sp\log \Width^+_\lambda(I)\succeq  \Width(\RR).\]
\label{case:3:thm:regul}
\end{enumerate}  
\end{thm}

\begin{rem}[Parameter $\bdelta$] We remark that Theorem~\ref{thm:regul} implicitly fixes a parameter $\bdelta$, see \cite[Remark~7.5]{DL22}; this is a key parameter controlling the non-invariance of $\wZ^m$ see \cite[Assumption 2 in~\S5.1]{DL22}. 
To simplify the exposition, an explicit reference to $\bdelta$ is omitted in this paper. 
\end{rem}

\subsubsection{Comments to Theorem~\ref{thm:regul}} \label{sss:comments:thm:regul}

As we have already mentioned, \cite[Welding Lemma~7.1]{DL22} relies only on the Log-Rule in $\wZ^{m+1}$ and in parabolic fjords in the form of~\eqref{eq:2:thm:parabolicfjords}; therefore, as a consequence of Theorem~\ref{thm:parabolicfjords}, the Welding Lemma is also applicable in our case of $\psi^\bullet$-ql maps. It ultimately implies a fundamental property of $\wZ^m$ that annuli $A_I$ have positive moduli: $\mod(A_I)\ge \delta$; see~\cite[Assumption 2 in \S5]{DL22}. Theorem~\ref{thm:regul} follows from the Welding Lemma by implementing the following steps, as in~\cite[\S7.2]{DL22}: 
\begin{enumerate}
    \item First, we can remove $\asymp 1$ outer buffer from $\RR$ to obtain $\RR^\new$ and claim Log-Rule below the buffer by Theorem \ref{thm:parabolicfjords}. In particular, the $\RR^\new$ obeys the Log-Rule stated in~\eqref{eq:1:thm:parabolicfjords}.
    \item We further remove from $\RR^\new$ outer-buffer with width $\frac 12 \Width(\RR^\new)$. If $\RR^\New$ is not central, then Case~\ref{case:2:thm:regul} occurs by Lemma~\ref{thm:center:rect}, where $I\coloneqq E$.
    \item By the Welding Lemma (as in~\cite[Lemma~7.1]{DL22}), either we have regularization $\wZ^m=Z^m\cup \wZ^{m+1}$ by constructing dams and protecting annuli in $\RR^\New$, and thus in Case~\ref{case:1:thm:regul}, or we have Case~\ref{case:3:thm:regul}.
\end{enumerate}

\subsubsection{Exponential Boosts}\label{sss:ExponBoost} 
We refer to Case~\ref{item:3:ParabolicF} of Lemma~\ref{thm:center:rect} and to Cases~\eqref{case:2:thm:regul}, \eqref{case:3:thm:regul}   of Theorem~\ref{thm:regul} as \emph{Exponential Boosts}. Roughly, if a pathology in one of these cases occurs, then the Log-Rule implies an exponential amplification of the degeneration.

\section{Covering and Lair lemmas}\label{sss:AppC:proof:Sect8} 

In this section, we state Amplification Theorem~\ref{thm:SpreadingAround}, a modification of \cite[Amplification Theorem 8.1]{DL22}. The main difference is that the degeneration for $\psi^\bullet$-maps can be vertical; this leads to Alternative~\ref{item:2:thm:SpreadingAround} in Theorem~\ref{thm:SpreadingAround}.

Amplification is achieved in two steps: spreading around using the Covering Lemma, followed by the Snake-Lair Lemma to detect a large amount of submergence and amplify the degeneration at a deeper level. The Covering Lemma step is stated as Lemma~\ref{lem:SpredAroundWidth}; it is almost the same as for quadratic polynomials and relies on push-forwards that were discussed in \S\ref{ss:push forw curves},~\S\ref{sss:CL:wZ^m}.

The Snake Lemma step is stated as Lemma \ref{lem:Hive Lemma}, where  Alternative~\ref{SLconclusion:2} is new compared to the case of quadratic polynomials. The strategy of the proof is to show if \ref{SLconclusion:2} does not hold, then most degenerations are non-vertical, and we can reduce to the same argument as in the quadratic polynomial case.

Let $F=(f,\iota)\colon \sp ( U, \overline Z_U)  \rightrightarrows (V, \overline Z )$ be $\psi^\bullet$-ql Siegel map with rotation number $\theta_0$.
Recall that the width of $F$ is 
$$
K_F = \Width^\bullet(F)\coloneqq \Width(V\setminus \overline Z).
$$

\begin{ampthm}
\label{thm:SpreadingAround}
There are increasing functions \[\blambda_\bbt,\sp  \bK_\bbt \sp\sp \text{ for }\tt>1\sp\sp\text{ with }\sp \blambda_\bbt,  \ \bK_\bbt\ \underset{t\to \infty}\longrightarrow \ +\infty \] such that the following holds. Suppose that there is a combinatorial interval
\[ I\subset \partial Z\sp\sp \text{ such that }\sp  \Width^+_{\blambda_\bbt}(I) \eqqcolon K \ge \bK_\bbt \sp \text{ and }\sp |I|\le |\theta_0|/(2\blambda_\bbt). 
\]
Consider a geodesic pseudo-Siegel disk $\wZ^m$, where $m$ is the level of $I$.  Then 
\begin{enumerate}[label=(\roman*)]
\item either there is a grounded rel $\wZ^m$ interval \[J\subset \partial Z \sp\sp \text{ such that } \sp \Width^+_{\blambda_\bbt} (J) \ge \bbt K \sp \text{ and }\sp |J|\le |I|.
\]
\item \label{item:2:thm:SpreadingAround} or $K_F\succeq_{\blambda_\bbt} K \qq_{m+1}$. 
\end{enumerate}
\end{ampthm}
\begin{proof}
   As in~\cite[\S8]{DL22}, the proof follows from Lemma~\ref{lem:SpredAroundWidth} (application of the Covering Lemma) followed by Snake-Lair Lemma~\ref{lem:Hive Lemma}.  

   For convenience, as in~\cite[end of \S8.2]{DL22}, we can assume that one of the endpoints of $I$ is in $\CP_m$. Indeed, we cover $I$ with two combinatorial intervals $I_a\cup I_b\supset I$ satisfying such a property and observe that $ \Width^+_{\blambda_\bbt}(I) =K$ implies that either 
 \begin{equation}
     \label{eq:trick:I:I_a}
        \Width^+_{\blambda_\bbt-1}(I_a) \ge K/2\qquad \text{ or }\qquad    \Width^+_{\blambda_\bbt-1}(I_b) \ge K/2 
 \end{equation}  
 holds.
\end{proof}
\subsection{Applying the Covering Lemma}\label{ss:ApplyingCoveringLemma} We will denote by $I^m$ the projection of a regular interval $I\subset \partial Z$ onto $\partial \wZ^m$.
The same argument as in \cite[Lemma 8.2]{DL22} allows us to apply the Covering Lemma (see \cite{KL} or \cite[Lemma 8.5]{DL22}) and obtain the following theorem for $\psi^\bullet$-maps.
See \S\ref{subsubsec:simplepush},~\S\ref{subsubsec:near-degeneratepush}, ~\S\ref{sss:CL:wZ^m} and Figure~\ref{Fig:psi:PushForward} for the preparation to apply the covering lemma.

\begin{lem}
\label{lem:SpredAroundWidth}
For every $\bkappa>1$ and $\lambda>10$, there is $\bK_{\lambda,\bkappa}>1$ and $C_\bkappa$ (independent of $\lambda$) such that the following holds. Suppose that there is a combinatorial interval \[I\subset\partial Z\sp\sp \text{ such that }
\sp\Width^+_{\lambda+2}(I)=K \ge \bK_{\lambda,\bkappa},\sp\sp |I|\le |\theta_0| / (2\lambda+4),\sp\sp m =\Level(I)\]
and such that one of the endpoints of $I$ is in $\CP_m$. Let $\wZ^m$ be a geodesic pseudo-Siegel disk, and \[ I_s\subset \partial Z,\sp\sp I_s=f^{i_s}(I),\sp\sp s\in \{0,1,\dots, \qq_{m+1}-1\}\] be the intervals obtained by spreading around $I=I_0$ (see \cite[\S 2.1.5]{DL22}). Then every interval $I_s$ is well-grounded rel $\wZ^{m}$ and its projection $I^m_s\subset \partial \wZ^m$ is  
\begin{enumerate}
\item\label{item:1:lem:SpredAroundWidth} either $[\bkappa K,10]$-wide: $\Width_{10}(I^m_s)\ge \bkappa K$ 
\item\label{item:2:lem:SpredAroundWidth} or $[ C_{\bkappa} K,\lambda]$-wide: $\Width_{\lambda}(I^m_s)\ge C_{\bkappa} K$ 
\end{enumerate} 
\end{lem}

\subsection{Lair of snakes}
For our convenience, we enumerate intervals clockwise in the following lemma. We remark that the Snake-Lair lemma for $\psi^\bullet$-maps needs to be modified. In particular, we have Alternative~\ref{SLconclusion:2} in the conclusion (c.f. \cite[Lemma 8.6]{DL22}). 
\begin{lairlmm}
\label{lem:Hive Lemma}
For every $\bbt>2$ there are $\bkappa,\blambda,\bK\gg_{\bdelta}1$ such that the following holds. Suppose that $\wZ^m$ is a pseudo-Siegel disk with $\length_m< \blambda/4$. Let \[I_{n+1}=I_0, I_1,\dots , I_n\subset \partial \wZ^m, \sp\sp\sp |I_k| = \length_m,\sp  \dist(I_k, I_{k+1})\le \length_m\]
be a sequence of well-grounded intervals enumerated clockwise such that every $I_s$ is one of the following two {\bf types}
\begin{enumerate}
\item either $I_s$ is $[\bkappa K,10]$-wide, \label{case:I_k:1}
\item or $I_s$ is $[ C_{\bkappa} K,\blambda]$-wide,\label{case:I_k:2}
\end{enumerate} 
where $K\ge \bK$ and $C_\bkappa$ is a constant (from Lemma~\ref{lem:SpredAroundWidth}) independent of $\blambda$. Then 
   \begin{enumerate}[label=(\roman*)]
    \item either there is a grounded rel $\wZ^m$ interval \[J\subset \partial Z \sp\sp \text{ such that } \sp \Width^+_{\blambda_\bbt} (J) \ge \bbt K \sp \text{ and }\sp |J|\le |I|.\label{SLconclusion:1}
\]
\item or $K_F \succeq_\blambda K \qq_{m+1}$. \label{SLconclusion:2}
\end{enumerate}
\end{lairlmm}

\noindent The proof below follows the strategy of \cite[Lemma 8.6]{DL22}. In the case of quadratic polynomials, degeneration can not go to infinity \cite[Lemma~6.14]{DL22}. Claim~\ref {claim:1} is a substitution for \cite[Lemma~6.14]{DL22}: if Alternative~\ref{SLconclusion:2} does not hold, then there is a long sequence of intervals from where the degeneration does not go to the outer boundary $\partial V$.  With respect to the notion of ``relatively peripheral'' curves, Items~\ref{case:I_k:1} and~\ref{case:I_k:2} of Lemma~\ref{lem:Hive Lemma} take a form similar to that in \cite[Lemma 8.6]{DL22}, see Claim~\ref{claim:2} below. After Claim~\ref{claim:2}, the proof is similar to the case of quadratic polynomials.

\begin{proof}
    Let us assume that the alternative \ref{SLconclusion:2} of Lemma \ref{lem:Hive Lemma} does not hold. We will show that the alternative \ref{SLconclusion:1} of Lemma \ref{lem:Hive Lemma} must hold.

    Let us enlarge every $I_i$ into a well-grounded interval $\wI_i\subset \partial \wZ^m$ by adding to $I_i$ the interval between $I_i$ and $I_{i+1}$ if $I_i$ and $I_{i+1}$ are disjoint. Since the distances between the $I_i$ and $I_{i+1}$ are $\le \length_{m}$, \footnote{In many applications such as spreading around $I$, the distances between $I_i$ and $I_{i+1}$ are $\le \length_{m+1}$, see \cite[\S 2.1.4]{DL22}.} we have $|\wI_i|\le 2 \length_m$.

    Let $\mathcal{I}$ consist of all intervals $\wI_i$ so that $\Width^{+, ver}(\wI_i) \geq \frac{K}{\blambda^3}$.
    Suppose $\#\mathcal{I} \geq \frac{n}{\blambda}$. Since $n \asymp \qq_{m+1}$, we have 
    $$
    K_F \succeq \frac{K}{\blambda^4} \qq_{m+1}.
    $$
    Therefore, the alternative \ref{SLconclusion:2} of Lemma \ref{lem:Hive Lemma} holds.
    Thus, we may assume that $\#\mathcal{I} \leq \frac{n}{\blambda}$, i.e., not a lot of intervals have large vertical degenerations.

    Let us assume that $\#\mathcal{I} \geq 2$. The case $\#\mathcal{I} = 1$ can be treated similarly by allowing $\wI_k=\wI_{k+l}$ in Claim~\ref{claim:1} below. The case $\#\mathcal{I} = 0$ is essentially the case of quadratic polynomials -- the outer boundary $\partial V$ is well separated from $\wZ^m$. 
    \begin{claim}[Substitution of {\cite[Lemma~6.14]{DL22}}; see Figure~\ref{Fig:RelativeV}]
    \label{claim:1}
        There is a sub-sequence 
    \begin{equation}
    \label{eq:seq:I_k  I_k+l}
    \wI_{k}, \wI_{k+1},\dots, \wI_{k+l}\subset\partial \wZ^m, \sp\sp\sp |\wI|\le 2 \length_m
    \end{equation}
    such that 
    \begin{enumerate}
    \item $l \geq \blambda/2$,
    \item $\Width^{+, ver}(\wI_k), \Width^{+, ver}(\wI_{k+l}) \geq \frac{K}{\blambda^3}$,
    \item $\Width^{+, ver}(\wI_j) \leq \frac{K}{\blambda^3}$ for all $k< j < k+l$.
    \end{enumerate} 
    \end{claim}
    \begin{proof}
        This claim follows from $2\leq \# \mathcal{I} \leq \frac{n}{\blambda}$.
    \end{proof}

Let $L = \lfloor \wI_{k+1}, \wI_{k+l-1}\rfloor$.
Let $I \subseteq L$. We say an arc $\alpha \in \mathcal{F}_\lambda(I)$ is {\em relatively peripheral} with respect to $L$ if it is either peripheral, or it is vertical and the last time $\alpha$ intersects $\wZ^m$ is outside of $L$.
We denote by $\Width^{per/L}_{\lambda}(I)$ the extremal width of relatively peripheral arcs.
Since relatively peripheral arc outside of $\wZ^m$ is in fact peripheral, we have
$$
\Width^{+, per/L}_{\lambda}(I) = \Width^{+, per}_{\lambda}(I).
$$
Let us also denote by $\mathcal{R}_k$ and $\mathcal{R}_{k+l}$ the rectangles connecting $I_k$ and $\partial V$ with width $\Width(\mathcal{R}_k), \Width(\mathcal{R}_{k+l}) \succeq \frac{K}{\blambda^3}$.
\begin{figure}
   \includegraphics[width=10cm]{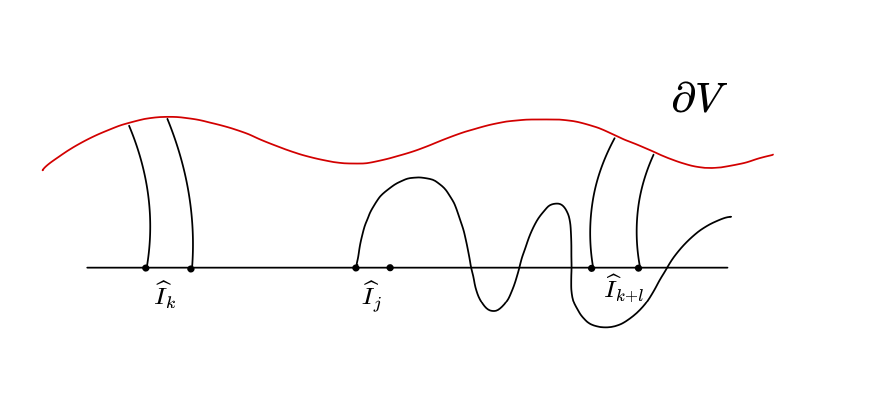}
   \caption{Illustration to the intervals $\widehat I_j$ from Claim~\ref{claim:1} and to families emerging from the $\widehat I_J$.}
   \label{Fig:RelativeV}
\end{figure}

\begin{claim}[c.f. {\cite[(1) and (2) of Lemma 8.6]{DL22}}]    \label{claim:2}

    For any $k < i < k+l$, we have that $\wI_i$ is one of the following two types
    \begin{enumerate}[label=(\Roman*)]
        \item \label{case:relI:1} either $\Width^{per/L}_{10}(\wI_i) \geq \bkappa K/2$,
        \item \label{case:relI:2} or $\Width^{per/L}_{\blambda}(\wI_i) \geq C_\bkappa K/2$
    \end{enumerate}
\end{claim}
\begin{proof}
    By Claim \ref{claim:1}, $\Width^{+, ver}(\wI_j) \leq \frac{K}{\blambda^3}$ for all $k< j < k+l$. Therefore, the non relatively peripheral degeneration
    $$
    \Width_{\blambda}(\wI_i) - \Width^{per/L}_{\blambda}(\wI_i) \leq \blambda \frac{K}{\blambda^3} = \frac{K}{\blambda^2}, \text { and }
    $$
    $$
    \Width_{10}(\wI_i) - \Width^{per/L}_{10}(\wI_i) \leq 10 \frac{K}{\blambda^3} = \frac{10K}{\blambda^3}.
    $$
    Since $1/ \blambda \ll 1$, the claim follows.
\end{proof}

\begin{claim}(c.f. Claim 1 in \cite[\S 8.2]{DL22})\label{cl:seq:Type1}
    Let $k+20 < i < k+l - 20$. Lemma \ref{lem:Hive Lemma} holds if there is a Type~\ref{case:relI:1} interval $\wI_i$ such that
    $$
    \Width^{per/L}_{10}(\wI_i) - \Width^{+, per}_{10}(\wI_i) \geq \bkappa K/4.
    $$
\end{claim}
\begin{proof}
    Since $k+20 < i < k+l - 20$, we conclude that $10 \wI_i \subseteq L$. Thus, $\Width^{per/L}_{10}(\wI_i) = \Width^{per}_{10}(\wI_i)$. Therefore, the claim follows immediately from \cite[Lemma 6.8]{DL22}.
\end{proof}
Therefore, we may assume that any Type~\ref{case:relI:1} interval $\wI_i$ for $k+20 < i < k+l - 20$ satisfies 
$$
\Width^{+, per}_{10}(\wI_i) \geq \bkappa K/4.
$$

\begin{claim}(c.f. Claim 2 in \cite[\S 8.2]{DL22})
\label{cl:seq:I_i}
There is a sub-sequence \begin{equation}
\label{eq:seq:I_i  I_i+lambda}
\wI_{i}, \wI_{i+1},\dots, \wI_{i+\blambda/20}\subset\partial \wZ^m,
\end{equation}
such that $[i, i+\blambda/20] \subseteq [k+20, k+l-20]$, and for
every interval $\wI_j$ in~\eqref{eq:seq:I_i  I_i+lambda}, we have $\Width^{+,\per}_{\blambda/4}(\wI_j)\leq 5$.
\end{claim} 
\begin{proof}
Suppose for contradiction that $(\wI_j)$ has a sub-sequence $(L_j), j = 0,..., t$ with 
\begin{itemize}
    \item $L_0 = \wI_k$, $L_t = \wI_{k+l}$, and 
    \item $L_j=\wI_{\ell(j)}$ such that $\ell(j)<\ell(j+1)\le \ell(j)+\blambda/10$, and
    \item $\Width_{\blambda/4}^{+, per}(L_j) \geq 5$ for $j = 1,..., t-1$.
\end{itemize}
Since $\Width(\mathcal{R}_k), \Width(\mathcal{R}_{k+l}) \succeq \frac{K}{\blambda^3} \geq 5$, most of the curves in $\mathcal{F}^{+, per}_{\blambda/4}(\wI_j)$ do not cross the rectangles $\mathcal{R}_k$ and $\mathcal{R}_{k+l}$.
Thus, by \cite[Lemma 6.14]{DL22}, such families $\mathcal{F}^{+, per}_{\blambda/4}(\wI_j), j = 1,..., t-1$ would block each other, which is a contradiction and the claim follows.
\end{proof}

\begin{claim}(c.f. Claim 3 in \cite[\S 8.2]{DL22})
There is $\mathbf{k}$ depending on $\mathbf{t}$ and $\bkappa$ but not on $\blambda$ such that the alternative \ref{SLconclusion:1} of Lemma \ref{lem:Hive Lemma} holds if there are $\mathbf{k}$ consecutive Type~\eqref{case:relI:1} intervals in~\eqref{eq:seq:I_i  I_i+lambda}
\end{claim}
\begin{proof}
    Suppose that there are $\mathbf{k}$ consecutive intervals of Type~\eqref{case:relI:1}, denoted by $\wI_a,..., \wI_b$ with $b-a = \mathbf{k}-1$. Denote by $$
    I_{a,b}\coloneqq \lfloor \wI_{a}, \wI_{b}\rfloor, \text{ and }
    J_{a,b}\coloneqq \bigcap_{i=a}^b (\frac{\blambda}{4}\wI_i)^c.
    $$
    Then we have $W^+(I_{a,b}, J_{a,b}) = O(\mathbf{k})$ by Claim~\ref{cl:seq:I_i}.
    Since families $\mathcal{F}^{+, per}_{10}(\wI_i)$ have small overlaps (they block each other), by Claim~\ref{cl:seq:Type1}, there is a rectangle
    $$
    \mathcal{R} \subseteq \bigcup_{i=a}^b\mathcal{F}^{+, per}_{10}(\wI_i) \text{ with } \Width(\mathcal{R}) \succeq \mathbf{k}\bkappa K.
    $$
    The proof now proceeds exactly the same as the proof of Claim 3 in \cite[\S 8.2]{DL22}.
\end{proof}

We now assume that among $\mathbf{k}$ consecutive intervals in~\eqref{eq:seq:I_i  I_i+lambda}, there is at least one Type~\eqref{case:relI:2} interval.
Now we can perform the similar construction as in the quadratic polynomial case, and enumerate Type~\eqref{case:relI:2} intervals in~\eqref{eq:seq:I_i  I_i+lambda} as
$$
\wI_{i_0},..., \wI_{i_{s}}, i_j < i_{j+1} < i_j + \mathbf{k},
$$
where $s \geq \blambda/(22\mathbf{k})$.

We enlarge each $\wI_{i_t}$ to well grounded intervals $\widetilde I_{i_t} \supseteq \wI_{i_t}$ such that
\begin{itemize}
    \item $\widetilde I_{i_t}$ ends where $\wI_{i_{t+1}}$ starts; and
    \item either $\widetilde I_{i_0}$ (respectively $\widetilde I_{i_s}$) contains $\wI_k$ (respectively $\wI_{k+l}$) or $\widetilde I_{i_0}$ (respectively $\widetilde I_{i_s}$) has length between $\frac{\blambda}{2} \length_m$ and $(\frac{\blambda}{2} +2) \length_m$;
\end{itemize}

Let $B\coloneqq \lfloor \widetilde I_{i_0}, \widetilde I_{i_s}\rfloor$.
Note that $\Width(\mathcal{R}_k), \Width(\mathcal{R}_{k+l}) \succeq \frac{K}{\blambda^3} \geq 5$, no wide family of curves cross these two rectangles.
Together with Claim \ref{cl:seq:I_i}, we have
$$
\Width^+(\wI_j, B^c) = O(1)
$$
if $\wI_j \subseteq \widetilde I_{i_t}$ for some $t \in \{1,..., s-1\}$.
Since $\widetilde I_{i_t}$ contains at most $\mathbf{k}$ consecutive intervals $\wI_i$,
we have the following claim.
\begin{claim}(c.f. Claim 4 in \cite[\S 8.2]{DL22})
For any $t \in \{1,..., s-1\}$, we have
$$
\Width^+(\widetilde I_{i_t}, B^c) = O(\mathbf{k}).
$$
\end{claim}

With the modifications above, the argument proceeds the same as in \cite[Lemma 8.6]{DL22} (see, more precisely, Claim 5 and Claim 6 in \cite[\S 8.2]{DL22}), and we conclude that the alternative \ref{SLconclusion:1} of Lemma \ref{lem:Hive Lemma} must hold.
\end{proof}

\section{Calibration Lemma and Combinatorial Localization}\label{sec:calibration}
In this section, we explain how the Calibration Lemma \cite{DL22} for quadratic polynomials is modified for $\psi^\bullet$ maps, see \S\ref{ss:calibrationlem}. We will then introduce the Combinatorial Localization Lemma in \S\ref{ss:comblocal}, which will be an essential ingredient for the equidistribution at shallow levels \S\ref{s:Proof}.

Let us remark that the proof of \cite[Calibration Lemma 9.1]{DL22} relies on the following steps:
\begin{enumerate}[label=\text{(\Alph*)},font=\normalfont,leftmargin=*]
    \item\label{cal_lmm:step:1} univalent pushforwards to spread around $\Fam_{\lambda,\div, m}^+(I)\cup \Fam^{+,\ver}(I)$;
    \item \label{cal_lmm:step:2} removing combinatorial buffers of width $O(K)$ to produce almost invariant families; 
    \item \label{cal_lmm:step:3}  the ``$1\le 1\oplus 1=0.5$" argument illustrated in \cite[Figure~16]{DL22}; it implies that an almost invariant family traverses a bubble in a localized pattern; this leads to a boost in degeneration towards a deeper scale~\cite[Claim 6 in \S9]{DL22}.
\end{enumerate}
In the case of quadratic polynomials, spreading around in~\ref{cal_lmm:step:1} does not create wide vertical families. For $\psi^\bullet$-ql maps, spreading around in~\ref{cal_lmm:step:1} leads to Conclusion \ref{Cal:Lmm:Concl:a} of Calibration Lemma~\ref{lmm:CalibrLmm}. Such a conclusion can only occur at shallow levels. On deep levels, Calibration Lemma~\ref{lmm:CalibrLmm} takes a form similar to the quadratic case, see Remark~\ref{rem:CalLmm}. 

Arguments~\ref{cal_lmm:step:2} and~\ref{cal_lmm:step:3} follow the lines of the proof of \cite[Calibration Lemma 9.1]{DL22} and roughly lead to Conclusions~\ref{Cal:Lmm:Concl:b} and~\ref{Cal:Lmm:Concl:c} of Calibration Lemma~\ref{lmm:CalibrLmm} respectively.

\subsection{Calibration Lemma}\label{ss:calibrationlem}
We follow the same notation as in \cite[\S 9]{DL22}. As we have mentioned, the semi-equidistribution Alternative~\ref{Cal:Lmm:Concl:a} is a new addition compared to the case of quadratic polynomials; see also Reamrk~\ref{rem:CalLmm}. We stress that this semi-equidistribution is valid for \emph{all} combinatorial level $m$ intervals.

\begin{caliblem}
\label{lmm:CalibrLmm}
There is an absolute constant $\bchi> 1$ such that the following holds for every $\lambda\ge 10$. Let $\wZ^{m+1}$ be a geodesic pseudo-Siegel disk and consider an interval $T\subset\partial Z$ in the diffeo-tiling $\Dbb_m$.

Assume that there exists an interval \[I\subset T\sp\sp 
\text{ with }\sp\sp \length_{m+1} \le |I|\le \length_{m}
\]
such that
\begin{multline}\label{eq:CalLemm:chiK}
\hspace{2cm}\bchi K := \Width^{+,\per}_{\lambda, \div,m}(I)+\Width^{+,\ver}(I) = \\ =O(1)+\Width^{+}_{\lambda}(I)-\Width^{+,\per}_{\lambda, \ext, m}(I)\gg_{\lambda} 1.\hspace{2cm}\end{multline}
Then 
\begin{enumerate}[label=\text{(\Roman*)},font=\normalfont,leftmargin=*]

\item \label{Cal:Lmm:Concl:a} either the following {\bf~semi-equidistribution} of vertical degeneration holds: for every level $m$ combinatorial interval $J\subset\partial Z$, we have \[\Width^{+,\ver}(J)\ge  K;\] in particular, we have \[K_F\ge K \qq_{m+1}; \text{ or}\] 

\item \label{Cal:Lmm:Concl:b} there is a level-$(m+1)$ combinatorial interval $I$ with $\Width^+_{\lambda}(I) \geq K$; or

\item \label{Cal:Lmm:Concl:c} there is an interval $I'\subset \partial Z,\sp |I'|<\length_{m+1}$ grounded rel $\wZ^{m+1}$ with $\Width^+_{\lambda}(I') \geq \bchi^{1.5} K$.
\end{enumerate}
\end{caliblem}

\begin{rem}\label{rem:CalLmm} If $K < K_F/\qq_{m+1}$, then the statement takes form similar to the \cite[Calibration Lemma 9.1]{DL:sector bounds} for quadratic polynomials: only Item~\ref{Cal:Lmm:Concl:b} and Item~\ref{Cal:Lmm:Concl:c} appear in the outcome.
\end{rem}

\begin{rem}[Almost Invariance in Item~\ref{Cal:Lmm:Concl:a}]\label{rem:CalLmm:AI}
In Item~\ref{Cal:Lmm:Concl:a}, the proof provides a stronger conclusion, see~\S\ref{sss:prf:CL:semi:Equid}: for every interval $T_s\in \Dbb_m$ in the level $m$ diffeo-tiling, there is a subinterval  $I^\new_s\subset T_s$ with $|I^\new_s|\gg \length_{m+1}$ and a vertical lamination $\RR_{I^\new_s}$ emerging from $I^\new_s$ with $\Width(\RR_{I_s^\new})=\bchi K -O(\bchi^{0.9} K)$ such that the $\RR_{I^\new_s}$ are almost invariant up to $O(K)$  under $f^i$ for all $i\le \qq_{m+1}$.
\end{rem}

\begin{proof} As we have mentioned at the beginning of this section, the main ingredients are~\ref{cal_lmm:step:1},~\ref{cal_lmm:step:2}, and~\ref{cal_lmm:step:3}; they roughly lead to Conclusions~\ref{Cal:Lmm:Concl:a},~\ref{Cal:Lmm:Concl:b}, and~\ref{Cal:Lmm:Concl:c} respectively.

The proof below is organized as a sequence of items; each item represents either a notion or a step in the proof. Conclusion~\ref {Cal:Lmm:Concl:b} is encoded in Step~\ref{proofCali:enumer:4}.  Conclusion~\ref {Cal:Lmm:Concl:a} and Conclusions~\ref {Cal:Lmm:Concl:c} are justified in \S\ref{sss:prf:CL:semi:Equid}  and \S\ref{sss:prf:CL:concl:3} respectively.

Overall, the proof follows the case of quadratic polynomials; in many items, we provide a general outline together with a reference to \cite[\S 9]{DL22} for detailed calculation estimates.

\subsubsection{Lamination $\RR$}\label{sss:prf:CL} In Item~\ref{proofCali:enumer1} below, we specify a lamination~$\RR$ representing the degeneration~\eqref{eq:CalLemm:chiK}. The lamination $\RR$ can be viewed as a union of at most $3$ rectangles, see Item~\ref{proofCali:enumer2}.   In Item~\ref{proofCali:enumer:4}, we assume that Conclusion~\ref{Cal:Lmm:Concl:b} of the Calibration Lemma does not hold. Consequently, the sublaminations $\RR_J$ of $\RR$ emerging from $J\subset I$ are almost invariant up to $O(K)$-error, see Item~\ref{proofCali:enumer:AlmInvar}.

\begin{enumerate}[label=\text{(\arabic*)},font=\normalfont,leftmargin=*]
     \item\label{proofCali:enumer1} We select a lamination $\RR \subset \Fam^{+,\per}_{\lambda, \div,m}(I)\cup \Fam^{+,\ver}(I)$ representing most of the width in the selected families:
    \[\Width(\RR)=\Width^{+,\per}_{\lambda, \div,m}(I)+\Width^{+,\ver}(I)- O(1)=\bchi K  -O(1) .\]

    \item\label{proofCali:enumer2} Here and later, by choosing regions bounded by the left- and right-most curves in the same homotopy class, we can replace $\RR$ by a union of at most three rectangles \[\RR^{\per,-}, \RR^\ver, \RR^{\per,+}\ \subset \ \Fam^{+,\per}_{\lambda, \div,m}(I)\cup \Fam^{+,\ver}(I)\] with
   \begin{equation}
       \label{eq:proofCali:enumer2} \Width(\RR^{\per,-})+\Width(\RR^{\ver})+ \Width(\RR^{\per,+})\ge \Width(\RR) -O(1),
   \end{equation} 
    where $\RR^{\per,-}$ represents peripheral curves going left, $\RR^{\ver}$ represents vertical curves, and $\RR^{\per,+}$ represents peripheral curves going right. See \cite[A.1.8]{DL22} and the references within for more details. 

   \item Below we assume that $m>-1$; i.e., there are $\qq_{m+1}>1$ intervals in the diffeo-tiling $\Dbb_m$. The case $m=-1$ is similar and only requires minor adjustments in notation -- there is only one interval in $\Dbb_{-1}$. We enumerate all intervals $T_s\in \Dbb_m$ from left to right so that $T_{s-1}$ is on the left of $T_s$. We write $T=T_0$.

\item\label{proofCali:enumer:RR_J} Following~\cite[\S~9.1.2]{DL22}, for an interval $J\subset I$, we denote by $\RR_J$ the sublamination of $\RR$ consisting of curves in $\RR$ emerging from $J$.
    
    \item\label{proofCali:enumer:4} As in~\cite[\S9.1.2]{DL22}, we assume that Item~\ref{Cal:Lmm:Concl:b} does not hold:
     \begin{equation}
         \label{eq:assump:proofCali:enumer:4} \Width(\RR_J)\le K\qquad\qquad \text{ for every $J\subset I$ with $|J|=\length_{m+1}$}.
     \end{equation}
     
    \item\label{proofCali:enumer:O(K)-buff} By a \emph{combinatorially substantial $O(K)$-buffer} of $\RR$, we mean a buffer of the form $\RR_J$ (i.e., $I\setminus J$ is connected) with \[ \Width(\RR_J)\ge 1, \qquad \Width(\RR_J)=O(K),\qquad \text{ and } \qquad |J|\ge \length_{m+1}.\] By removing combinatorially substantial $O(K)$-buffers, we can assume that the (new) interval $I$ is combinatorially separated from the endpoints of $T$. 
    
    \item\label{proofCali:enumer:AlmInvar}  By the same argument as \cite[Claim 3 in \S~9]{DL22}, the lamination $\RR_J$ with $J\subset I$ is \emph{almost invariant} up to $O(K)$ under $f^{\qq_{m+1}}$: the univalent pushforward of $\RR_J-O(K)$ under $F^{\qq_{m+1}}$ overflows $\RR_J$, where the latter is replaced by three rectangles as in Item~\ref{proofCali:enumer2}. 
    \\ More precisely, there is a sublamination \[\RR^\new_J\subset \RR_J\qquad\text{ with } \qquad \Width(\RR_J^\new)\ge \Width(\RR_J)-O(K),\]
    and there are three rectangles 
    \[\RR_J^{\per,-},\ \RR_J^\ver,\ \RR_J^{\per,+}\ \subset \ \Fam^{+,\per}_{\lambda, \div,m}(I)\cup \Fam^{+,\ver}(I)\] emerging from $J$ and obtained from (and replacing) $\RR_J$ by selecting appropriate left- and right-most curves as in Item~\ref{proofCali:enumer2}
   \[
       \Width(\RR^{\per,-}_J)+\Width(\RR^{\ver}_J)+ \Width(\RR^{\per,+}_J)\ge \Width(\RR_J) -O(1),
   \] 
        such that the univalent pushforward $F^{\qq_{m+1}}_* \RR^\new_J$ overflows their union $\RR_J^{\per,-}\sqcup \RR_J^\ver\sqcup \RR_J^{\per,+}$.
       \end{enumerate}

\begin{proof}[Comment to Item~\ref{proofCali:enumer:AlmInvar}] The required lamination $\RR^\new_J$ is constructed in two steps. First, we remove from $\RR_J$ two combinatorially substantial $O(K)$-buffers, see Item~\ref{proofCali:enumer:O(K)-buff}; the result is a lamination $\RR_{J^\new}$ with $J^\new\subset J$ such that $f^{\qq_{m+1}}(J^\new)\Subset J$ and such that there are still substantial families (of width $\ge 1$) emerging from both components of $J\setminus J^\new$. Therefore, the pushfoward $F^{\qq_{m+1}}_* \RR_{J^\new}$ of $\RR_{J^\new}$ (see~\eqref{eq:dfn:push:RR:J:new} below) must follow the lamination $\RR_J$: by removing $O(1)$-buffers from $F^{\qq_{m+1}}_* \RR_{J^\new}$, the remaining lamination $F^{\qq_{m+1}}_* \RR_{J^\new}-O(1)$ overflows the rectangles $\RR_J^{\per,-}\sqcup \RR_J^\ver\sqcup \RR_J^{\per,+}$. We can now remove preimages of these $O(1)$ buffers in $\RR_{J^\new}$ and construct $\RR_J^\new$ as $\RR_{J^\new}-O(1)$. We refer to \cite[Claim 3 in \S~9]{DL22} for routine details.

Let us briefly recall the definition of the univalent pushforward $F^{\qq_{m+1}}_* \RR_{J^\new}$ from~\S\ref{subsubsec:univalentpush}. Every curve $\gamma\in \RR_{J^\new}$, has a front-subcurve $\gamma'$ such that $\gamma'$ lifts under $\iota^{\qq_{m+1}}$ to $\tilde \gamma'\subset U^{\qq_{m+1}}$. By construction, 
\[\Width(\{\tilde \gamma'\})=\Width(\{\gamma'\})\ge \Width(\{\gamma\})=\Width(\RR_{J^\new}).\]
By removing $O(1)$ curves from the family $\{\tilde \gamma'\}$, the remaining curves are mapped univalently:
\begin{multline}
\label{eq:dfn:push:RR:J:new}\qquad\qquad
F^{\qq_{m+1}}_* \RR_{J^\new} \ \coloneqq \ F^{\qq_{m+1}}\big(\{\tilde \gamma'\} -O(1)\big), \\   \Width\Big(F^{\qq_{m+1}}_* \RR_{J^\new}\Big)=\Width\big(\{\tilde \gamma'\}\big)-O(1).\qquad\qquad
\end{multline}   
\end{proof}

\subsubsection{Spreading around $\RR-O(K)$ into $\RR_i$} In \S\ref{sss:prf:CL}, we studied the invariance properties of the lamination $\RR$ under the first return map $F^{\qq_{m+1}}$. Below, we spread around $\RR-O(K)$ and restate items from~\S\ref{sss:prf:CL} for the respective images $\RR_s$. The almost $F^\qq_{m+1}$-invariance of $\RR_J$ up to $O(K)$-error (see Item~\ref{proofCali:enumer:AlmInvar}) is refined in Item~\ref{proofCali:enumer:AI:refine} to the respective almost invariance under any $F^i$ with $i\le \qq_{m+1}$.

     \begin{enumerate}[label=\text{(\arabic*)},font=\normalfont,leftmargin=*,start=8]
    \item We spread $\RR-O(K)$ around by $f^i$ with $i<\qq_{m+1}$ to construct laminations $\RR_s$ emerging from every interval $T_{s}\in \Dbb_m$ in the diffeo-tiling (see \S~\ref{ss:push forw curves}), where $s=s(i)$. Here $-O(K)$ represents combinatorially substential buffers (Item~\ref{proofCali:enumer:O(K)-buff}) so that the lamination $\RR_s$ can be sent back to $\RR$ (via $F^{\qq_{m+1}-i}$) using Item~\ref{proofCali:enumer:O(K)-buff}. We have 
    \begin{equation}
        \label{eq:RR_s_vs_RR} \Width(\RR_s) = \Width(\RR)-O(K)=\bchi K-O(K).
    \end{equation}
We write $I_s\coloneqq f^i(I)$; all curves in $\RR_s$ are originated on $I_s$.
    
       \item\label{proofCali:enumer:rect:RR_s} As in Item \ref{proofCali:enumer2}, we can replace $\RR_s$ with the union of at most $3$ rectangles $\RR_s^{\per,-}\sqcup\RR^\ver\sqcup \RR^{\per,+}$ with 
    \[\Width(\RR_s^{\per,-})+\Width(\RR^\ver)+\Width(\RR_s^{\per,-})=\Width(\RR_s)-O(1)= \bchi K -O(K).\]

  \item  As in Item~\ref{proofCali:enumer:RR_J}, for an interval $J\subset I_s$, we denote by $\RR_{J}$ the sublamination of $\RR_s$ consisting of curves starting at $J$. 

 \item Since families emerging from subintervals in $I_s$ can be univalently be brought back to $\RR$, the assumption~\eqref{eq:assump:proofCali:enumer:4} in Item~\ref{proofCali:enumer:4} implies
     \begin{equation}
         \label{eq:assump:proofCali:enumer:4:2} \Width(\RR_J)\le K+O(1)\qquad\qquad \text{ for every $J\subset I_0$ with $|J|=\length_{m+1}$}.
     \end{equation}

\item \label{proofCali:enumer:AI:refine} Item~\ref{proofCali:enumer:AlmInvar} can be now refined as follows. Consider $J\subset I_s$ and its image $f^{i}(J)$ under $i\le \qq_{m+1}$. Then most of the interval $f^{i}(J)$ is in some $I_q$; i.e., $f^{i}(J)\setminus I_q$ has length $\le \length_{m+1}$. We write $J_q\coloneqq f^{i}(J)\cap I_q$.

\noindent Then the univalent pushforward of $\RR_J-O(K)$ under $F^{i}$ is $\RR_{J_q}-O(K)$. In particular, $F^{i}_*(\RR_J-O(K))$ overflows the three rectangles representing $\RR_{J_q}-O(K)$; compare with Item~\ref{proofCali:enumer2}.

\noindent The justification is the same as for Item~\ref{proofCali:enumer:AlmInvar} by factorizing 
\begin{multline*}
\qquad\qquad  \RR_J-O(K)\ \overset{F^{\qq_{m+1}}_*}\leadsto\  \RR_J-O(K) \qquad \text{into}\\ \RR_J-O(K)  \overset{F^{i}_*}\leadsto\  \RR_{J_q}-O(K) \ \overset{F^{\qq_{m+1}-i}_*}\leadsto\   \RR_J-O(K). \qquad 
\end{multline*} 
\end{enumerate}

\subsubsection{Conclusion~\ref{Cal:Lmm:Concl:a}: semi-equidistribution of the vertical degeneration}\label{sss:prf:CL:semi:Equid} In Item~\ref{proofCali:enumer:most:vertic} below, we will assume that the vertical part in every $\RR_s$ is $\bchi K-O(\bchi^{0.9} K)$.  Then, up to $O(\bchi^{0.9} K)$, the vertical part is almost invariant under any $f^i$; this leads to a required semi-equidistribution. 

\begin{enumerate}[label=\text{(\arabic*)},font=\normalfont,leftmargin=*,start=13]
  
 \item \label{proofCali:enumer:most:vertic} For all items in~\S\ref{sss:prf:CL:semi:Equid}, we assume that the vertical part of every $\RR_s$ is at least $\bchi K-O(\bchi^{0.9} K)$; i.e., for rectangles $\RR^\ver_s$ as in Item~\ref{proofCali:enumer:rect:RR_s}, we have:
 \[\Width(\RR^\ver_s)= \bchi K-O(\bchi^{0.9} K)\qquad\qquad \text{ for every $s$}.\]
The opposite assumption is in Item~\ref{proofCali:enumer:many:per}.

\item \label{proofCali:enumer:most:vertic:2} From every $\RR_s\equiv \RR_{I_s}$, we remove buffers of width $\sim \bchi^{0.9}K$ to dominate the $O(\bchi^{0.9}K)$ in Item~\ref{proofCali:enumer:most:vertic} so that the remaining laminations $\RR_{I_s^\new}$ are vertical. 

\noindent By removing additional $O(K)$ buffers and using Item~\ref{proofCali:enumer:AI:refine}, we can assume that the $\RR_{I^\new_s}$ are almost invariant up to $O(K)$: for every $i\le \qq_{m+1}$ and $s$, there is a $q$ such that
\begin{itemize}
    \item the symmetric difference between $f^{i}(I^\new_s)$ and $I_q^\new$ has length $\le \length_{m+1}$;
    \item the univalent pushforward of $\RR_{I_s^\new}-O(K)$ under $F^i_*$ is $\RR_{I_q^\new}-O(K)$.
\end{itemize}

\item \label{proofCali:enumer:most:vertic:3}Consider a level $m$ combinatorial interval $J$. It intersects one or two neighboring $I^\new_s, I_{s+1}^\new$. Because of almost invariance (Item~\ref{proofCali:enumer:most:vertic:2} ), we have
\begin{multline*}
   \qquad \Width^\ver(J)\ge\Width\Big(\RR_{J\cap I_s^\new}\Big)+\Width\Big(\RR_{J\cap I_{s+1}^\new}\Big)=\\ \Width\Big(\RR_{J\cap I_s^\new}\Big)+ \Width\Big(\RR_{ I_{s+1}^\new}\Big) - \Width\Big(\RR_{I_{s+1}^\new\setminus  J}\Big) - O(K) =\\
   \Width\Big(\RR_{J\cap I_s^\new}\Big) +\bchi K -O(\bchi^{0.9} K)- \Width\Big(\RR_{I_{s+1}^{\new}\setminus J}\Big)= \\ =\bchi K -O(\bchi^{0.9} K),\qquad
\end{multline*}
because $\Width\Big(\RR_{J\cap I_s^\new}\Big)=\Width\Big(\RR_{ I_{s+1}^{\new}\setminus J}\Big)+O(K)$ by Item~\ref{proofCali:enumer:AI:refine}.
This justifies a required semi-equidistribution.
\end{enumerate}

\subsubsection{The `$1 \leq 1\oplus 1 = 0.5$' argument} From now on, the remaining arguments are the same
as in the quadratic polynomial setting. In the items below, we will justify the existence of rectangle $\mathcal P$ submerging into the pseudo-bubble $\wZ_{s+1}$ as shown on Figure~\ref{Fig:Fig_B_Bstar} such that $\Width(\mathcal P)\succeq \bchi^{0.9} K$ and the vertical boundary $B=\partial ^{h,1}\mathcal P$ of $\mathcal P$ is combinatorially small.

\begin{enumerate}[label=\text{(\arabic*)},font=\normalfont,leftmargin=*,start=16]
 \item \label{proofCali:enumer:many:per}  We assume that Item~\ref{proofCali:enumer:most:vertic} does not hold. Then the width of at least one of the rectangles in $\big\{\RR^{\per,\pm}_s\big\}_s\equiv \big\{\RR^{\per,+}_s, \RR^{\per,-}_s\big\}_s$ is at least $\bchi^{0.9} K$.
  
    \item  Since wide rectangles do not intersect, the rectangles $\RR^{\per,\pm}_s$ block each other. Moreover, if a substantial part of $\RR^{\per,+}_s$ goes above $I_{s+1},I_{s+2},\dots I_{s+k}$, then the rectangles $\RR^\ver_{s+1},\RR^\ver_{s+2},\dots \RR^\ver_{s+k}$ are empty. Similar statement holds if the ``$+$'' is replaced with ``$-$''. Thus, we can find a rectangle $\RR^{\per,\pm}_s$ 
   that goes to $T_{s\pm 1 }\cup f^{\qq_{m+1}}\big(T_{s\pm 1 }\big)$ and has width $\bchi^{0.9} K$.
   
   \noindent For the definiteness, we assume that the $\RR^{\per,+}_s$ goes to $T_{s+ 1 }\cup f^{\qq_{m+1}}\big(T_{s+ 1 }\big)$ and has width $\bchi^{0.9} K$.

   \item\label{proofCali:enumer:11} Now by the same `$1 \leq 1\oplus 1 = 0.5$' argument as in \cite[Claim 5 in \S~9]{DL22}, there is a subrectangle $\mathcal P\subset \RR_s^{\per,+}$ with $\Width(\mathcal P)\succeq \bchi^{0.9} K$ such that  $B=  \partial^{1,h}\mathcal P$ has length $|B| \leq \frac{\length_{m+1}}{5}$, see \cite[Figure 26]{DL22} for illustration.
\end{enumerate}
\begin{proof}[Comments to Item~\ref{proofCali:enumer:11}]
  Write $\mathcal P\coloneqq \RR_s^{\per,+}$, and consider the lifts $\mathcal P_-,\mathcal P_+$ of $\mathcal P$ with respect to $\partial^{h,0} \mathcal P$ and $\partial^{h,1} \mathcal P$ respectively as illustrated on \cite[Figure 26]{DL22}. More precisely, since $\mathcal P-O(1)$ is within the coastal zone of $\overline Z$, the rectangle $\mathcal P$ can be lifted by $f^{\qq_{m+1}}$ and then mapped by $\iota^{\qq_{m+1}}$ back to $V$, see Lemma~\ref{lem:iota|CC_k}. This justifies $\mathcal P_-,\mathcal P_+ \subset V$.

  By Item~\ref{proofCali:enumer:AI:refine}, $\mathcal P- O(K)$ overflows its lift $\mathcal P_-$.  Item~\ref{proofCali:enumer:11} states that a substantial part of $\mathcal P$ lands at a subinterval of $\partial^{h,1}\mathcal P$ with length $\sim \length_{m+1}$. Assume this assertion is incorrect. Then for a small $\kappa>0$, we can remove combinatorially substantial $O(\kappa\ \Width(\PP))$-buffers from $\mathcal P$ (cf., Item~\ref{proofCali:enumer:O(K)-buff}) so that the remaining curves in $\PP$ also overflow $\mathcal P_+$. This contradicts the Gr\"otzsch inequality: most of the $\mathcal P$ can not overflow both of its preimages $\mathcal P_-,\mathcal P_+$.
    
\end{proof}

\subsubsection{Conclusion~\ref{Cal:Lmm:Concl:c}: amplified degeneration on deeper scales} \label{sss:prf:CL:concl:3}  Finally, we will detect a lamination $\mathcal L$ in $\mathcal P\setminus \wZ_{s+1}$ with an amplified degeneration and apply $F^{\qq_{m+1}}_*$ to $\mathcal L$ to finish the proof of the Calibration Lemma.

\begin{figure}
   \includegraphics[width=12cm]{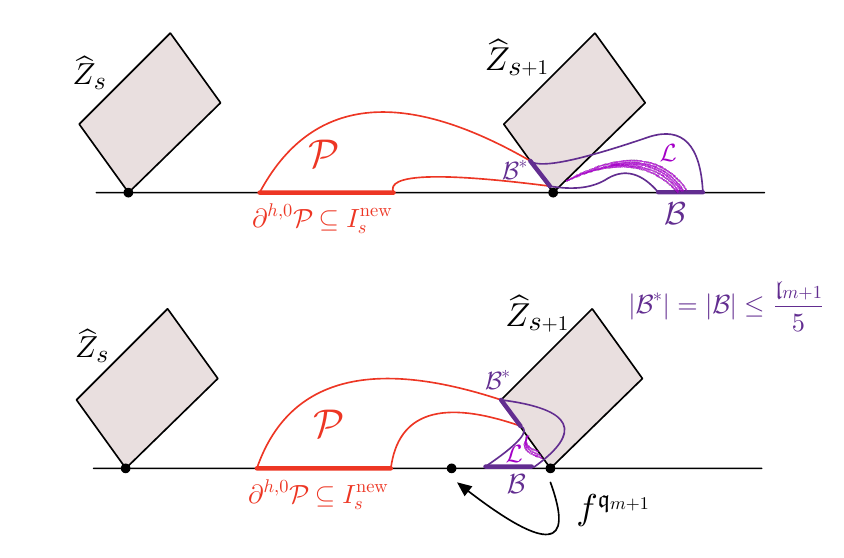}
   \caption{Illustration to the construction of a lamination $\mathcal L$. The rectangle $\mathcal P$ (red and purple) overflows its lift (the red part) before submerging into the pseudo-bubble $\wZ_{s+1}$. Localizing the submergence, we construct a lamination $\mathcal L\subset \mathcal P$ (light purple) outside of $\wZ_{s+1}$. There might be several configurations. Observe that if $\mathcal B$ intersects $T_s$ (bottom figure), then the intersection is combinatorially close to $\wZ_{s+1}$.}
   \label{Fig:Fig_B_Bstar}
\end{figure}

\begin{enumerate}[label=\text{(\arabic*)},font=\normalfont,leftmargin=*,start=19]

   \item\label{proofCali:enumer:12} Let $\filled_m\equiv K_{\qq_{m+1}}=f^{-\qq_{m+1}}(\overline Z)$, which we identify as a subset of $V$ via \ref{dfn:psib-ql:4:4} in the definition of $\psi^\bullet$-ql Siegel map (see \S \ref{ss:psi bullet:defn}). Let $Z_{s+1}$ be the bubble at $a \in T_s \cap T_{s+1}$, i.e., the component of $\filled_m - \overline{Z}$ that attached to $a$. Then there exists an interval $B^*$ on $\partial Z_{s+1}$ so that $f^{\qq_{m+1}}(B^*) = B$, and up to $O(K)$, any vertical leaves in $\mathcal{P}$ passes through $B^*$. 
   \item We now replace $\overline Z_{s+1}$ with the associated pseudo-bubble $\wZ_{s+1}$ and replace $B^*$ with its grounded proejction  $(B^*)^\grnd $ onto $\partial \wZ_{s+1}$. We remove $O(1)$ curves in $\mathcal P$ that do not enter $\wZ_{s+1}$ through $(B^*)^\grnd $.\\
   \noindent Let $Q \subseteq (B^*)^\grnd$ be the interval bounded by the first intersection of the two left- and right-most vertical leaves in $\mathcal{P}$. Let $\gamma$ be a vertical leaf of $\mathcal{P}$. We define $\gamma'$ as the shortest subarc connecting $Q$ to $\partial^{h,1}\mathcal{P}$, and set 
   $$
   \FF = \{\gamma': \gamma \in \FF(\mathcal{P})\}.
   $$

   \item Since $\mathcal P$ with $\Width(\mathcal P)\succeq \bchi^{0.9} K$ efficiently (up to $O(K)$) overflows its lift before $\wZ_{s+1}$, the family $ \FF$ (the part of $\mathcal P$ after $Q$) is extremely wide. We apply the argument in  \cite[Claim 6 in \S~9]{DL22} (i.e.,~ the Gr\"otzsch inequality) to conclude that \[\Width(\FF) \geq \bchi^{1.7}K.\]

   \item As in~\cite[proof after Claim 6 in \S~9]{DL22}, we replace $\mathcal F$ with a rectangle and apply \cite[Lemma 6.9]{DL22} to $\mathcal F$ to construct the lamination $\mathcal L \subset \mathcal P$ outside of $\wZ_{s+1}$ such that $\mathcal L$ emerges from a very small (of length $\ll \frac{\length_{m+1}}{\lambda}$) interval very close to $Q$ and such that 
   \begin{itemize}
   \item either every leaf of $\mathcal{L}$ lands back at $\partial \wZ_{s+1}$ with the $\lambda$-separation from the interval where $\mathcal L$ starts; or
   \item or every leaf of $\mathcal{L}$ lands at $\partial^{h,1}\mathcal P\subset \partial Z$. Since the combinatorial distance between $B$ and $B^*$ is $\ge \frac{2\length_{m+1}}{5}$, we also have the $\lambda$-separation in this case. (Also, $\partial^{h,1}\mathcal P$ can be replaced with its grounded projection on $\wZ^{m}$.)
   \end{itemize}
    \item Finally, we apply the univalent pushforward $F^{\qq_{m+1}}_*$ to $\mathcal L$. The result is a required degeneration for Conclusion~\ref{Cal:Lmm:Concl:c}.
\end{enumerate}
\end{proof}

\subsection{Combinatorial Localization Property} \label{ss:comblocal} 
The name ``Combinatorial Localization" is referred to Item~\ref{item:repli:3} of Lemma~\ref{lem:replication} below, where the original degeneration $K=\Width_3^+(I) \gg 1 $ is localized on a grounded interval $J$ with $|J|<\length_{j}<\frac{|I|
}{T}$.

Lemma~\ref{ss:comblocal} is meant to be used near the transitional level where there is a level $m$ combinatorial interval $I$ witnessing a big degeneration $K=\Width_3^+(I)\gg 1$ with $K\succeq \ K_F/\qq_{m+1}$. Such degeneration is spread around and then, if applicable, localized and calibrated. The result is either Equidistribution (Item~\ref{item:repli:1}), or Combinatorial Localization (Item~\ref{item:repli:3}), or a big external family (Item~\ref{item:repli:2}).

We remark that the equidistribution property in Item~\ref{item:repli:1} can only be ruled out with a ``global external input'', see~\cite{DL:HypComp}. We also remark that if $\wZ^m$ were constructed in the ``optimal way'' (to consume most of the $\Width^{+,\per}_{\lambda, \ext, m}$-degeneration), then Item~\ref{item:repli:2} does not occur, which justifies the name of the lemma.

\begin{lem}[Equidistribution or Combinatorial Localization]\label{lem:replication}
    For every $\lambda, T \geq 1$, there exists $K_{T, \lambda} \gg 1$ so that the following holds.

    Let $I \subseteq \partial Z$ be a level $m$ combinatorial interval with $\Width^{+}_5(I) = K \geq K_{T, \lambda}$. Assume that $\wZ^m$ is a pseudo-Siegel disk. Then 
    \begin{enumerate}[label=\text{(\roman*)},font=\normalfont,leftmargin=*]
        \item\label{item:repli:1} either there is an {\bf~equidistribution} of the degeneration: for every combinatorial  interval $J \subset \partial Z$ of level $m$, we have
        $$
        \Width^{+, \ver}(J) \asymp K;
        $$
        \item\label{item:repli:2} or there is a big {\bf~external degeneration}: there exists an interval $J \subseteq \partial Z$ grounded rel $\wZ^{j+1}$ with $\length_{j+1} \le |J| < \length_j\le \length_{m}$ and
        $$
        \Width^{+,\per}_{\lambda, \ext, j}(J) \succeq K,  
        $$
        \item\label{item:repli:3} or there is a {\bf~combinatorial localization}: there exists an interval $J \subseteq \partial Z$ grounded rel $\wZ^j$ with $|J| < \length_j \leq \frac{\length_m}{T}$ and 
        $$
        \Width^+_\lambda(J) \succeq K.
        $$
    \end{enumerate}
\end{lem}

\begin{proof} Let us first prove the following slightly weaker version of the lemma:
\vspace{0.2cm}

{\bf  Claim.} Assume that $\Width_3^+(I)\ge K_{T,\lambda}/2$ holds for a level $m$ combinatorial interval $I \subseteq \partial Z$. Then either
   \begin{enumerate}[label=\text{(\roman*')},font=\normalfont,leftmargin=*]
        \item\label{item:repli:1:weaker} for every  interval $J = F^i(I) \subset \partial Z$, with $0 \leq i \leq 2\qq_{m+1}$, we have
        $$
        \Width^{+, \ver}(J) \asymp K;
        $$    
\end{enumerate}
or Conclusion~\ref{item:repli:2} or Conclusion~\ref{item:repli:3} of the Lemma~\ref{lem:replication} hold.
    \begin{proof}[Proof of the claim]
     We project the interval $I$ into the pseudo-Sigel disks $\wZ^m_{-k*}$ and use the simple transformation rule to spread around $\wZ^m$; see~\S\ref{sss:SpreadAround:rel:wZ} and~\eqref{eq:SpreadAround:wZm}. We obtain that \[\proj^\grnd_{V,\wZ^m} \big(\Fam_{3, Z}(I_k)\big)\succeq K\qquad\qquad \text{ for all }J\equiv I_i=F^i(I).\]   Suppose Item~\ref{item:repli:1:weaker} does not hold. Then there exists $I_s = F^s(I)$ and a rectangle $\RR$ of width $\Width(\RR) \succeq K$ connecting the projections $I^{\grnd, m}_s$ and $\big((3 I_s)^c\big)^{\grnd, m}\cup V$ that submerges between $I^{\grnd, m}_s$ and $\big((3 I_s)^c\big)^{\grnd, m}$. By \cite[Snake Lemma 6.1]{DL22} and the fact $K \gg T, \lambda$, we obtain some grounded interval $I'\subset \partial Z$ with $|I'| \leq \frac{\length_m}{T}$ so that $\Width^+_\lambda(I') \succeq K$. Let $k$ be the integer so that $\length_{k+1} \leq |I'| < \length_k$. Suppose that Item~\ref{item:repli:2} does not hold, then we can apply the Calibration Lemma~\ref{lmm:CalibrLmm}. In the case of Item~\ref{Cal:Lmm:Concl:a} of the Calibration Lemma~\ref{lmm:CalibrLmm}, we have Item~\ref{item:repli:1}. Otherwise, we have Item~\ref{item:repli:3}.
    \end{proof}
   Let $J \subset \partial Z$ be an arbitrary combinatorial interval. Then there exists $J_1 = F|_{\partial Z}^{-i_1}(J), J_2 = |_{\partial Z}^{-i_2}(J)$ with $0\leq i_1< i_2 \leq 2 \qq_{m+1}$ so that $I \subset J_1 \cup J_2$. Note that $J_1, J_2$ are in general not disjoint. Since $\Width^{+}_5(I) = K \geq K_{T, \lambda}$, we conclude that $\max\{\Width^{+}_3(J_1), \Width^{+}_3(J_2)\} \geq K_{T, \lambda}/2$. Without loss of generality, we assume that $\Width^{+}_3(J_1) \geq K_{T, \lambda}/2$. The lemma follows by applying the claim to $J_1$. 
\end{proof}

\section{A priori-bounds for $\psi^\bullet$-ql Siegel maps}\label{s:Proof}  In this section, we prove Theorem~\ref{thm:main:psi*-ql maps} restated as Theorems~\ref{thm:main:psi-ql maps} and~\ref{thm:main:psi-ql maps:extra}. The central induction is in the proof of Theorem~\ref{thm:main:psi-ql maps} claiming the equidistribution property. Theorem~\ref{thm:main:psi-ql maps:extra} establishes explicit combinatorial thresholds for regularization $\wZ^{m+1}\leadsto \wZ^m$; the proof of Theorem~\ref{thm:main:psi-ql maps:extra} is in the \emph{a posteriori} regime relative to Theorem~\ref{thm:main:psi-ql maps}.

The outline of the section is presented in~\S\ref{ss:Proof:outline}. Let us briefly recall the main notations.

Let $F$ be an eventually-golden-mean $\psi^\bullet$-ql map as in~\S\ref{ss:psi bullet:defn}. 
Recall that the definition of pseudo-Siegel disks for quadratic polynomials is introduced in \cite[Definition 5.1]{DL22}. The adjustments for $\psi^\bullet$-ql Siegel maps is introduced in \S\ref{sss:iota peripheral}.

Let $K_F\coloneqq \Width^\bullet(F)$ be the width of $F$.
Let $\bK\gg 1$ be a sufficiently big threshold.
Recall that the {\em special transition level} $\bbm_F$ for $F$ with respect to the threshold $\bK$ is defined as follows.
\begin{itemize}
    \item If $K_F\le \bK$, we set $\bbm_F\coloneqq -2$;
    \item Otherwise, we set $\bbm_F$ to be the level satisfying 
\[ \frac{1}{\length_{\bbm_F}} < \frac{K_F}{\bK} \le \frac{1}{\length_{\bbm_F\ +1}} \sp\sp\sp  \begin{array}{c}\text{ or, }\\ \text{equivalently,}\end{array}\sp\sp\sp \sp   \begin{array}{c}
\length_{\bbm_F} K_F>\bK, \text{ and}\vspace{0.2cm}\\
\length_{\bbm_F\ +1} K_F\le \bK.
\end{array} \]
\end{itemize} 

In this section, we first prove the following a priori-bound for $\psi^\bullet$-ql map $F$, which is the first part of Theorem~\ref{thm:main:psi*-ql maps} and implies also the first part of Theorem~\ref{thm:main:ql maps} for quadratic-like maps as a special case.
\begin{thm}[Theorem~\ref{thm:main:psi*-ql maps}: Equidistribution]\label{thm:main:psi-ql maps}
There exists an absolute constant $\bK\gg 1$ so that the following holds.

Consider an eventually-golden-mean $\psi^\bullet$-ql map $F$ (see~\S\ref{ss:psi bullet:defn}) of width $K_F= \Width^\bullet(F)$ and the special transition level $\bbm_F$ with respect to $\bK$. Then there is a nested sequence of pseudo-Siegel disks $\wZ^m,\  m\ge -1$ such that for every grounded interval $J\subset \partial Z$ with $\length_{m+1}< |J|\le \length_{m}$ the following holds for the projection $J^m$ of $J$ to $\wZ^m$:
\begin{enumerate}[ label=(\Alph*)]
    \item \label{thm:apbs:A} if $m> \bbm_F$, then \[\Width_{\wZ^m}^{+,\ver} (J^m)=O(1) \sp\sp\text{ and }\sp\sp \Width_{3,\wZ^m}^{+,\per} (J^m)\asymp 1,\]
    \item\label{thm:apbs:B} if $m= \bbm_F$, then 
    \[\Width_{\wZ^m}^{+,\ver} (J^m)= O(\length_m K_F)  \sp\sp\text{ and }\sp\sp \Width_{3,\wZ^m}^{+,\per} (J^m)=O\big(\sqrt{\length_m K_F }\big),\]
        \item \label{thm:apbs:C} if $m< \bbm_F$, then 
    \[\Width_{\wZ^m}^{+,\ver} (J^m)\asymp |J|K_F  \sp\sp\text{ and }\sp\sp \Width_{3,\wZ^m}^{+,\per} (J^m)=O(1).\]
\end{enumerate}
\end{thm}

\subsubsection{Outline of the section}
\label{ss:Proof:outline} The the central induction of the proof of Theorem~\ref{thm:main:psi-ql maps} is encoded in Propositions~\ref{prop:deepandtransitional} and is illustrated on the diagram in Figure~\ref{fig:main pattern}. This induction is from deeper to shallower levels and is quite similar to the quadratic case; compare with the diagram on \cite[Figure 27]{DL22}. The key difference is the ``contradiction box $\bA_m \ge \dots \gg \bA_m$'' describing a possibility that an unexpected big vertical degeneration is developed. We also remark that Statements~\ref{eq:case1:main prf},~\ref{eq:case2:main prf},~\ref{eq:case3:main prf} of Proposition~\ref{prop:deepandtransitional} are similar to the respective statements in \cite[\S10.0.1]{DL22}.

 Proposition~\ref{prop:deepandtransitional} is stated with an explicit universal constant $\bK\asymp \bK_m$, see~\eqref{eq:dfn:bK_m}, \eqref{eq:bK_vs_bK_m}. (We find it convenient to slightly modify $\bK$ into $\bK_m$). Theorem~\ref{thm:main:psi-ql maps} suppresses $\bK$ within $O(1)\equiv O(\bK)$. The choice of constants in the proof of Proposition~\ref{prop:deepandtransitional} is summarized in~\S\ref{sss:choice of constants}. Constants $\bbt,\bchi,\blambda_\bbt$ are selected as in the quadratic case. An additional constant $\bL$ (tied to the ``contradiction box $\bA_m \ge \dots \gg \bA_m$'') is selected to accommodate Alternative~\ref{item:2:thm:SpreadingAround} of Amplification Theorem~\ref{thm:SpreadingAround} so that~\eqref{eqn:choiceL} and thus~\eqref{eq:bLvsbA} hold.

The proof of Theorem~\ref{thm:main:psi-ql maps} is completed in Propositions~\ref{prop:equidcombinbatorial} and~\ref{prop:equidcombinbatorial:shallow_levels}, where the equidistribution of the vertical families is justified. These propositions rely on Lemma~\ref{lem:replication}, ``Equidistribution or Combinatorial Localization''. Item~\ref{item:repli:3} of Lemma~\ref{lem:replication} leads to a contradiction by constructing a bigger degeneration. To accommodate such an argument, an additional constant $\bT\gg \bL$ is selected so that~\eqref{eq:dfn:bT} holds. (The constant $\bT$ does not directly influence the main induction of Proposition~\ref{prop:deepandtransitional}.) 

The constant $\bK$, the key geometric threshold, is selected last.

\subsection{Deep and transitional levels}\label{subsec:deeptransitional}
We will prove Theorem~\ref{thm:main:psi-ql maps} by induction from deep to shallow levels as follows. We will show that there are \[\blambda \gg 1,\quad \bL \gg_{\blambda} 1,\quad \bT \gg \bL,\quad\text{ and }\quad\bK\ \gg_{\blambda, \bT}\ \  1\] such that the following properties hold.

Consider an eventually-golden-mean $\psi^\bullet$-ql map $F$ of width $K_F= \Width^\bullet(F)$ and the transition level $\bbm_F$ with respect to $\bK$.
We define the {\em average vertical degeneration} at level $m$ by
$$
\bA_m:= \length_m K_F.
$$
We define the degeneration threshold for $m \geq \bbm_F$ by
\begin{equation}
\label{eq:dfn:bK_m}
    \bK_m:= \frac{\bK}{\bT} + \bL \cdot \bA_m
\end{equation}

It easily follows from definition that for $m> \bbm_F$, the constants $\bK_m$ and $\bK$ are comparable:
\begin{equation}
\label{eq:bK_vs_bK_m}
   \begin{matrix} \bK_m\ \asymp _{\bT} \  \bK/\bT ,\qquad \quad& \text{ i.e., }\quad \bK/\bT \le  \bK_m \le 3\bL \bK \ll \bT \bK,\\ \bK_m\ \to \ \bK/\bT \qquad\quad & \text{ monotonically as }\quad m\to \infty.\end{matrix}
\end{equation}

By design, $\bK\ \asymp_\bT\ \bK_m\ \gg_{\blambda, \bT}\ \  1$ allows arbitrary large (but still fixed)  dependence on all combinatorial constants; see~\S\ref{sss:choice of constants}. 

We now state a quantified version of Theorem~\ref{thm:main:psi-ql maps} at the deep and transitional levels.
\begin{prop}\label{prop:deepandtransitional}
For all $m \geq \bbm_F$,
\begin{enumerate}[label=\text{(\alph*)},font=\normalfont,leftmargin=*]
\item there exists a geodesic pseudo-Siegel disk $\wZ^m$ so that
\[\Width^{+, \per}_{\blambda,\ext, m} (I) =O(\sqrt{\bK_m})\]
for every interval $I$ grounded rel $\wZ^m$ with $\length_{m+1}\le |I|\le \length_m$. (In other words, $\wZ^m$ consumes all but $O(\sqrt{\bK_m})$-external $\lambda$-separated rel $m$ families.)
\label{eq:case1:main prf}
\item $\Width_\blambda^+(I)\le (2\bchi)\bK_m$ for every grounded rel $\wZ^m$ interval $m$ with $|I|\le \length_m$, where $\bchi$ is the constant from Calibration Lemma~\ref{lmm:CalibrLmm}.\label{eq:case2:main prf}
\item $\Width_\blambda^+(I)\le  \bK_m$ for every combinatorial interval $I$ of level $\ge m$.\label{eq:case3:main prf}
\end{enumerate}
\end{prop}

\subsubsection{The choice of constants in the proof of Proposition~\ref{prop:deepandtransitional}} \label{sss:choice of constants}
We first choose $\bbt \gg \bchi$, where $\bchi$ is the constant in the Calibration Lemma~\ref{lmm:CalibrLmm}. We choose $\blambda := \blambda_{\bbt}$ as in the Amplification Theorem~\ref{thm:SpreadingAround}.

Recall that $\length_m \asymp \frac{1}{\qq_{m+1}}$.  Let $\bL \gg_\blambda 1$ be chosen so that second alternative in Amplification Theorem~\ref{thm:SpreadingAround} $K_F\succeq_{\blambda_\bbt} K \qq_{m+1}$ gives
\begin{equation}\label{eqn:choiceL}
    \bA_m = \length_m K_F \gg \frac{K}{\bL},
\end{equation}
see~\eqref{eq:bLvsbA} for the application of~\eqref{eqn:choiceL}.

The constant $\bT$ (the input for Item~\ref{item:repli:3} of Lemma~\ref{lem:replication}) is chosen so that 
\begin{equation}
\label{eq:dfn:bT} \bT \gg \bL.    
\end{equation}
Finally, the constant $\bK$ is chosen so that 
\begin{equation}
\label{eq:dfn:bT}\bK\ \gg_{\bT}\ \  \bK_\bbt,
\end{equation}
 where $\bK_\bbt$ is the constant in the Amplification Theorem~\ref{thm:SpreadingAround}.

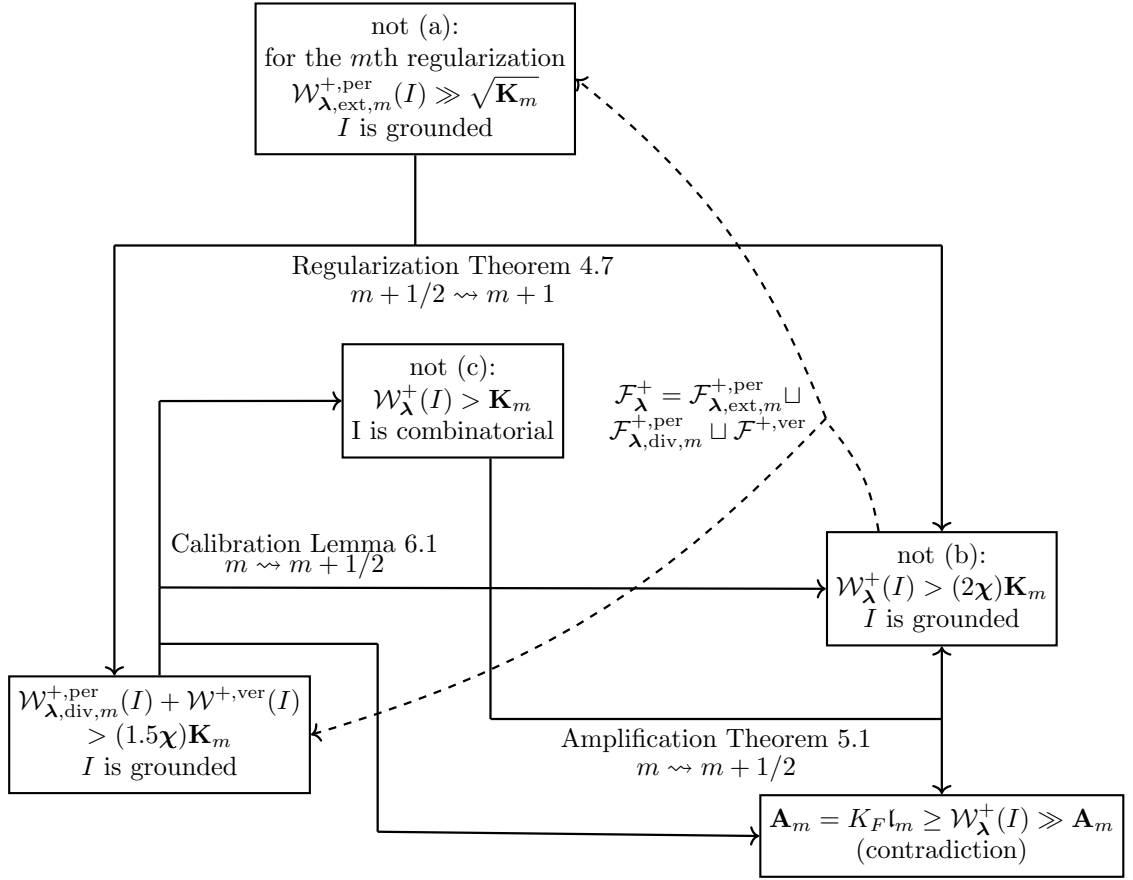
\begin{figure}[t!]

\[
\begin{tikzpicture}[line width=0.3mm]

\coordinate (Comb) at (5-1.5,2.5);
\draw
         node[below] at (Comb)[block] (comb) {$\begin{matrix}
                 \text{not~\ref{eq:case3:main prf}:}
               \\ \Width^+_\blambda (I)> \bK_m
         \\
        $I$\text{ is combinatorial}
        \end{matrix}$};

               \coordinate (grnd) at (10,0);

     \draw
         node[below](Gr) at (grnd)[block]  {$\begin{matrix}
         \text{not~\ref{eq:case2:main prf}:}
         \\
          \text{$\Width_\blambda^+(I)> (2\bchi)\bK_m$} 
         \\
          \text{$I$ is grounded} 
        \end{matrix}$};

        \coordinate (ver) at (10,-3.5);

     \draw
         node[below](Ver) at (ver)[block]  {$\begin{matrix}
         \bA_m = K_F \length_m \geq \Width^+_\blambda(I) \gg \bA_{m} \\  \text{(contradiction)}
        \end{matrix}$};

       \coordinate (regul) at (3,5);
        \draw
         node[above]  at (regul) [block] (Reg) {$\begin{matrix}
         \text{not~\ref{eq:case1:main prf}:}\\
        \text{for the $m$th regularization} 
         \\ \Width^{+,\per}_{\blambda, \ext,m} (I)\gg \sqrt{\bK_m}
         \\
          \text{$I$ is grounded}             
          \end{matrix}$};

       \coordinate (diving) at (-0.4,-3.5);
        \draw
         node[above]  at (diving) [block] (Div) {$\begin{matrix}
           \text{$\Width_{\blambda,\div,m}^{+, \per}(I) + \Width^{+, \ver}(I)$}\\
           \text{$>(1.5\bchi)\bK_m$} 
         \\
          \text{$I$ is grounded} 
          \end{matrix}$};

       \draw  (Reg.south) -- (3,3.8);
       %\draw  (3,3.8) edge[->] (divingNP.east);
       \draw (3,3.8) edge node[below, shift={(-3,0)}]{Regularization Theorem~\ref{thm:regul}} node[below,shift={(-3,-0.35)}]{$m+1/2 \leadsto m+1$}   (10,3.8);
       \draw  (10,3.8) edge[->] (grnd.north);
       \draw  (3,3.8) edge (-1, 3.8);
       \draw (-1, 3.8) edge[->] ([shift={(-0.6,0)}]Div.north);

    \draw  ([shift={(0.5,0)}]comb.south) -- (4,-2.5);
     
    \draw (4,-2.5) edge node[below]{\text{Amplification Theorem~\ref{thm:SpreadingAround}}}
    node[below,shift={(0,-0.35)}]{$m \leadsto m+1/2$} (10,-2.5); 
          
     \draw (10,-2.5) edge[->] (Gr.south);
     \draw (10,-2.5) edge[->] (Ver.north);

 \draw ([shift={(0.7,0)}]Gr.north west) edge[dashed, bend right=15] node[left, shift={(-0.6,0.8)}]{$\begin{matrix} \Fam_\blambda^+=\Fam^{+,\per}_{\blambda,\ext,m}\sqcup\\ \Fam^{+,\per}_{\blambda,\div,m}\sqcup \Fam^{+,\ver}\end{matrix}$} ([shift={(0.05,1.4)}]Gr.north west);     
 
 \draw ([shift={(0,1.5)}]Gr.north west) edge[dashed, ->, bend right = 15] (Reg.east);
% \draw ([shift={(0,1.5)}]Gr.north west) edge[dashed, ->] (comb.east);
 \draw ([shift={(0,1.5)}]Gr.north west) edge[dashed, ->, bend left = 15] (Div.east);

    \draw  (Div.north) -- (-0.4,-0.75);
\draw (-0.4,-0.75) edge[->]node[above,shift={(-2.5,0.35)}]{Calibration Lemma~\ref{lmm:CalibrLmm}}node[above,shift={(-2.5,0)}]{$m \leadsto m+1/2$}(Gr.west);    

    \draw  (-0.4,-1.5) -- (2.5, -1.5);
    \draw  (2.5,-1.5) -- (2.5, -4);
    \draw  (2.5, -4) edge[->]  (Ver.west);
    
% \draw  ([shift={(-4.93,0.)}]Gr.west) edge[->]  (Gr.west);
\draw  (-0.4,-0.75) -- (-0.4,1.73);
\draw  (-0.4,1.73) edge[->]  (comb.west);

\end{tikzpicture}
\]

\caption{Statements~\ref{eq:case1:main prf},~\ref{eq:case2:main prf}, and~\ref{eq:case3:main prf} are proved by contradiction: if one of the statements is violated on levels $m$ or $m+1/2$, then there will be even bigger violation on deeper scales. Here ``$m$'' indicates level $m$ combinatorial intervals, ``$m+1/2$'' indicates intervals with $\length_m\ge |I|\ge \length_{m+1}$, and ``$m+1$'' indicates intervals with $|I|\le \length_{m+1}$. The dashed arrows illustrate the decomposition $\Fam_\blambda^+=\Fam^{+,\per}_{\blambda,\ext,m}\sqcup \Fam^{+,\per}_{\blambda,\div,m}\sqcup \Fam^{+,\ver}$.} 
\label{fig:main pattern}
\end{figure}

\begin{proof}[Proof of Proposition~\ref{prop:deepandtransitional}]
The proof is similar in the quadratic polynomial setting, with only minor modifications. We summarize the induction steps as follows, see Figure~\ref{fig:main pattern} for an illustration.

First, consider $m \geq \bbm_F$ with $\length_m \leq |\theta_0|/(2\blambda + 4)$, where $\theta_0$ is the rotation number for $F$.
\begin{enumerate}[label=\text{(\roman*)},font=\normalfont,leftmargin=*]
\item Suppose Statement~\ref{eq:case3:main prf} is violated at level $m$, i.e., there exists a level $m$ combinatorial interval $I$ with $K :=\Width^+_{\blambda}(I) > \bK_m$. Then by Amplification Theorem~\ref{thm:SpreadingAround} and \eqref{eqn:choiceL},
\begin{itemize}
    \item either 
    \begin{equation}\label{eq:bLvsbA}
    \bA_m \gg \frac{K}{\bL} > \frac{\bK_m}{\bL} =  \frac{\bK/\bT + \bL \cdot \bA_m}{\bL} \geq \bA_m    
    \end{equation}
     which is a contradiction; or
    \item there exists an interval $J$ with $|J| < \length_m$ and 
    $$
    \Width^+_{\blambda}(J) \gg (2\bchi)\bK_m \geq (2\bchi)\bK_{j} \text{ for any } j \geq m+1.
    $$
    Thus, Statement~\ref{eq:case2:main prf} is violated for some interval $J$ at level $\ge m+1$.
\end{itemize}

\item Suppose Statement~\ref{eq:case1:main prf} is violated at level $m$, i.e., a pseudo-Siegel disk $\wZ$ can not be constructed to absorb all but $O(\sqrt{\bK_m})$-external rel $m$ families. Then by the exponential boost in Theorem~\ref{thm:regul}, 
\begin{itemize}
    \item either there exists some interval $J$ with $|J| \leq \length_{m+1}$ and 
    $$
    \Width^+_{\blambda}(J) \geq a^{\sqrt{\bK_m}} \gg (2\bchi)\bK_m \geq (2\bchi)\bK_{j}
    $$ 
    for any $j \geq m+1$, and for some fixed $a>1$. Thus Statement~\ref{eq:case2:main prf} is violated for some interval $J$ at level $\ge m+1$; or
    \item there exists some interval $J$ with $\length_{m+1} < |J| \leq \length_{m}$ and 
    $$
    \Width^{+,\per}_{\blambda, \div, m}(J) + \Width^{+, \ver}(J) \geq a^{\sqrt{\bK_m}} \gg (2\bchi)\bK_m
    $$
    for some $a>1$ is fixed.  Then by Calibration Lemma~\ref{lmm:CalibrLmm}, either Statement~\ref{eq:case3:main prf} or Statement~\ref{eq:case2:main prf}  is violated on levels $\ge m+1$; or $\bA_m = K_F \length_m \geq \Width^+_\blambda(J) \geq \bK_m \gg \bA_m$, which is not possible.
\end{itemize}

\item Suppose Statement~\ref{eq:case2:main prf} is violated at level $m$. By the decomposition 
$$
\Fam_\blambda^+(I)={\Fam^{+, \per}_{\blambda,\ext,m}(I)\sqcup \Fam^{+, \per}_{\blambda,\div,m}(I)\sqcup \Fam^{+, ver}(I)},
$$ 
we have the following two cases.
\begin{itemize}
    \item $\Width^{+, \per}_{\blambda,\ext,m}(I) \geq \bK_m \gg \sqrt{\bK_m}$. Thus Statement~\ref{eq:case1:main prf} is violated at level $m$; or
    \item $\Width^{+, \per}_{\blambda,\div,m}(I) + \Width^{+, ver}(I) > (1.5\bchi)\bK_m$. Then by Calibration Lemma~\ref{lmm:CalibrLmm}, either Statement~\ref{eq:case3:main prf} or Statement~\ref{eq:case2:main prf}  is violated on levels $\ge m+1$; or $\bA_m = K_F \length_m \geq \Width^+_\blambda(I) \geq \bK_m \gg \bA_m$, which is not possible.
\end{itemize} 
\end{enumerate}

Let $\bbs_{\theta_0}$ be the smallest integer so that $\length_\bbs \leq |\theta_0|/(2\blambda + 4)$. Then $\bbs_{\theta_0} \le 2 \log_2(2\blambda+4)$, i.e., there is a uniform bound on the number of levels $m$ with $\length_m \leq |\theta_0|/(2\blambda + 4)$. Thus, by increasing the constants, Proposition~\ref{prop:deepandtransitional} follows (see \cite[Lemma 10.4]{DL22} for more details).
\end{proof}

Finally, note that Proposition~\ref{prop:deepandtransitional} gives the bound for $\Width^{+, \ver}_{\wZ^m}$ and $\Width_{\blambda,\wZ^m}^{+,\per}$. To obtain the bound for $\Width_{3,\wZ^m}^{+,\per}$, we apply exactly the same argument as in \cite[\S 10.1.1]{DL22}, and obtain Theorem~\ref{thm:main:psi-ql maps} for $m \geq \bbm_F$.

\subsection{Equidistribution on trnasitional and shallow levels}\label{subsec:shallow}

The next proposition is an application of Proposition~\ref{prop:deepandtransitional} and Lemma~\ref{lem:replication}, ``Equidistribution or Combinatorial Localization''.

\begin{prop}[Equidistribution: the transitional level]\label{prop:equidcombinbatorial}
    Vertical degenerations on level $\bbm_F$ combinatorial intervals are equidistributed, i.e., for any interval $I \in \mathfrak{D}_{\bbm_F}$ in the $\bbm_F$-th diffeo-tiling, we have
    $$
    \Width^{+, \ver}(I) \asymp \bA_{\bbm_F},
    $$
    where $\bA_{\bbm_F} = K_F\length_{\bbm_F}$ is the average vertical degeneration at level $\bbm_F$.
\end{prop}
\begin{proof}
The proof is by contradiction. Suppose that vertical degenerations on level $\bbm_F$ combinatorial intervals are not equidistributed.

Since the average vertical degeneration is $\bA_{\bbm_F}$, there exists $I \in \mathfrak{D}_{\bbm_F}$ in the $\bbm_F$-th diffeo-tiling with 
$$
\Width^{+, \ver}(I) = K \succeq \bA_{\bbm_F}.
$$

By Proposition~\ref{prop:deepandtransitional}, for any $m \geq \bbm_F$, $\Width^{+, \per}_{\blambda,\ext, m} (I) =O(\sqrt{\bK_m}) \ll K$. 
Thus by Combinatorial Localization Lemma~\ref{lem:replication}, there exists a grounded interval $J \subseteq \partial \wZ^j$ with $|J| < \length_j \leq \frac{\length_{\bbm_F}}{\bT}$ and $\Width^+_\lambda(J) \succeq K$.
Since $\bT \gg \bchi, \bL$ and $K_F\length_{\bbm_F} \geq \bK \gg \bchi \bK/\bT$, we conclude that 
\begin{align*}
    \bA_{\bbm_F} &= K_F\length_{\bbm_F} \gg (2\bchi) (\bK/\bT + \frac{\bL K_F\length_{\bbm_F}}{\bT}) \\
    &\geq (2\bchi) (\bK/\bT + \bL K_F\length_{j}) = (2\bchi) \bK_j.
\end{align*}
Therefore, $\Width^+_\blambda(J) \gg (2\bchi) \bK_j$, which is a contradiction to Statement~\ref{eq:case2:main prf} in Proposition~\ref{prop:deepandtransitional} at level $j \geq \bbm_F+1$.
\end{proof}

Finally, applying Proposition~\ref{prop:equidcombinbatorial} and parallel law, we obtain: 

\begin{prop}[Equidistribution: shallows levels]
\label{prop:equidcombinbatorial:shallow_levels}
    Let $m < \bbm_F$, and $\wZ^m = \wZ^{\bbm_F}$. Let $\length_{m+1} < J\leq \length_m$ be a grounded interval.
    Then
    \[\Width_{\wZ^m}^{+,\ver} (J^m)\asymp |J|K_F  \sp\sp\text{ and }\sp\sp \Width_{3,\wZ^m}^{+,\per} (J^m)=O(1).\]
\end{prop}
\begin{proof}
    Since $m < \bbm_F$, there exists $N = N(J)$ so that $I \subset J\subset I'$, where $I$ is a union of $N$ level-$\bbm_F$ combinatorial intervals, and $I'$ is a union of $N+2$ level-$\bbm_F$ combinatorial intervals.
    By Proposition~\ref{prop:equidcombinbatorial}, the vertical degeneration on a level-$\bbm_f$ combinatorial interval is comparable to $\bA_{\bbm_F}$. Thus, by parallel law (see, for e.g. \cite[A.1.4]{DL22}),  $\Width_{\wZ^m}^{+,\ver} (J^m) \succeq N \bA_{\bbm_F} \asymp |J|K_F$.
    Since the same argument also holds for $\partial Z \setminus J$, we conclude that $\Width_{\wZ^m}^{+,\ver} (J^m) \asymp |J|K_F$.

    On the other hand, since each component of $3 J \setminus J$ contains a level-$\bbm_F$ combinatorial interval $I_\pm$, any arc in $\Fam^{+, \per}_{3,\wZ^m}(J^m)$ intersects an arc in $\Fam^{+, \ver}_{\wZ^m}(I_\pm^m)$. Since $\Width^{+, \ver}_{\wZ^m}(I_\pm^m) = \Width^{+, \ver}(I_{\pm}) + O(1) \asymp \bA_{\bbm_F} \gg 1$, we conclude that $\Width^{+, \per}_{3,\wZ^m}(J^m) = O(1)$.
\end{proof}
\subsubsection{Proof of Theorem~\ref{thm:main:psi-ql maps}}
    The theorem follows from Propositions~\ref{prop:deepandtransitional},~\ref{prop:equidcombinbatorial}, and~\ref{prop:equidcombinbatorial}. \qed

\subsection{Pseudo-Siegel disks with explicit combinatorial thresholds} The next theorem provides a construction of pseudo-Siegel disks with explicit combinatorial thresholds on all levels except transitional:

\begin{thm}[Theorem~\ref{thm:main:psi*-ql maps}: combinatorial thresholds]
\label{thm:main:psi-ql maps:extra}
Let $\bbm_F$ be the transitional level of $F$ as in Theorem~\ref{thm:main:psi-ql maps}. There exist a sufficiently big $\bM\gg 1$ and an increasing function $H\colon \R_{>0}\to \R_{\ge \bM}$ such that the following holds. We set
\begin{align}
    \label{eq:dfn:bM} 
    \begin{split}
    \bM_m &\coloneqq \begin{cases}
    \bM, \,\, \text{ if $m>\bbm_F$,}\\
     H(\length _m K_F) \,\, \text{ if $m = \bbm_F$,}\\
    \infty, \,\, \text{ if $m < \bbm_F$.}
\end{cases}      
    \end{split}
\end{align} 

Then geodesic pseudo-Siegel disks $\wZ^m$ for Theorem~\ref{thm:main:psi-ql maps} can be constructed explicitly as follows. For every (near-parabolic) $m$ with $\length_{m}> \bM_m \length_{m+1}$, add geodesic parabolic fjords $\widehat{\mathfrak F}_I$ at depth $\left\lfloor e^{\sqrt{ \ln \bM_{m}}}\right\rfloor$, see for details~\S\ref{sss:dfn:regular}.

Consequently, $\wZ^{-1}$ is $\mu(K_F)$-qc disk; i.e.~the dilatation of the qc disk $\wZ^{-1}$ is bounded in terms of $K_F$.
\end{thm}
 
As it is indicated in Theorem~\ref{thm:main:psi-ql maps:extra}, we say a level $m$ is \emph{near-parabolic} if and only if $\length_{m}> \bM_m \length_{m+1}$; otherwise $m$ is \emph{non-parabolic.} Since $F$ is eventually-golden-mean, all sufficiently deep levels $m\gg_F\  1$ are non-parabolic.
 In short, $\bM_m$ will be a combinatorial threshold for regularization: if $\frac{\length_{m}}{\length_{m+1}}\ge \bM_m$, then $\wZ^{m+1}$ is regularized into $\wZ^m$ at depth $e^{\sqrt{ \ln \bM_{m}}}$, see~\S\ref{sss:dfn:regular}; otherwise $\wZ^m:=\wZ^{m+1}$. 
 We remark that by our definition of $\bM_m$, if $m<\bbm_F$, then we always set $\wZ^m := \wZ^{m+1}$.

\hide{In short, $\bM_m$ will be a combinatorial threshold for regularization: if $\frac{\length_{m}}{\length_{m+1}}\ge \bM$ and $m> \bbm_F$, where $\bbm_F$ is the transition level defined in~\S\ref{subsec:aprioribounds}, then $\wZ^{m+1}$ is regularized into $\wZ^m$ at depth $\bM_{\sqrt{\ln}}$, see~\S\ref{sss:dfn:regular}; otherwise $\wZ^m:=\wZ^{m+1}$. If $m<\bbm_F$, then we always set $\wZ^m := \wZ^{m+1}$. The case $m=\bbm_F$ is discussed in~\S\ref{sss:dfn:regular:m_G}.}

\begin{proof}[Proof of Theorem~\ref{thm:main:psi-ql maps:extra}] The explicit threshold on the regularization follows from Lemma~\ref{lem:treshold:xy} below. It implies that the Log-Rule of Theorem~\ref{thm:parabolicfjords} kicks off at an explicit threshold. Therefore, the regularization  $\wZ^{m+1}\leadsto \wZ^{m}$ can also be achieved at the explicit threshold as required by Theorem~\ref{thm:regul}: if the regularization fails, then, for a universal $a>1$, either
\begin{itemize}
    \item there is a degeneration $\succeq a^{\sqrt{\bM_m}}$ on a deeper level, 
    \item or the vertical degeneration is $\succeq a^{\sqrt{\bM_m}} \qq_{m+1}$;
\end{itemize}
compare with Item ``not~\ref{eq:case1:main prf}'' on Figure~\ref{fig:main pattern}.

The quasiconformality claim is a non-dynamical statement that follows in the same way as in \cite[Section 11]{DL22}. More precisely, we introduce a nest of tilings $\TT(\partial \wZ^m)$ on $\partial \wZ^m$ by projecting the diffeo-tiling of $\partial \Z$ onto $\partial \wZ^{-1}$. On dams, the tiling is introduced geometrically in the same way as in~\cite[\S11.3.]{DL22}. Applying~\cite[Lemma~11.3]{DL22} and using the main estimates, we obtain the quasiconformality claim.  
\end{proof}

\subsubsection{Combinatorial threshold for Log-Rule} Since Theorem~\ref{thm:main:psi-ql maps} is established, we are now in the \emph{a posteriori} setting where the quantities $\bK_m,\  m>\bbm_F$, in Proposition~\ref{prop:deepandtransitional} are absolute; we write $\bK_m=O(1)$.

For the following lemma, recall that the notations of $x, y, v, w$ are introduced in \S~\ref{ss:Log-Rule}; see Figure~\ref{Fig:ParabolicF}.

\begin{lem}\label{lem:treshold:xy}
  For levels $m>\bbm_F$, the interval $[x,y]$ satisfies
  \begin{equation}
    \label{eq:1:lem:treshold:xy}
  \dist(v,x)\ \asymp \ \dist(y,w) \ \asymp \ \length_{m+1}.
  \end{equation}
If $\dist(v,w)\gg 1$, then $[x,y]$ is non-empty and satisfies~\eqref{eq:1:lem:treshold:xy}.

  For $m=\bbm_F$, the interval  $[x,y]$ satisfies 
   \begin{equation}
    \label{eq:2:lem:treshold:xy}\dist(v,x)\ \asymp_{(\length_{m+1}K_F)} \ \dist(y,w) \  \asymp_{(\length_{m+1}K_F)} \ \length_{m+1}.
    \end{equation}
\end{lem}
\begin{proof} The argument is the same as in  \cite[\S 10.2]{DL22}. We assume that $[v,x]\asymp [y,w]$ are long, we split these intervals into the $\bigcup_k X^k$ and $\bigcup_k Y^k$ and estimate the dual family to deduce~\eqref{eq:X^k:Y^k} and~\eqref{eq:X^k:Y^k:m_F}. For $m=\bbm_F$, the vertical family affects ``$\succeq$''.

Assume that $x$ is sufficiently far from $v$. As in~\cite[\S 10.2]{DL22}, we present $[v,x]$ a concatenation of $X^1 \cup X^2\cup \dots \cup X^K$, where $|X^1|= 10 \length_{m+1}$ and $|X^{i+1}|=2|X^{i}|$. Similarly, present $[y,w]$ as a concatenation of $Y^k\cup Y^{k-1}\cup \dots \cup Y^1$ satisfying the same properties: $|Y^1|= 10 \length_{m+1}$ and $|Y^{i+1}|=2|Y^{i}|$. 

As in \cite[Lemma 10.5]{DL22}, we {\bf claim} that 
\begin{equation}
    \label{eq:X^k:Y^k}\Width^{+,\per}_{\ext, m}(X^k, Y^k)\ \succeq 1\qquad \qquad\text{ if }m>\bbm_F,
\end{equation}
and 
\begin{equation}
    \label{eq:X^k:Y^k:m_F}\Width^{+,\per}_{\ext, m}(X^k, Y^k)\ \succeq_{(\length_{m+1}K_F)}\ \  1\qquad \qquad\text{ if }m=\bbm_F,
\end{equation}
\begin{proof}[Proof of the Claim] Following notations of \cite[Lemma 10.5]{DL22}, we denote by $T^{k}$ the interval between $X^k$ and $Y^k$ and we have $T^{k}=X^{k+1}\cup T^{k+1}\cup Y^{k+1}$. To deduce~\eqref{eq:X^k:Y^k} and~\eqref{eq:X^k:Y^k:m_F}, we need to estimate from above the width of the dual family:
\[\Width^{+,\ver}(T^k)+\Width^{+,\per}_{\filled_{m}}(T^k,\partial \filled_{m}\setminus T^{k-1}).\]
For $m>\bbm_F$, we have $\Width^{+,\ver}(T^k)=O(1)$. For $m=\bbm_F$, we have $\Width^{+,\ver}(T^k)\asymp K_F\length_m+ O(1)$.

The argument to estimate $\Width^{+,\per}_{\filled_{m}}(T^k,\partial \filled_{m}\setminus T^{k-1})$ is the same as in~\cite[Lemma 10.5]{DL22}: apply a univalent pushforward $f^{\qq_{m+1}}$ to replace $\filled_{m}$ with $\overline Z$, and then use the Shift Argument to bound the $\Width^{+,\per}_{\ext,m}$-component and the Calibration Lemma to bound the $\Width^{+,\per}_{\div,m}$-component.   
\end{proof}

Applying the parallel law, we deduce that $\Width^{+,\per}_{\ext,m} ([v,x],[y,w])$ is sufficiently wide if $x,y$ are sufficiently far from $v,w$ as required. 
\end{proof}

\subsubsection{Construction of parabolic fjord $\widehat{\mathfrak F}_I$ and $S^\inn _I$}\label{sss:dfn:regular} Consider a near-parabolic level $m$. Let $I=[a,b]\in \mathfrak{D}_m$ be an interval in the $m$th diffeo-tiling of $\partial Z$. Choose $a', b'\in I$ with $a<a'<b'<b$ such that \[\dist(a,a')=\dist(b',b) =\left\lfloor e^{\sqrt{ \ln \bM_{m}}}\right\rfloor \  \length_{m+1} \]
and set $\beta_I$ to be the hyperbolic geodesic of $V\setminus \overline Z$ connecting $a',b'$. This defines the parabolic fjord $\widehat{\mathfrak F}_I$.
We remark that since the level $m$ is near-parabolic, we have
$$
|I| = \dist(a,b) \asymp \length_m \geq \bM_m \length_{m+1} \gg \left\lfloor e^{\sqrt{ \ln \bM_{m}}}\right\rfloor \  \length_{m+1}.
$$

To construct $S^\inn _I$, we choose a sufficiently big $\bbv\gg 1$ with the understanding that $e^{\sqrt{ \ln \bM_{m}}} \gg \bbv$.
Choose $a''\in [a,a']$ and $b''\in [b',b]$ such that
\begin{equation}
\label{eq:dfn:S^inn}
    \dist(a,a'')=\dist(b'',b) =\left\lfloor \frac{e^{\sqrt{ \ln \bM_{m}}}}{ \bbv}\right\rfloor \length_{m+1} \gg \length_{m+1},
\end{equation} 
this defines $S^\inn _I$.

\subsubsection{Construction of $A_I$ and $\upbullet \XX_I$}
Let $\mathfrak G_I$ be the rectangle from $[a,a'']$ to $[b'', b]$ bounded by hyperbolic geodesics in $V\setminus \overline Z$; i.e., 
\[ \partial^{h,0} \mathfrak G_I =[a,a''],\sp\sp \partial^{h,1} \mathfrak G_I =[b'',b]\]
and $\partial^{v}\mathfrak G_I $ is the pair of hyperbolic geodesics. The condition that $e^{\sqrt{ \ln \bM_{m}}} \gg \bbv$ implies that (see \cite[\S 4]{DL22})\[ \Width (\mathfrak G_I) \ \asymp\  \ln \left(e^{\sqrt{ \ln \bM_{m}}}/\bbv \right) \ \gg \ 1.\]
We can now select a required rectangle $\XX_I$ conformally deep in $\mathfrak G_I$ (i.e., conformally close to $\beta_I$) so that its width is of the size $\Delta \ll  \sqrt{ \ln \bM_{m}}$. We can also select $A_I$ separating $\XX_I$ from $S^\inn_I$. In particular, we can assume that the interval \[[x_a, x_b]\coloneqq \partial \upbullet \XX_I \cap I\sp\sp\sp\text{ with }\sp x_a\in[a,a''], \sp x_b\in [b'', b]\]
is defined similar to \eqref{eq:dfn:S^inn}:
\[     \dist(a, x_a)=\dist(x_b,b) =\left\lfloor \frac{e^{\sqrt{ \ln \bM_{m}}}}{ \bbw}\right\rfloor \length_{m+1}, \]
where $\bbw \gg \bbv\gg 1$ with the understanding that we still have $e^{\sqrt{ \ln \bM_{m}}}\gg \bbw$.

\subsubsection{Bistability of ${\wZ^m}$} \label{sss:big:bistability}
We recall the definition of bistability in \S~\ref{subsec:pullback} and \S~\ref{sss:F:bistab}. Since $e^{\sqrt{ \ln \bM_{m}}}\gg \bbw$, the construction guarantees that 
$$
\dist(\partial^h \XX_I, \partial I) = \left\lfloor \frac{e^{\sqrt{ \ln \bM_{m}}}}{ \bbw}\right\rfloor \length_{m+1} \gg \length_{m+1}.
$$
Thus, we see that the Fjords are obtained by submerging into depth $\geq \bM$; see \eqref{eq:depth:K}.
Thus, by the discussion in \S~\ref{ss:bistablepartunderprotection} and Lemma~\ref{lem:geodesicpseudoSiegel}, we see that ${\wZ^m}$ can be assumed to be $T$-stable for arbitrarily large $T$; see \S~\ref{subsec:pullback}.

\subsubsection{Proof of Theorem~\ref{thm:combinatorial main}}
We now apply Theorems~\ref{thm:main:psi-ql maps} and~\ref{thm:main:psi-ql maps:extra} to prove Theorem~\ref{thm:combinatorial main}.

    Tor every $\theta=[0;a_1,a_2,a_3, \dots]$, $|a_i|\le M_\theta$ of bounded type, define the approximating sequence
\[\theta_{n}\coloneqq [0;a_1,a_2,\dots, a_n,1,1,1,\dots]\to \theta.\]
    Note that $\length_{m,n}$ for $\theta_n$ equals $\length_m$ for all sufficiently large $n$.
    We can construct a sequence $f_n \to f$ of Siegel quadratic like maps with rotation number $\theta_n \to \theta$.
    Since $Z_{\theta_n}\to Z_{\theta}$ as qc disks (with dilatation depending on $M_\theta$), the estimate $\Width^+_3(I) =O(\length_m K_F+1)$ in Theorem~\ref{thm:main:psi-ql maps} also holds for $Z_\theta$, where the constant is independent of $M_\theta$. The moreover part corresponds to the shallow levels and follows by a similar argument.\qed

\subsection{Strategy to apply the main Theorem} \label{ss:Application} A strategy to apply Theorem~\ref{thm:main:psi-ql maps} to rational maps has been designed in~\cite{DL:HypComp}. Assume that we have a class of rational maps $g$ where we wish to establish uniform \emph{a priori} bounds. If the dynamical plane of $g$ has a periodic eventually-golden-mean Siegel disk $\overline Z$, then often the analysis requires understanding how wide laminations $\RR$ cross its preperiodic preimages $\overline Z'$. Here we can assume that $\Width(\RR)\asymp \Width_\text{total}(g)$, where $\Width_\text{total}(g)$ is the total degeneration of $g$. The goal is to produce a degeneration bigger than $\Width_\text{total}(g)$.

Assuming there is a partial self-covering structure~\eqref{eq:part:self-cover}, we can apply the inflation around $\overline Z$, see Proposition~\ref{prop:g:unwind} and Figure~\ref{Fig:UnwindingSiegelFull}, and then regularize $\overline Z\leadsto \wZ^{-1}$. Lifting the regularization, we obtain a preperiodic pseudo-Siegel disk $\overline Z'\leadsto \wZ'$.  

Now we wish to combinatorially localize the $\RR$ rel $\wZ'$ on a sufficiently deep level (cf., Lemma~\ref{lem:replication}), then move the localized degeneration onto $\wZ^{-1}$ and spread around $\wZ^{-1}$ to obtain a value bigger than $\Width_\text{total}(g)$. The current paper prepares the analysis that can be carried out locally at $\wZ^{-1}$ in the $\psi^\bullet$-setting. However, most significantly, properties of a global map $g$ (a ``global input'') are required to handle Alternative~\ref{Cal:Lmm:Concl:a} of Calibration Lemma~\ref{lmm:CalibrLmm}. By Remark~\ref{rem:CalLmm:AI}, one needs to establish a variant of Calibration Lemma in the dynamical plane of $g$ to handle almost invariant families associated with this Alternative~\ref{Cal:Lmm:Concl:a}.


\begin{thebibliography}{*****}

\bibitem[A]{A} L. Ahlfors. Conformal Invariants: Topics in Geometric Function Theory, McGraw Hill Book Co., New York, 1973.






\bibitem[ACh]{ACh} A. Avila and D. Cheraghi.
  Statistical properties of quadratic polynomials with a neutral fixed point.
  JEMS, v. 20 (2018), 2005--2062.

 


\bibitem[AL2]{AL-posmeas} A. Avila and M. Lyubich. 
Lebesgue measure of Feigenbaum Julia sets. arXiv:1504.02986. 



\bibitem[BBCO]{BBCO} A. Blokh, X. Buff, A.Chéritat, L. Oversteegen. The solar Julia sets of basic quadratic Cremer polynomials. Ergodic Theory and Dynamical Systems, 30(1), 51-65, 2010.


\bibitem[BC]{BC} X. Buff and A. Cheritat. Quadratic Julia sets of
  positive area. Annals Math., v. 176 (2012), 673--746.

\bibitem[BD]{BD} B. Branner and A. Douady. Surgery on complex polynomials. Holomorphic dynamics, {M}exico, 1986.

 \bibitem [Ch1] {Ch1} D. Cheraghi. Typical orbits of quadratic polynomials with a neutral fixed point: non-Brjuno type. Ann. Sci. \'Ec. Norm. Sup., v. 52 (2019), 59--138.

 \bibitem [Ch2] {Ch2} D. Cheraghi.  Topology of irrationally indifferent attractors. arXiv:1706.02678.
   
 \bibitem[ChC] {ChC}  D. Cheraghi and A. Cheritat. A proof of the Marmi-Moussa-Yoccoz conjecture for rotation numbers of high type. Invent. Math. 202, no. 2, pp. 677–742, 2015.

 \bibitem[Che]{Che:nearpar} A. Cheritat. Near parabolic renormalization for unicritical holomorphic maps. arXiv:1404.4735.


   
 \bibitem[Chi]{Ch} D. K.~Childers. Are there critical points on the boundaries of mother hedgehogs? Holomorphic dynamics and renormalization, Fields Inst. Commun., 53 (2008), 75–87.  


\bibitem[CS]{CS} D. Cheraghi and M. Shishikura. Satellite renormalization of quadratic polynomials. arXiv:1509.0784


\bibitem[D1]{D-Siegel} A. Douady. Disques de Siegel et anneaux de Herman. In S\'eminaire Bourbaki,
  Ast\'erisque, v. 152--153 (1987), 151--172. 



\bibitem[DH1]{DH:Orsay} A. Douady and J. H. Hubbard.
  \'{E}tude dynamique des polyn\^{o}mes complexes.
 Publication Mathematiques d'Orsay,
84-02 and 85-04.

\bibitem[DH2]{DH}
 A. Douady and J. H. Hubbard. On the dynamics of polynomial-like maps.
   Ann. Sc. \'{E}c. Norm. Sup., v. 18 (1985), 287 -- 343.




\bibitem[DL1]{DL} D. Dudko and M. Lyubich. Local connectivity of the Mandelbrot set at some satellite parameters of bounded type. arXiv:1808.10425

\bibitem[DL2]{DL:Feigen} D. Dudko and M. Lyubich. MLC at Feigenbaum points. arXiv:2309.02107.

\bibitem[DL3]{DL22} D. Dudko and M. Lyubich. Uniform a priori bounds for neutral renormalization. arXiv:2210.09280


\bibitem[DL4]{DL:sector bounds} D. Dudko and M. Lyubich. Uniform \emph{a priori} bounds for neutral renormalization. Variation~\RN{1}:  Sector Renormalization. Manuscript $2024$.




\bibitem[DLu]{DL:HypComp} D. Dudko and Y. Luo. Sierpinski Carpet Hyperbolic Components of Disjoint Type are Bounded. arXiv:2509.25658.


\bibitem[DLS]{DLS} D. Dudko, M. Lyubich, and N. Selinger. Pacman renormalization and self-similarity of the Mandelbrot set near Siegel parameters. Journal of the AMS, 33 (2020), 653-733.



\bibitem[dF]{dF} E. de Faria.  Asymptotic rigidity of scaling ratios
  for critical circle maps. Erg. Th. and Dyn. Syst., v. 19 (1999),
  995--1035. 
  

\bibitem[F]{F} P. Fatou. Sur les équations fonctionnelles. Bull. Soc. Math. Fr., 48:208–314, 1920.

  
\bibitem[GY]{GY} D. Gaidashev and M. Yampolsky. Renormalization of almost commuting pairs. arXiv:1604.00719.


\bibitem[Her]{H} M. Herman. Conjugaison quasi symm\'etrique des diff\'eomorphisms du cercle \`a des rotations
et applications aux disques singuliers de Siegel. Manuscript (1986).


\bibitem[Hub]{Hub:Yoccoz} J.H. Hubbard.
   Local connectivity of Julia sets and bifurcation
     loci: three theorems of J.-C. Yoccoz. In:
    ``Topological Methods in Modern Mathematics, A Symposium in
     Honor of John Milnor's 60th Birthday", Publish or Perish, 1993.




\bibitem[IS]{IS} H. Inou and M. Shishikura.
  The renormalization for parabolic fixed points and their
  perturbations. Manuscript 2008. 

\bibitem[K]{K}  J. Kahn. A priori bounds for some infinitely
  renormalizable quadratics: I. Bounded primitive
  combinatorics. Preprint IMS at Stony Brook, \# 5 (2006). 


\bibitem[KL1]{KL} J. Kahn and M. Lyubich. The Quasi-Additivity Law in conformal geometry. Annals Math., v. 169 (2009), 561--593.

 \bibitem[KL2]{KL2}  J. Kahn and M. Lyubich. A priori bounds for some infinitely renormalizable quadratics: \RN{2}. Decorations. Annals Sci. Ecole Norm. Sup., v. 41 (2008), 57--84.


\bibitem[KL3]{molecules} J. Kahn \& M. Lyubich. {\it A priori bounds} for some infinitely
renormalizable quadratics: III.  Molecules.  In: ``Complex Dynamics: Families and Friends''.
    Proceeding of the conference dedicated to Hubbard's 60th birthday (ed.:  D. Schleicher). 
    Peters, A K, Limited, 2009. 


\bibitem[L2]{L-book} M. Lyubich. Conformal Geometry and Dynamics of
  Quadratic Polynomials. Book in preparation, www.math.sunysb.edu/~mlyubich/book.pdf. 


\bibitem[Lim]{Lim} Willie Rush Lim. A priori bounds and degeneration of Herman rings with bounded type rotation number. arXiv:2302.07794.

\noindent Thesis 2024, Stony Brook University, in preparation.



\bibitem[McM]{McM3} C. McMullen. Self-similarity of Siegel disks and Hausdorff dimension of Julia sets.
  Acta Math., v. 180 (1998), 247--292.  

\bibitem[M1]{Mil} J. Milnor.
     Local connectivity of Julia sets: expository lectures.
     In: ``The Mandelbrot Set, Themes and Variations'',  67-116, ed. Tan Lei.  
     Cambridge University Press, 2000.




\bibitem[Pe]{Pe} C. Petersen. Local connectivity of some Julia sets containing a circle with an irrational rotation, Acta. Math.,177, 1996, 163-224.


\bibitem[PM]{PM}  R. P\'erez-Marco. Fixed points and circle maps, Acta Math., 179 (1997), 243-294.


\bibitem [ShY]{ShY} M. Shishikura and Fey Yang. The high type quadratic Siegel disks are Jordan domains. arXiv:1608.04106.

 


\bibitem[Sw]{Sw} G. Swiatek. On critical circle homeomorphisms. Bol. Soc. Bras. Mat., v. 29 (1998), 329--351.

 




\bibitem[Ya1]{Ya-posmeas} M. Yampolsky.  Siegel Disks and Renormalization Fixed Points.
  Fields Inst. Comm., v. 53 (2008), 377--393.  


\bibitem[Ya2]{Ya} M. Yampolsky. Complex bounds for
  renormalization of critical circle maps. Erg. Th. and Dyn. Syst.,
  v. 19 (1999), 227--257. 
   

\bibitem[Yo]{Y} J.-C. Yoccoz. Petits diviseurs en dimension 1. Ast\'erisque, v. 231 (1995).


\bibitem[Zha11]{Zha11} G. Zhang. All bounded type Siegel disks of rational maps are quasi-disks. Invent. Math., 185(2):421–466, 2011.


\end{thebibliography}
\end{document}